\documentclass[12pt,letterpaper]{amsart}
\pdfoutput=1

\usepackage[T1]{fontenc}
\usepackage[utf8]{inputenc}
\usepackage[english]{babel}
\usepackage{lmodern}
\usepackage{placeins}
\usepackage{float}
\usepackage{geometry}
\usepackage{hyperref}
\usepackage{booktabs}
\usepackage{enumitem}
\usepackage{url}

\usepackage{xcolor}
\usepackage{colortbl}
\definecolor{darkgreen}{rgb}{0.33, 0.55, 0.13}
\definecolor{darkpink}{rgb}{0.91, 0.33, 0.5}
\definecolor{electriccyan}{rgb}{0.0, 1.0, 1.0}
\definecolor{electricultramarine}{rgb}{0.25, 0.0, 1.0}
\definecolor{greenyellow}{rgb}{0.30, 0.82, 0.0}
\definecolor{lightbrown}{rgb}{0.71, 0.4, 0.11}
\definecolor{prune}{rgb}{0.45, 0.11, 0.11}
\definecolor{mygold}{RGB}{212, 175, 55}
\definecolor{levired}{RGB}{220, 50, 50}   
\definecolor{leviblue}{RGB}{50, 100, 220} 
\definecolor{levigold}{RGB}{218, 165, 32} 

\usepackage{amsmath}
\usepackage{amssymb}
\usepackage{amsthm}
\usepackage{caption}
\usepackage{subcaption}
\usepackage{tikz}
\usetikzlibrary{3d,arrows.meta,calc,decorations.markings,positioning,intersections}

\hypersetup{
    colorlinks=true,
    linkcolor=blue,
    urlcolor=blue,
    pdftitle={12-TET Harmony Revisited}
}
\newtheorem{theorem}{Theorem}
\newtheorem{proposition}{Proposition}

\newtheorem{definition}{Definition}

\begin{document}

\title[The Harmonic Multiverse of 12TET]{In Search of the Canonical Harmony for 12-TET}

\author{Paweł Nurowski}
\address{Centrum Fizyki Teoretycznej,
Polska Akademia Nauk, Al.
Lotnik\'ow 32/46, 02-668 Warszawa, Poland, and
Guangdong Technion -- Israel Institute of Technology, No. 241, Daxue Road,
Jinping District, Shantou, Guangdong Province, China}
\email{nurowski@cft.edu.pl}

\date{\today}

\begin{abstract}
Is the specific structure of Western tonal harmony a physical inevitability derived from acoustics, or is it merely one solution among many in a purely algebraic landscape? In this paper, we strip away the physics of vibrating strings and treat harmony as the solution to a simple linear system within the cyclic group $\mathbb{Z}_{12}$. By defining a harmonic system as a partitioning of a generator interval (the "Fifth") into two complementary thirds, we derive a complete classification of all possible harmonic universes in 12-Tone Equal Temperament. We show that every such system corresponds to a specific topological structure, visualized via its Levi graph.

Our analysis reveals a counter-intuitive fact: the topological structure of standard Western harmony is not unique. It is one of exactly twelve mathematically isomorphic systems. However, we demonstrate that these shadows are not created equal. We identify a "privileged quartet" of systems that preserve the global connectivity of the Circle of Fifths. Furthermore, by invoking the Chinese Remainder Theorem, we present a compelling argument that the most number-theoretically natural system in $\mathbb{Z}_{12}$ is not the Western one, but a specific "alien" system defined by the partition $(9,4)$. This classification provides a rigorous mathematical foundation for neo-Riemannian theory and offers composers a precise "translation map" to explore syntactically familiar, yet sonically distinct, harmonic worlds.
\end{abstract}
\maketitle

\section{Introduction}

For centuries, the study of musical harmony has been a dialogue between two distinct voices. The first voice is that of the physicist, grounded in acoustics [10, 14], arguing that the Major and Minor triads are inevitable consequences of the harmonic series. To this camp, a chord is "good" because nature says so; the frequency ratios $4:5:6$ dictate our aesthetic reality.

The second voice is that of the mathematician [8, 11, 13, 15, 16], who views harmony not as a physical phenomenon, but as a combinatorial game played on a set of pitch classes. In this view, notes are simply elements of a group, and chords are geometric shapes moving through a discrete space.

In this paper, we adopt the latter view. However, our starting point is not an abstract axiom, but a direct observation of the standard Western major triad. We observe that a major chord is built by splitting a Perfect Fifth ($q=7$) into a Major Third ($t=4$) and a Minor Third ($s=3$). These intervals satisfy two elementary arithmetic conditions:
\begin{enumerate}
    \item Their sum forms the generator: $4+3 = 7$.
    \item Their difference is the structural parameter: $4-3 = 1$.
\end{enumerate}

We generalize this observation by treating the specific intervals $4, 3,$ and $7$ as variables. We pose a simple question: What other harmonic systems exist if we maintain this algebraic form but allow the values to change? This leads us to study the following system of linear equations in $\mathbb{Z}_{12}$:
\begin{equation}
\label{eq:fundamental}
t+s \equiv q \pmod{12} \quad \text{and} \quad t-s \equiv \Delta \pmod{12}
\end{equation}
where $q$ is any generator of the group and $\Delta$ is a parameter defining the harmonic "color."

Since the standard Western harmony (where $q=7, \Delta=1$) is a valid solution, we know this system is consistent. Our goal is to perform a complete scan of the solution space to identify every other possible harmonic universe that satisfies these constraints.

However, as we traverse this algebraic landscape, a deeper hierarchy emerges. While topology suggests a democracy of "shadow systems" that mirror our own, number theory points toward a single, underlying monarch. Our journey proceeds as follows:
\begin{itemize}
    \item In \textbf{Section 2}, we establish the general algebraic theory, defining chords not by their acoustic qualities, but by their generation rules in $\mathbb{Z}_{n}$.
    \item In \textbf{Sections 3--9}, we explore the solutions for the standard generator $q=7$ and beyond. We encounter exotic systems, such as the "Wide Thirds" universe where chords are Janus-faced (modal degeneracy), and the "Tritone" universe connected by harmonic wormholes.
    \item In \textbf{Sections 10--11}, we bridge music and topology, identifying our harmonic systems with specific symmetric graph structures. We observe that three of these classes correspond to specific $12_3$ geometric configurations ($D_{222}$, $D_{226}$, and $D_{228}$).
    \item Finally, in \textbf{Section 12}, we present the complete Classification Theorem. We show that while the topological class of Western harmony contains twelve isomorphic "shadows," a strict hierarchy exists among them. We distinguish a "privileged quartet" of systems generated by coprime intervals and conclude with a number-theoretic outlook based on the Chinese Remainder Theorem. This leads to the surprising observation that the Western system is perhaps only an acoustic approximation of a more fundamental, "Canonical" harmony defined by the structure of 12 itself.
\end{itemize}

Most strikingly, we demonstrate that the geometric structure underpinning Bach and Beethoven is not unique to the interval set $\{3,4,7\}$. There exist eleven other systems mathematically identical but sonically alien waiting to be explored. This suggests that the ``grammar'' of Western music is a robust topological invariant, separable from the specific ``vocabulary'' of the intervals we have used for the last 500 years.

To prove that these alternative harmonic universes are not merely abstract theoretical constructs, the concluding section of this paper bridges the gap between our algebraic classification and direct auditory experience. We provide concrete acoustic realizations of these mathematical results by translating existing classical repertoire---specifically Händel's \textit{Passacaglia in G minor} and Schubert's \textit{Piano Trio Op.~100}---from the standard $(4,3)$ system into the alien $(9,4)$ harmonic universe. By systematically applying the affine transformation $x \mapsto 5x + 1 \pmod{12}$ to the MIDI data of these canonical works and rendering the results with the physical modeling synthesizer Pianoteq 9 PRO, we invite the reader to step outside the familiar acoustic constraints and actually hear the geometric structure of this mathematically privileged harmony.

\section{General Theory of $n$-TET Harmony}

Our fundamental postulate is that a harmonic system in any cyclic group $\mathbb{Z}_n$ is generated by three intervals: a ``Fifth'' ($q$), a ``Major Third'' ($t$), and a ``Minor Third'' ($s$).

\subsection{The Interval $q$ (The Generalized Fifth)} The interval $q$ plays the structural role of the "Fifth" in our triad. In classical Pythagorean tuning, one typically requires that stacking the interval $q$ repeatedly must eventually visit every note in the chromatic scale before returning to the origin. This condition of maximal connectivity requires:
\begin{equation}
\label{eq:gcd}
\gcd(q, n) = 1.
\end{equation}
For $n = 12$, the values for $q$ that satisfy this condition are $\{1, 5, 7, 11\}$. The choice $q = 7$ corresponds to the classical perfect fifth (700 cents) and serves as the primary focus for the initial sections of this study (Sections 3--7).

However, our Diophantine formulation does not strictly enforce this condition. We treat $q$ as a free parameter ranging over the entire set $\{1, \dots, 11\}$. While choosing a value coprime to $n$ ensures a single connected "circle of fifths," choosing a value where $\gcd(q, n) > 1$ (such as $q=6$ or $q=9$) is mathematically valid. As we will demonstrate in the later sections on exotic systems, such choices lead to "reducible" or disconnected harmonic universes, which nonetheless possess their own internal consistency.


\subsection{The Harmonic Mediants (The Thirds)}
Once the generator $q$ is fixed, we define the harmonic mediants $t$ (Major Third) and $s$ (Minor Third). We impose the ``Pythagorean Consonance Condition,'' which requires that the stacking of a Major Third and a Minor Third yields the Fifth:
\begin{equation}
    t + s \equiv q \pmod n.
\end{equation}
To classify the possible solutions to this equation, we introduce the structural parameter $\Delta$, defined as the difference between the thirds:
\begin{equation}
    t - s \equiv \Delta \pmod n.
\end{equation}
This formulation allows us to treat $\Delta$ as the primary independent variable. Rather than solving these equations anew for every case, we note that the problem is equivalent to finding integer partitions of the generator $q$. Since musical intervals are magnitudes, we consider $t, s \in \{0, \dots, n-1\}$ and impose the strict integer condition $t+s=q$ (avoiding modular wraparound within the triad). Furthermore, to remove the symmetry redundancy where $t$ and $s$ are simply swapped (renaming Major as Minor), we impose the ordering $t \ge s$.
For $n=12$ and $q=7$, we will simply scan through the integer partitions of 7, which corresponds to scanning the possible values of $\Delta$.

\subsection{General Definition of Chords}
With the intervals $q, t, s$ determined for a given $\Delta$, we define the two fundamental species of triads for any root $k \in \mathbb{Z}_{n}$.

\textbf{The Major Triad ($M_k$)} is constructed by stacking the Major Third $t$ on the root $k$, which implicitly generates the Fifth $q$ (since $t+s=q$).
\begin{equation}
    M_k = \{k, k+t, k+q\}\label{maj}
\end{equation}
\textbf{The Minor Triad ($m_k$)} is constructed by stacking the Minor Third $s$ on the root $k$, which implicitly generates the Fifth $q$.
\begin{equation}
    m_k = \{k, k+s, k+q\}\label{min}
\end{equation}
These definitions ensure that every root supports a Major and a Minor triad, both bounded by the same perfect fifth $q$, but divided internally by the distinct mediants determined by $\Delta$.

\subsection{The Geometric Space: Tonnetz and Transformations}
The collection of all such Major and Minor triads forms the vertices of a harmonic graph. Historically, this structure traces back to Euler \cite{Euler1739}, who first represented harmonic relationships on a lattice. It was later formalized by Naumann \cite{Naumann1858}, Oettingen \cite{Oettingen1866}, and Riemann \cite{Riemann1893,Riemann1914} into the \emph{Tonnetz} (tone-network), which we will hereafter refer to as the \textbf{Euler--Riemann Tonnetz}. In our generalized theory, the connectivity of the \textbf{Tonnetz}---our analog of the Euler--Riemann Tonnetz---is defined by three specific voice-leading transformations that map a triad to a neighbor sharing a common dyad (two notes).

Adapting the formalism of \cite{Riemann1893}, we define these transformations in full generality for any $t$ and $s$:
\begin{itemize}
    \item \textbf{P (Parallel):} Exchanges a Major triad with the Minor triad on the same root. They share the ``Fifth'' interval $\{0, q\}$.
    $$ P(M_k) = m_k \quad \text{and} \quad P(m_k) = M_k. $$
    \item \textbf{L (Leading-Tone Exchange):} Exchanges a triad with one that shares the ``Minor Third'' interval $s$. This maps a Major chord to a Minor chord rooted $t$ steps higher.
    $$ L(M_k) = m_{k+t} \quad \text{and} \quad L(m_k) = M_{k-t}. $$
    \item \textbf{R (Relative):} Exchanges a triad with one that shares the ``Major Third'' interval $t$. This maps a Major chord to a Minor chord rooted $s$ steps lower.
    $$ R(M_k) = m_{k-s} \quad \text{and} \quad R(m_k) = M_{k+s}. $$
\end{itemize}
By systematically applying these transformations to the chords generated by each $\Delta$, we can construct the resulting Tonnetz and analyze its topological properties.

\section{Classification of Harmonies in 12-TET}

We now proceed to analyze the three distinct cases derived from the partitions of $q=7$.

\subsection{Overview of Cases}
\begin{itemize}
\item \textbf{Case $\Delta=1$:} There are two solutions here:
  \begin{itemize}
\item The partition $(t,s) = (4,3)$. This corresponds to the standard Major and Minor Thirds. We designate this the \textbf{``Standard}, or \textbf{Narrow}, \textbf{Thirds'' System}.
\item The partition $(t,s)=(10,9)$.  We call this the \textbf{``Far Narrow Thirds'' System}.
  \end{itemize}
\item \textbf{Case $\Delta=3$:} Here as well we have two solutions:
\begin{itemize}
\item
  The partition $(t,s) = (5,2)$. Here, the ``Major'' interval is a Fourth, and the ``Minor'' is a Major Second. We designate this the \textbf{``Wide Thirds'' System}.
\item The partition $(t,s)=(11,8)$. We call this the \textbf{``Far Wide Thirds'' System}.
  \end{itemize}
\item \textbf{Case $\Delta=5$:} The partition $(t,s) = (6,1)$. Here, the ``Major'' interval is a Tritone, and the ``Minor'' is a Semitone. We designate this the \textbf{``Tritone'' System}.
\end{itemize}

\section{The Case $\Delta=1$, $(t,s)=(4,3)$: The Standard Harmony}

We begin with the smallest odd parameter $\Delta=1$. Solving the structural equations yields the solution $(t, s) = (4, 3)$.
In this ``Standard Thirds'' System, as in the standard Western harmony, the ``Major Third'' is an interval of 4 semitones, and the ``Minor Third'' is an interval of 3 semitones.
\subsection{Major and Minor chords} 
Using our generative definitions \eqref{maj}-\eqref{min}, we construct the two families of chords. We list the full set of generated triples in Table \ref{tab:delta1}.
 \begin{table}[H]
   \centering
   \scriptsize 
    \renewcommand{\arraystretch}{0.9}
    \begin{tabular}{|c|c|c|c|c|}
        \hline
        \textbf{Root ($k$)} & \textbf{Major Chord $M_k$} & \textbf{Standard Name} & \textbf{Minor Chord $m_k$} & \textbf{Standard Name} \\
        \hline
        0 & $\{0, 4, 7\}$ & C Major & $\{0, 3, 7\}$ & c minor \\
        1 & $\{1, 5, 8\}$ & C$^\sharp$ Major & $\{1, 4, 8\}$ & c$^\sharp$ minor \\
        2 & $\{2, 6, 9\}$ & D Major & $\{2, 5, 9\}$ & d minor \\
        3 & $\{3, 7, 10\}$ & E$^\flat$ Major & $\{3, 6, 10\}$ & e$^\flat$ minor \\
        4 & $\{4, 8, 11\}$ & E Major & $\{4, 7, 11\}$ & e minor \\
        5 & $\{5, 9, 0\}$ & F Major & $\{5, 8, 0\}$ & f minor \\
        6 & $\{6, 10, 1\}$ & F$^\sharp$ Major & $\{6, 9, 1\}$ & f$^\sharp$ minor \\
        7 & $\{7, 11, 2\}$ & G Major & $\{7, 10, 2\}$ & g minor \\
        8 & $\{8, 0, 3\}$ & A$^\flat$ Major & $\{8, 11, 3\}$ & g$^\sharp$ minor \\
        9 & $\{9, 1, 4\}$ & A Major & $\{9, 0, 4\}$ & a minor \\
        10 & $\{10, 2, 5\}$ & B$^\flat$ Major & $\{10, 1, 5\}$ & b$^\flat$ minor \\
        11 & $\{11, 3, 6\}$ & B Major & $\{11, 2, 6\}$ & b minor \\
        \hline
    \end{tabular}
    \caption{\scriptsize The complete set of harmonic triads generated by the solution $(t,s)=(4,3)$ in 12-TET. We observe an immediate identity: the generated sets $M_k$ and $m_k$ correspond exactly to the standard Major and Minor triads of Western tonal music.}
    \label{tab:delta1}
\end{table}

The explicit enumeration confirms that the choice $\Delta=1$ is not merely one of many theoretical possibilities; it rigorously reproduces the foundational vocabulary of the common practice period. The intervals 4 and 3 correspond to the best rational approximations of the frequency ratios $5:4$ and $6:5$ available in 12-TET.
\subsection{The Euler -- Riemann Tonnetz}
We now apply the {\bf P}-{\bf L}-{\bf R} transformations to obtain system's $(4,3)$ harmonic network, i.e. its Tonnetz. The lists of three minor chords connected to each Major chord, and of three Major chords connected to each minor chord are given below.

\begin{table}[H]
    \centering
    \scriptsize
    \renewcommand{\arraystretch}{0.7}
    \caption{\scriptsize Transformations of Major Chords ($M_k$)}
    \label{tab:major_neighbors}
    \begin{tabular}{|c|c||c|c|c|}
        \hline
        \textbf{Chord} & \textbf{Name} & \textbf{P-Neighbor} ($m_k$) & \textbf{L-Neighbor} ($m_{k+4}$) & \textbf{R-Neighbor} ($m_{k+9}$) \\
        \hline
        $M_0$ & C Major & $m_0$ (c min) & $m_4$ (e min) & $m_9$ (a min) \\
        $M_1$ & C$^\sharp$ Major & $m_1$ (c$^\sharp$ min) & $m_5$ (f min) & $m_{10}$ (b$^\flat$ min) \\
        $M_2$ & D Major & $m_2$ (d min) & $m_6$ (f$^\sharp$ min) & $m_{11}$ (b min) \\
        $M_3$ & E$^\flat$ Major & $m_3$ (e$^\flat$ min) & $m_7$ (g min) & $m_0$ (c min) \\
        $M_4$ & E Major & $m_4$ (e min) & $m_8$ (g$^\sharp$ min) & $m_1$ (c$^\sharp$ min) \\
        $M_5$ & F Major & $m_5$ (f min) & $m_9$ (a min) & $m_2$ (d min) \\
        $M_6$ & F$^\sharp$ Major & $m_6$ (f$^\sharp$ min) & $m_{10}$ (b$^\flat$ min) & $m_3$ (e$^\flat$ min) \\
        $M_7$ & G Major & $m_7$ (g min) & $m_{11}$ (b min) & $m_4$ (e min) \\
        $M_8$ & A$^\flat$ Major & $m_8$ (g$^\sharp$ min) & $m_0$ (c min) & $m_5$ (f min) \\
        $M_9$ & A Major & $m_9$ (a min) & $m_1$ (c$^\sharp$ min) & $m_6$ (f$^\sharp$ min) \\
        $M_{10}$ & B$^\flat$ Major & $m_{10}$ (b$^\flat$ min) & $m_2$ (d min) & $m_7$ (g min) \\
        $M_{11}$ & B Major & $m_{11}$ (b min) & $m_3$ (e$^\flat$ min) & $m_8$ (g$^\sharp$ min) \\
        \hline
    \end{tabular}
\end{table}

\begin{table}[h!]
    \centering
    \scriptsize
    \renewcommand{\arraystretch}{0.7}
    \caption{\scriptsize Transformations of Minor Chords ($m_k$)}
    \label{tab:minor_neighbors}
    \begin{tabular}{|c|c||c|c|c|}
        \hline
        \textbf{Chord} & \textbf{Name} & \textbf{P-Neighbor} ($M_k$) & \textbf{L-Neighbor} ($M_{k+8}$) & \textbf{R-Neighbor} ($M_{k+3}$) \\
        \hline
        $m_0$ & c minor & $M_0$ (C Maj) & $M_8$ (A$^\flat$ Maj) & $M_3$ (E$^\flat$ Maj) \\
        $m_1$ & c$^\sharp$ minor & $M_1$ (C$^\sharp$ Maj) & $M_9$ (A Maj) & $M_4$ (E Maj) \\
        $m_2$ & d minor & $M_2$ (D Maj) & $M_{10}$ (B$^\flat$ Maj) & $M_5$ (F Maj) \\
        $m_3$ & e$^\flat$ minor & $M_3$ (E$^\flat$ Maj) & $M_{11}$ (B Maj) & $M_6$ (F$^\sharp$ Maj) \\
        $m_4$ & e minor & $M_4$ (E Maj) & $M_0$ (C Maj) & $M_7$ (G Maj) \\
        $m_5$ & f minor & $M_5$ (F Maj) & $M_1$ (C$^\sharp$ Maj) & $M_8$ (A$^\flat$ Maj) \\
        $m_6$ & f$^\sharp$ minor & $M_6$ (F$^\sharp$ Maj) & $M_2$ (D Maj) & $M_9$ (A Maj) \\
        $m_7$ & g minor & $M_7$ (G Maj) & $M_3$ (E$^\flat$ Maj) & $M_{10}$ (B$^\flat$ Maj) \\
        $m_8$ & g$^\sharp$ minor & $M_8$ (A$^\flat$ Maj) & $M_4$ (E Maj) & $M_{11}$ (B Maj) \\
        $m_9$ & a minor & $M_9$ (A Maj) & $M_5$ (F Maj) & $M_0$ (C Maj) \\
        $m_{10}$ & b$^\flat$ minor & $M_{10}$ (B$^\flat$ Maj) & $M_6$ (F$^\sharp$ Maj) & $M_1$ (C$^\sharp$ Maj) \\
        $m_{11}$ & b minor & $M_{11}$ (B Maj) & $M_7$ (G Maj) & $M_2$ (D Maj) \\
        \hline
    \end{tabular}
\end{table}

\subsection{The Levi Graph Visualization}\label{sec43}
We visualize the harmonic space as the incidence graph of these triads (the Levi Graph). Figure \ref{fig:levi12_combined} arranges the 24 chords in two distinct layouts to highlight the topological symmetry.

\begin{figure}[H]
    \centering
    \begin{minipage}{0.48\textwidth}
        \centering
        \begin{tikzpicture}[scale=0.36, every node/.style={transform shape}]
            \def\R{5.5}
            \def\L{6.8}

            \foreach \k in {0,...,11} {
                \pgfmathsetmacro{\angleM}{90 - \k*30}
                \pgfmathsetmacro{\anglem}{90 - \k*30 - 15}
                \coordinate (M\k) at (\angleM:\R);
                \coordinate (m\k) at (\anglem:\R);
                \node[circle, fill=blue!30, draw=blue, minimum size=0.8cm, inner sep=0pt] at (M\k) {};
                \node[circle, fill=red!30, draw=red, minimum size=0.8cm, inner sep=0pt] at (m\k) {};
            }

            \foreach \k in {0,...,11} {
                \draw[thick, black] (M\k) -- (m\k); 
                \pgfmathsetmacro{\nextL}{int(mod(\k+4,12))}
                \draw[gray, thick] (M\k) -- (m\nextL); 
                \pgfmathsetmacro{\nextR}{int(mod(\k+9,12))}
                \draw[gray!70, thick, dashed] (M\k) -- (m\nextR); 
            }

            \def\names{{"C","C$^\sharp$","D","E$^\flat$","E","F","F$^\sharp$","G","A$^\flat$","A","B$^\flat$","B"}}
            \def\minnames{{"cm","c$^\sharp$m","dm","e$^\flat$m","em","fm","f$^\sharp$m","gm","g$^\sharp$m","am","b$^\flat$m","bm"}}
            \foreach \k in {0,...,11} {
                \pgfmathsetmacro{\angleM}{90 - \k*30}
                \pgfmathsetmacro{\anglem}{90 - \k*30 - 15}
                \pgfmathparse{\names[\k]} \let\currname\pgfmathresult
                \node[blue!80!black] at (\angleM:\L) {\textbf{\currname}};
                \pgfmathparse{\minnames[\k]} \let\currname\pgfmathresult
                \node[red!80!black] at (\anglem:\L) {\textbf{\currname}};
                \coordinate (M\k) at (\angleM:\R);
                \coordinate (m\k) at (\anglem:\R);
                \node[circle, fill=blue!30, draw=blue, minimum size=0.8cm, inner sep=0pt] at (M\k) {$M_{\k}$};
                \node[circle, fill=red!30, draw=red, minimum size=0.8cm, inner sep=0pt] at (m\k) {$m_{\k}$};
            }
        \end{tikzpicture}
    \end{minipage}
    \hfill
    \begin{minipage}{0.48\textwidth}
        \centering
        \begin{tikzpicture}[scale=0.36, every node/.style={transform shape}]
            \def\R{5.5}
            \def\L{6.8}

            \def\names{{"C","C$^\sharp$","D","E$^\flat$","E","F","F$^\sharp$","G","A$^\flat$","A","B$^\flat$","B"}}
            \def\minnames{{"cm","c$^\sharp$m","dm","e$^\flat$m","em","fm","f$^\sharp$m","gm","g$^\sharp$m","am","b$^\flat$m","bm"}}

            \foreach \s in {0,...,11} {
                \pgfmathsetmacro{\k}{int(mod(5*\s, 12))}      
                \pgfmathsetmacro{\j}{int(mod(9 + 5*\s, 12))}  
                \pgfmathsetmacro{\angleM}{90 - (\s*2)*15}
                \pgfmathsetmacro{\anglem}{90 - (\s*2 + 1)*15}

                \coordinate (M\k) at (\angleM:\R);
                \coordinate (m\j) at (\anglem:\R);

                \node[circle, fill=blue!30, draw=blue, minimum size=0.8cm, inner sep=0pt] at (M\k) {};
                \node[circle, fill=red!30, draw=red, minimum size=0.8cm, inner sep=0pt] at (m\j) {};

                \pgfmathparse{\names[\k]} \let\currname\pgfmathresult
                \node[blue!80!black] at (\angleM:\L) {\textbf{\currname}};
                \pgfmathparse{\minnames[\j]} \let\currname\pgfmathresult
                \node[red!80!black] at (\anglem:\L) {\textbf{\currname}};
            }

            \foreach \k in {0,...,11} {
                \draw[thick, black] (M\k) -- (m\k); 
                \pgfmathsetmacro{\nextL}{int(mod(\k+4,12))}
                \draw[gray!60, thick] (M\k) -- (m\nextL); 
                \pgfmathsetmacro{\nextR}{int(mod(\k+9,12))}
                \draw[gray!60, thick, dashed] (M\k) -- (m\nextR); 
            }

            \foreach \s in {0,...,11} {
                \pgfmathsetmacro{\k}{int(mod(5*\s, 12))}
                \pgfmathsetmacro{\j}{int(mod(9 + 5*\s, 12))}
                \pgfmathsetmacro{\angleM}{90 - (\s*2)*15}
                \pgfmathsetmacro{\anglem}{90 - (\s*2 + 1)*15}
                \node[circle, fill=blue!30, draw=blue, minimum size=0.8cm, inner sep=0pt] at (\angleM:\R) {$M_{\k}$};
                \node[circle, fill=red!30, draw=red, minimum size=0.8cm, inner sep=0pt] at (\anglem:\R) {$m_{\j}$};
            }
        \end{tikzpicture}
    \end{minipage}

    \caption{\scriptsize Two topological views of the $\Delta=1$ Harmonic System (The Euler-Riemann Tonnetz).
    \textbf{Left:} The ``Chromatic'' layout. Chords are arranged by semitonal root progression ($M_0, m_0, M_1, m_1 \dots$). Here, the Parallel transformation \textbf{P} (black lines) connects immediate neighbors.
    \textbf{Right:} The \textbf{``Cycle of Fourths''} layout \cite{lane}. Chords are arranged by alternating Relative and Leading-tone steps ($M_0, m_9, M_5, m_2 \dots$).
    The dashed lines specifically denote the \textbf{Relative (R)} transformation, distinguished from the \textbf{Leading-tone (L)} transformation (gray solid). In the Right figure, this distinction highlights that the perimeter of the graph consists entirely of alternating L and R steps (the ``hexatonic cycle''), while P transformations form the internal cross-connections.}
    \label{fig:levi12_combined}
\end{figure}
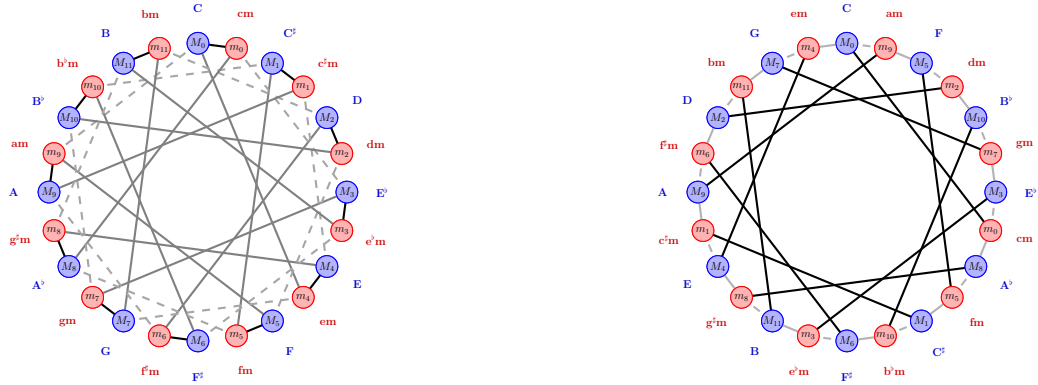
\subsection{Properties of the Levi Graph ($12_3$ Configuration)}
Since we have identified the solution $(4,3)$ as the standard harmonic system, its associated Levi graph is precisely the \textbf{Euler--Riemann Tonnetz}. This graph corresponds to the $12_3$ configuration discussed in \cite{lane}.
As established in \cite{lane}, this graph possesses several remarkable graph-theoretic properties:
\begin{enumerate}
    \item \textbf{Girth 6:} The shortest cycles in the graph have length 6.
    \item \textbf{Bipartite:} The graph is bipartite (Major vs. Minor chords).
    \item \textbf{Hamiltonicity:} The graph admits a Hamiltonian cycle visiting every chord exactly once.
\end{enumerate}

We illustrate these properties in Figure \ref{fig:levi12_cycles} using the \textbf{``Cycle of Fourths''} layout (Figure \ref{fig:levi12_combined}b), which is particularly well-suited for displaying these symmetries.

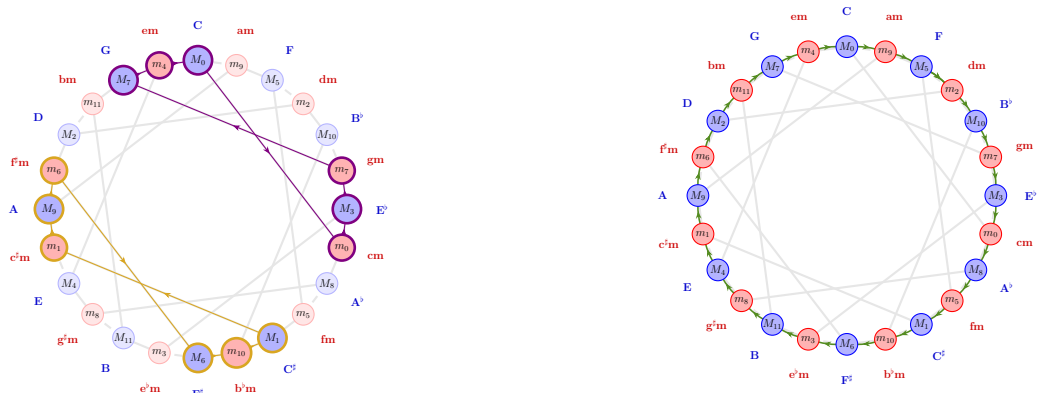
\begin{figure}[H]
    \centering
    \definecolor{mygold}{rgb}{0.85, 0.65, 0.13}
    \definecolor{violet}{rgb}{0.5, 0.0, 0.5}
    \tikzset{midarrow/.style={decoration={markings, mark=at position 0.5 with {\arrow{Stealth[scale=0.6]}}}, postaction={decorate}}}
    \tikzset{midarrowrot/.style={decoration={markings, mark=at position 0.6 with {\arrow{Stealth[scale=0.5]}}}, postaction={decorate}}}

    \begin{minipage}{0.48\textwidth}
        \centering
        \begin{tikzpicture}[scale=0.36, every node/.style={transform shape}]
            \def\R{5.5}
            \def\L{6.8}

            \def\names{{"C","C$^\sharp$","D","E$^\flat$","E","F","F$^\sharp$","G","A$^\flat$","A","B$^\flat$","B"}}
            \def\minnames{{"cm","c$^\sharp$m","dm","e$^\flat$m","em","fm","f$^\sharp$m","gm","g$^\sharp$m","am","b$^\flat$m","bm"}}

            \foreach \s in {0,...,11} {
                \pgfmathsetmacro{\k}{int(mod(5*\s, 12))}
                \pgfmathsetmacro{\j}{int(mod(9 + 5*\s, 12))}
                \pgfmathsetmacro{\angleM}{90 - (\s*2)*15}
                \pgfmathsetmacro{\anglem}{90 - (\s*2 + 1)*15}
                \coordinate (M\k) at (\angleM:\R);
                \coordinate (m\j) at (\anglem:\R);

                \pgfmathparse{\names[\k]} \let\currname\pgfmathresult
                \node[blue!80!black] at (\angleM:\L) {\textbf{\currname}};
                \pgfmathparse{\minnames[\j]} \let\currname\pgfmathresult
                \node[red!80!black] at (\anglem:\L) {\textbf{\currname}};
            }
            \foreach \k in {0,...,11} {
                \pgfmathsetmacro{\nextL}{int(mod(\k+4,12))}
                \pgfmathsetmacro{\nextR}{int(mod(\k+9,12))}
                \draw[black!10, thick] (M\k) -- (m\k); 
                \draw[black!10, thick] (M\k) -- (m\nextL); 
                \draw[black!10, thick, dashed] (M\k) -- (m\nextR); 
            }


            \draw[mygold,  thin, midarrow] (M1) -- (m1);
            \draw[mygold, thin, midarrow] (m1) -- (M9);
            \draw[mygold, thin, midarrow] (M9) -- (m6);
            \draw[mygold, thin, midarrow] (m6) -- (M6);
            \draw[mygold, thin, midarrow] (M6) -- (m10);
            \draw[mygold, thin, midarrow] (m10) -- (M1);

            \draw[violet, thin, midarrow] (M0) -- (m0);
            \draw[violet, thin, midarrow] (m0) -- (M3);
            \draw[violet, thin, midarrow] (M3) -- (m7);
            \draw[violet, thin, midarrow] (m7) -- (M7);
            \draw[violet, thin, midarrow] (M7) -- (m4);
            \draw[violet, thin, midarrow] (m4) -- (M0);

            \foreach \s in {0,...,11} {
                \pgfmathsetmacro{\k}{int(mod(5*\s, 12))}
                \pgfmathsetmacro{\j}{int(mod(9 + 5*\s, 12))}
                \pgfmathsetmacro{\angleM}{90 - (\s*2)*15}
                \pgfmathsetmacro{\anglem}{90 - (\s*2 + 1)*15}
                \node[circle, fill=blue!10, draw=blue!30, minimum size=0.8cm, inner sep=0pt] at (\angleM:\R) {$M_{\k}$};
                \node[circle, fill=red!10, draw=red!30, minimum size=0.8cm, inner sep=0pt] at (\anglem:\R) {$m_{\j}$};
            }

            \foreach \n in {1,6,9} \node[circle, fill=blue!30, draw=mygold, line width=1pt, minimum size=0.8cm] at (M\n) {$M_{\n}$};
            \foreach \n in {1,6,10} \node[circle, fill=red!30, draw=mygold, line width=1pt, minimum size=0.8cm] at (m\n) {$m_{\n}$};

            \foreach \n in {0,3,7} \node[circle, fill=blue!30, draw=violet, line width=1pt, minimum size=0.8cm] at (M\n) {$M_{\n}$};
            \foreach \n in {0,4,7} \node[circle, fill=red!30, draw=violet, line width=1pt, minimum size=0.8cm] at (m\n) {$m_{\n}$};

        \end{tikzpicture}
        \subcaption{\scriptsize Minimal cycles ($g=6$) of opposed chiralities.
        \textcolor{mygold}{\textbf{Gold}}: The cycle around Note C$^\sharp$ (1) ($P$-$L$-$R$ order).
        \textcolor{violet}{\textbf{Violet}}: The cycle around Note G (7) ($P$-$R$-$L$ order).}
    \end{minipage}
    \hfill
    \begin{minipage}{0.48\textwidth}
        \centering
        \begin{tikzpicture}[scale=0.36, every node/.style={transform shape}]
            \def\R{5.5}
            \def\L{6.8}

            \def\names{{"C","C$^\sharp$","D","E$^\flat$","E","F","F$^\sharp$","G","A$^\flat$","A","B$^\flat$","B"}}
            \def\minnames{{"cm","c$^\sharp$m","dm","e$^\flat$m","em","fm","f$^\sharp$m","gm","g$^\sharp$m","am","b$^\flat$m","bm"}}

            \foreach \s in {0,...,11} {
                \pgfmathsetmacro{\k}{int(mod(5*\s, 12))}
                \pgfmathsetmacro{\j}{int(mod(9 + 5*\s, 12))}
                \pgfmathsetmacro{\angleM}{90 - (\s*2)*15}
                \pgfmathsetmacro{\anglem}{90 - (\s*2 + 1)*15}
                \coordinate (M\k) at (\angleM:\R);
                \coordinate (m\j) at (\anglem:\R);

                \pgfmathparse{\names[\k]} \let\currname\pgfmathresult
                \node[blue!80!black] at (\angleM:\L) {\textbf{\currname}};
                \pgfmathparse{\minnames[\j]} \let\currname\pgfmathresult
                \node[red!80!black] at (\anglem:\L) {\textbf{\currname}};
            }
            \foreach \k in {0,...,11} {
                \pgfmathsetmacro{\nextL}{int(mod(\k+4,12))}
                \pgfmathsetmacro{\nextR}{int(mod(\k+9,12))}
                \draw[black!10, thick] (M\k) -- (m\k);
                \draw[black!10, thick] (M\k) -- (m\nextL);
                \draw[black!10, thick, dashed] (M\k) -- (m\nextR);
            }

            \draw[darkgreen, line width=0.6pt, midarrowrot] (M0) to[bend left=10] (m9);
            \draw[darkgreen, line width=0.6pt, midarrowrot] (m9) to[bend left=10] (M5);
            \draw[darkgreen, line width=0.6pt, midarrowrot] (M5) to[bend left=10] (m2);
             \foreach \s in {1,...,10} {
                \pgfmathsetmacro{\k}{int(mod(5*\s, 12))}
                \pgfmathsetmacro{\j}{int(mod(9 + 5*\s, 12))}
                \pgfmathsetmacro{\knext}{int(mod(5*(\s+1), 12))}
                \pgfmathsetmacro{\jnext}{int(mod(9 + 5*(\s+1), 12))}

                \draw[darkgreen, line width=0.6pt, midarrowrot] (m\j) to[bend left=10] (M\knext);
                \draw[darkgreen, line width=0.6pt, midarrowrot] (M\k) to[bend left=10] (m\j);
            }
            \draw[darkgreen, line width=0.6pt, midarrowrot] (M7) to[bend left=10] (m4);
            \draw[darkgreen, line width=0.6pt, midarrowrot] (m4) to[bend left=10] (M0);

            \foreach \s in {0,...,11} {
                \pgfmathsetmacro{\k}{int(mod(5*\s, 12))}
                \pgfmathsetmacro{\j}{int(mod(9 + 5*\s, 12))}
                \pgfmathsetmacro{\angleM}{90 - (\s*2)*15}
                \pgfmathsetmacro{\anglem}{90 - (\s*2 + 1)*15}
                \node[circle, fill=blue!30, draw=blue, minimum size=0.8cm, inner sep=0pt] at (\angleM:\R) {$M_{\k}$};
                \node[circle, fill=red!30, draw=red, minimum size=0.8cm, inner sep=0pt] at (\anglem:\R) {$m_{\j}$};
            }
        \end{tikzpicture}
        \subcaption{\scriptsize A Hamiltonian Cycle ($C_{24}$). In the RL-layout, the perimeter forms a continuous cycle visiting all 24 chords exactly once. This corresponds to the alternating $R-L$ hexatonic chain.}
    \end{minipage}

    \caption{\scriptsize Cycles in the \texorpdfstring{$\Delta=1$}{Delta=1} Levi Graph ($12_3$ configuration). Background edges are deemed in light gray to highlight the cycle structures.}
    \label{fig:levi12_cycles}
\end{figure}

\section{The Case $\Delta=3$, $(t,s)=(5,2):$ The ``Wide Thirds'' System}

We now turn to the intermediate case defined by the partition of the generator $q=7$ into $(t, s) = (5, 2)$. In this system, the ``Major Third'' is the Perfect Fourth (5 semitones), and the ``Minor Third'' is the Major Second (2 semitones). The structural parameter is $\Delta = 5 - 2 = 3$.

\subsection{Explicit Enumeration and the Discovery of Identity}
The Major chord $M_k$ is defined by the set $\{k, k+5, k+7\}$, and the Minor chord $m_k$ by the set $\{k, k+2, k+7\}$. To inspect the structure of this harmonic universe, we list the pitch sets explicitly in Figure \ref{fig:chords_delta3_explicit}.

\begin{figure}[H]
  \centering
\scriptsize
    \renewcommand{\arraystretch}{1.1}
\begin{minipage}{0.45\textwidth}
\centering
\textbf{Major Chords ($M_k$)}\\
\begin{tabular}{ccc}
\toprule
Name & Root & \textbf{Set} \\
\midrule
$M_0$ & 0 & \textbf{\{0, 5, 7\}} \\
$M_1$ & 1 & \{1, 6, 8\} \\
$M_2$ & 2 & \{2, 7, 9\} \\
$M_3$ & 3 & \{3, 8, 10\} \\
$M_4$ & 4 & \{4, 9, 11\} \\
$M_5$ & 5 & \{5, 10, 0\} \\
$M_6$ & 6 & \{6, 11, 1\} \\
$M_7$ & 7 & \{7, 0, 2\} \\
$M_8$ & 8 & \{8, 1, 3\} \\
$M_9$ & 9 & \{9, 2, 4\} \\
$M_{10}$ & 10 & \{10, 3, 5\} \\
$M_{11}$ & 11 & \{11, 4, 6\} \\
\bottomrule
\end{tabular}
\end{minipage}
\hfill
\begin{minipage}{0.45\textwidth}
\centering
\textbf{Minor Chords ($m_k$)}\\
\begin{tabular}{ccc}
\toprule
Name & Root & \textbf{Set} \\
\midrule
$m_0$ & 0 & \{0, 2, 7\} \\
$m_1$ & 1 & \{1, 3, 8\} \\
$m_2$ & 2 & \{2, 4, 9\} \\
$m_3$ & 3 & \{3, 5, 10\} \\
$m_4$ & 4 & \{4, 6, 11\} \\
$m_5$ & 5 & \textbf{\{5, 7, 0\}} \\
$m_6$ & 6 & \{6, 8, 1\} \\
$m_7$ & 7 & \{7, 9, 2\} \\
$m_8$ & 8 & \{8, 10, 3\} \\
$m_9$ & 9 & \{9, 11, 4\} \\
$m_{10}$ & 10 & \{10, 0, 5\} \\
$m_{11}$ & 11 & \{11, 1, 6\} \\
\bottomrule
\end{tabular}
\end{minipage}
\caption{\scriptsize Explicit pitch sets of the ``Wide Thirds'' system ($\Delta=3$). Note the highlighted entries: $M_0$ and $m_5$ are identical sets.}
\label{fig:chords_delta3_explicit}
\end{figure}

\subsubsection{Theorem of Modal Degeneracy}
Inspection of the tables above reveals a startling fact. Compare, for example, the major chord rooted at 0 with the minor chord rooted at 5:
\begin{itemize}
    \item $M_0 = \{0, 5, 7\}$
    \item $m_5 = \{5, 7, 0\}$
\end{itemize}
These are the same set. This is not an isolated anomaly but a general property of the system.

\begin{theorem}[Modal Identity in $\Delta=3$]
In the 12-TET system with $(t=5, s=2)$, the set of pitch classes forming a major chord $M_k$ is identical to the set forming the minor chord $m_{k+5}$.
\[ M_k \equiv m_{k+5} \pmod{12}. \]
\end{theorem}

\begin{proof}
$M_k = \{k, k+5, k+7\}$.
$m_{k+5} = \{(k+5), (k+5)+2, (k+5)+7\} = \{k+5, k+7, k+12\} \equiv \{k+5, k+7, k\}$.
The sets are identical.
\end{proof}

\subsection{Musical Interpretation of Modal Degeneracy}

This mathematical identity has profound musical consequences for the ``Wide Thirds'' harmonic system.

\begin{itemize}
    \item \textbf{Structural Collapse:} In the standard $\Delta=1$ system, a Major triad (e.g., C-E-G) and its relative Minor (A-C-E) are distinct objects that share two tones. In the $\Delta=3$ system, the Major and relative Minor chords are the \emph{same physical object} sharing all three tones. Consequently, although we functionally label 24 chords ($12 \times M_k$ and $12 \times m_k$), there are only 12 unique physical triads in the entire universe. For instance, the pitch set $\{C, F, G\}$ is simultaneously the Major chord of C ($M_0$) and the Minor chord of F ($m_5$).

    \item \textbf{Context-Dependent Perception:} The interval structure of the chord, defined by the modular distances $\{5, 2, 5\}$ (corresponding to Fourth--Second--Fourth), is symmetric. The perceived quality of the chord depends entirely on inversion and context (specifically, which note is established as the root or bass):
    \begin{itemize}
        \item If the pitch $k$ is established as the root (e.g., C), the intervals above it are $5+2$ (Perfect Fourth + Major Second). The chord functions as \textbf{Major} ($M_k$) and is heard as a suspended fourth chord ($Csus4$).
        \item If the pitch $k+5$ is established as the root (e.g., F), the intervals above it are $2+5$ (Major Second + Perfect Fourth). The chord functions as \textbf{Minor} ($m_{k+5}$) and is heard as a suspended second chord ($Fsus2$).
    \end{itemize}

    \item \textbf{The Janus-Faced Triad:} Physically, the $\Delta=3$ triad is an ambiguous, ``Janus-faced'' harmonic object. It functions as a bridge between the $M$ and $m$ modes not by changing notes (voice leading), but by changing the listener's perspective (re-interpreting the gravitational center). This mathematical degeneracy formalizes the well-known equivalence between quartal and secundal voicings often utilized in modal jazz.
\end{itemize}

\subsection{The ``Wide Thirds'' System Tonnetz}

We now establish the connectivity of the harmonic space. The transformations $P, L, R$ are defined as usual:
\begin{itemize}
    \item \textbf{Parallel ($P$):} Preserves the generator (7). $M_k \leftrightarrow m_k$.
    \item \textbf{Leading-Tone ($L$):} Preserves the minor third (2). $M_k \leftrightarrow m_{k+5}$.
    \item \textbf{Relative ($R$):} Preserves the major third (5). $M_k \leftrightarrow m_{k-2}$ (or $m_{k+10}$).
\end{itemize}

However, because of the Modal Identity ($M_k \equiv m_{k+5}$), the **Leading-Tone ($L$)** transformation maps the chord to itself. Consequently, the edge corresponding to $L$ becomes a loop, and the functional degree of each vertex drops to 2.

Tables \ref{tab:neighbors_M_5_2} and \ref{tab:neighbors_m_5_2} list the neighbors. The $L$ column is crossed out because it represents a trivial identity operation on the set.

\begin{table}[H]
    \centering
    \scriptsize
    \renewcommand{\arraystretch}{1.3}
    
    \begin{minipage}[t]{0.48\textwidth}
        \centering
        \begin{tabular}{|c||c|c|c|}
            \hline
            \textbf{Chord} & \textbf{P-neigh.} & \textbf{L-neigh.} & \textbf{R-neigh.} \\[-0.5ex]
            $(M_k)$ & $(m_k)$ & $(m_{k+5})$ & $(m_{k-2})$ \\
            \hline
            $M_0$ & $m_0$ & \tikz[baseline=-3pt]{\node[inner sep=1pt] (A) {$m_5$}; \draw[red, thick] (A.north west) -- (A.south east) (A.north east) -- (A.south west);} & $m_{10}$ \\
            $M_1$ & $m_1$ & \tikz[baseline=-3pt]{\node[inner sep=1pt] (A) {$m_6$}; \draw[red, thick] (A.north west) -- (A.south east) (A.north east) -- (A.south west);} & $m_{11}$ \\
            $M_2$ & $m_2$ & \tikz[baseline=-3pt]{\node[inner sep=1pt] (A) {$m_7$}; \draw[red, thick] (A.north west) -- (A.south east) (A.north east) -- (A.south west);} & $m_0$ \\
            $M_3$ & $m_3$ & \tikz[baseline=-3pt]{\node[inner sep=1pt] (A) {$m_8$}; \draw[red, thick] (A.north west) -- (A.south east) (A.north east) -- (A.south west);} & $m_1$ \\
            $M_4$ & $m_4$ & \tikz[baseline=-3pt]{\node[inner sep=1pt] (A) {$m_9$}; \draw[red, thick] (A.north west) -- (A.south east) (A.north east) -- (A.south west);} & $m_2$ \\
            $M_5$ & $m_5$ & \tikz[baseline=-3pt]{\node[inner sep=1pt] (A) {$m_{10}$}; \draw[red, thick] (A.north west) -- (A.south east) (A.north east) -- (A.south west);} & $m_3$ \\
            $M_6$ & $m_6$ & \tikz[baseline=-3pt]{\node[inner sep=1pt] (A) {$m_{11}$}; \draw[red, thick] (A.north west) -- (A.south east) (A.north east) -- (A.south west);} & $m_4$ \\
            $M_7$ & $m_7$ & \tikz[baseline=-3pt]{\node[inner sep=1pt] (A) {$m_0$}; \draw[red, thick] (A.north west) -- (A.south east) (A.north east) -- (A.south west);} & $m_5$ \\
            $M_8$ & $m_8$ & \tikz[baseline=-3pt]{\node[inner sep=1pt] (A) {$m_1$}; \draw[red, thick] (A.north west) -- (A.south east) (A.north east) -- (A.south west);} & $m_6$ \\
            $M_9$ & $m_9$ & \tikz[baseline=-3pt]{\node[inner sep=1pt] (A) {$m_2$}; \draw[red, thick] (A.north west) -- (A.south east) (A.north east) -- (A.south west);} & $m_7$ \\
            $M_{10}$ & $m_{10}$ & \tikz[baseline=-3pt]{\node[inner sep=1pt] (A) {$m_3$}; \draw[red, thick] (A.north west) -- (A.south east) (A.north east) -- (A.south west);} & $m_8$ \\
            $M_{11}$ & $m_{11}$ & \tikz[baseline=-3pt]{\node[inner sep=1pt] (A) {$m_4$}; \draw[red, thick] (A.north west) -- (A.south east) (A.north east) -- (A.south west);} & $m_9$ \\
            \hline
        \end{tabular}
        \caption{\scriptsize Neighbors of Major chords. The $L$ column is crossed out because $m_{k+5} \equiv M_k$.}
        \label{tab:neighbors_M_5_2}
    \end{minipage}
    \hfill
    \begin{minipage}[t]{0.48\textwidth}
        \centering
        \begin{tabular}{|c||c|c|c|}
            \hline
            \textbf{Chord} & \textbf{P-neigh.} & \textbf{L-neigh.} & \textbf{R-neigh.} \\[-0.5ex]
            $(m_k)$ & $(M_k)$ & $(M_{k+7})$ & $(M_{k+2})$ \\
            \hline
            $m_0$ & $M_0$ & \tikz[baseline=-3pt]{\node[inner sep=1pt] (A) {$M_7$}; \draw[red, thick] (A.north west) -- (A.south east) (A.north east) -- (A.south west);} & $M_2$ \\
            $m_1$ & $M_1$ & \tikz[baseline=-3pt]{\node[inner sep=1pt] (A) {$M_8$}; \draw[red, thick] (A.north west) -- (A.south east) (A.north east) -- (A.south west);} & $M_3$ \\
            $m_2$ & $M_2$ & \tikz[baseline=-3pt]{\node[inner sep=1pt] (A) {$M_9$}; \draw[red, thick] (A.north west) -- (A.south east) (A.north east) -- (A.south west);} & $M_4$ \\
            $m_3$ & $M_3$ & \tikz[baseline=-3pt]{\node[inner sep=1pt] (A) {$M_{10}$}; \draw[red, thick] (A.north west) -- (A.south east) (A.north east) -- (A.south west);} & $M_5$ \\
            $m_4$ & $M_4$ & \tikz[baseline=-3pt]{\node[inner sep=1pt] (A) {$M_{11}$}; \draw[red, thick] (A.north west) -- (A.south east) (A.north east) -- (A.south west);} & $M_6$ \\
            $m_5$ & $M_5$ & \tikz[baseline=-3pt]{\node[inner sep=1pt] (A) {$M_0$}; \draw[red, thick] (A.north west) -- (A.south east) (A.north east) -- (A.south west);} & $M_7$ \\
            $m_6$ & $M_6$ & \tikz[baseline=-3pt]{\node[inner sep=1pt] (A) {$M_1$}; \draw[red, thick] (A.north west) -- (A.south east) (A.north east) -- (A.south west);} & $M_8$ \\
            $m_7$ & $M_7$ & \tikz[baseline=-3pt]{\node[inner sep=1pt] (A) {$M_2$}; \draw[red, thick] (A.north west) -- (A.south east) (A.north east) -- (A.south west);} & $M_9$ \\
            $m_8$ & $M_8$ & \tikz[baseline=-3pt]{\node[inner sep=1pt] (A) {$M_3$}; \draw[red, thick] (A.north west) -- (A.south east) (A.north east) -- (A.south west);} & $M_{10}$ \\
            $m_9$ & $M_9$ & \tikz[baseline=-3pt]{\node[inner sep=1pt] (A) {$M_4$}; \draw[red, thick] (A.north west) -- (A.south east) (A.north east) -- (A.south west);} & $M_{11}$ \\
            $m_{10}$ & $M_{10}$ & \tikz[baseline=-3pt]{\node[inner sep=1pt] (A) {$M_5$}; \draw[red, thick] (A.north west) -- (A.south east) (A.north east) -- (A.south west);} & $M_0$ \\
            $m_{11}$ & $M_{11}$ & \tikz[baseline=-3pt]{\node[inner sep=1pt] (A) {$M_6$}; \draw[red, thick] (A.north west) -- (A.south east) (A.north east) -- (A.south west);} & $M_1$ \\
            \hline
        \end{tabular}
        \caption{\scriptsize Neighbors of Minor chords. The $L$ column is crossed out because $M_{k+7} \equiv m_k$.}
        \label{tab:neighbors_m_5_2}
    \end{minipage}
\end{table}

\subsection{The Degenerate Tonnetz and its Levi Graph}

To understand the topology of this system, we analyze the effect of the modal identity on the standard Neo-Riemannian transformations. In the standard $\Delta=1$ system, the $L$ transformation connects a chord to a distinct neighbor ($L: M_k \to m_{k+4}$), creating a graph where every vertex has degree 3.

In the ``Wide Thirds'' $\Delta=3$ system, the Leading-Tone transformation ($M_k \to m_{k+5}$) maps the chord to itself ($M_k \equiv m_{k+5}$). Consequently, the edge corresponding to $L$ becomes a loop (or an identity mapping), and the functional degree of each vertex drops from 3 to 2. The only remaining active connections are $P$ (Parallel) and $R$ (Relative).

To visualize this structure globally, we construct the **Levi graph** of the configuration. We treat the set of 12 Major chords as ``points'' and the set of 12 Minor chords as ``lines''. An incidence (edge) exists if a Major chord transforms into a Minor chord via a Parallel ($P$) or Relative ($R$) motion.

Since each Major chord connects exactly to two distinct Minor chords ($P$ and $R$), and each Minor chord connects to two distinct Major chords, the global structure forms a **Levi graph of a $(12_2)$ configuration**.

We present this graph in Figure \ref{fig:levi_delta3_full} by arranging all 24 chords on a single circle in chromatic order ($M_0, m_0, M_1, m_1, \dots$). The connections reveal the hidden topology.

\begin{figure}[H]
\centering
\begin{tikzpicture}[scale=0.6, transform shape]
    \def\R{6.5} 

    \tikzset{
        majnode/.style={circle, draw=black!70, fill=white, thick, minimum size=0.8cm, inner sep=0pt, font=\scriptsize},
        minnode/.style={circle, draw=black!70, fill=white, thick, minimum size=0.8cm, inner sep=0pt, font=\scriptsize}
    }

    \foreach \k in {0,...,11} {
        \pgfmathsetmacro{\angM}{90 - \k*30}
        \pgfmathsetmacro{\angm}{90 - \k*30 - 15}
        \node[majnode] (M\k) at (\angM:\R) {$M_{\k}$};
        \node[minnode] (m\k) at (\angm:\R) {$m_{\k}$};
    }

    \draw[blue, thick] (M0) -- (m0);
    \draw[blue, thick] (m0) -- (M2);
    \draw[blue, thick] (M2) -- (m2);
    \draw[blue, thick] (m2) -- (M4);
    \draw[blue, thick] (M4) -- (m4);
    \draw[blue, thick] (m4) -- (M6);
    \draw[blue, thick] (M6) -- (m6);
    \draw[blue, thick] (m6) -- (M8);
    \draw[blue, thick] (M8) -- (m8);
    \draw[blue, thick] (m8) -- (M10);
    \draw[blue, thick] (M10) -- (m10);
    \draw[blue, thick] (m10) -- (M0);

    \draw[red, thick] (M1) -- (m1);
    \draw[red, thick] (m1) -- (M3);
    \draw[red, thick] (M3) -- (m3);
    \draw[red, thick] (m3) -- (M5);
    \draw[red, thick] (M5) -- (m5);
    \draw[red, thick] (m5) -- (M7);
    \draw[red, thick] (M7) -- (m7);
    \draw[red, thick] (m7) -- (M9);
    \draw[red, thick] (M9) -- (m9);
    \draw[red, thick] (m9) -- (M11);
    \draw[red, thick] (M11) -- (m11);
    \draw[red, thick] (m11) -- (M1);

\end{tikzpicture}
\caption{The full Levi graph of the ``Wide Thirds'' Tonnetz arranged on a circle. The blue edges connect the chords with even roots, while the red edges connect chords with odd roots. The interleaving of the nodes visually demonstrates the symmetry of the system.}
\label{fig:levi_delta3_full}
\end{figure}
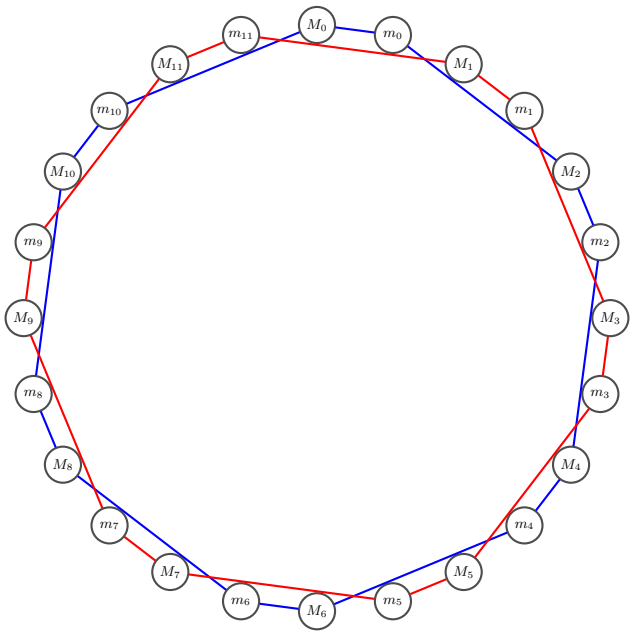

\subsubsection{Topological Decomposition}

The coloring in Figure \ref{fig:levi_delta3_full} makes the fundamental property of the system immediately visible. Although constructed as a single harmonic system, the graph naturally separates into two disjoint, non-intersecting subgraphs.

The blue path traces a closed cycle involving only the chords with even roots ($M_{2k}, m_{2k}$). The red path traces a similar cycle for the odd roots ($M_{2k+1}, m_{2k+1}$). No blue edge ever meets a red edge.

\begin{theorem}[Topological Decomposition]
The 12-TET Tonnetz for the solution $(t=5, s=2)$ is isomorphic to the Levi graph of a $(12_2)$ configuration, which decomposes into two disjoint cycles of length 12.
\end{theorem}

This means the harmonic Universe of the $\Delta=3$ system consists of two parallel ``harmonic realities'' that never intersect. One reality consists exclusively of the chords rooted on the Whole-Tone scale starting on C (0, 2, 4...), and the other on the Whole-Tone scale starting on C$^\sharp$ (1, 3, 5...).

Geometrically, the second graph is a perfect copy of the first, rotated by an angle of $\pi/12$ ($15^\circ$, one step in the circular arrangement). We display the separated components in Figure \ref{fig:levi_delta3_split}, preserving their original geometric positions to highlight the rotation.

\begin{figure}[H]
\centering
\begin{minipage}{0.48\textwidth}
\centering
\textbf{Component A (The Blue Cycle)}
\vspace{0.3cm}
\begin{tikzpicture}[scale=0.6, transform shape]
    \def\R{5.5} 
    \tikzset{majnode/.style={circle, draw=black!70, fill=white, thick, minimum size=0.8cm, inner sep=0pt}}
    \tikzset{minnode/.style={circle, draw=black!70, fill=white, thick, minimum size=0.8cm, inner sep=0pt}}

    \foreach \k in {0, 2, ..., 10} {
        \pgfmathsetmacro{\angM}{90 - \k*30}
        \pgfmathsetmacro{\angm}{90 - \k*30 - 15}
        \node[majnode] (M\k) at (\angM:\R) {$M_{\k}$};
        \node[minnode] (m\k) at (\angm:\R) {$m_{\k}$};
    }
    
    \draw[blue, thick] (M0)--(m0)--(M2)--(m2)--(M4)--(m4)--(M6)--(m6)--(M8)--(m8)--(M10)--(m10)--(M0);
\end{tikzpicture}
\end{minipage}
\hfill
\begin{minipage}{0.48\textwidth}
\centering
\textbf{Component B (The Red Cycle)}
\vspace{0.3cm}
\begin{tikzpicture}[scale=0.6, transform shape]
    \def\R{5.5} 
    \tikzset{majnode/.style={circle, draw=black!70, fill=white, thick, minimum size=0.8cm, inner sep=0pt}}
    \tikzset{minnode/.style={circle, draw=black!70, fill=white, thick, minimum size=0.8cm, inner sep=0pt}}

    \foreach \k in {1, 3, ..., 11} {
        \pgfmathsetmacro{\angM}{90 - \k*30}
        \pgfmathsetmacro{\angm}{90 - \k*30 - 15}
        \node[majnode] (M\k) at (\angM:\R) {$M_{\k}$};
        \node[minnode] (m\k) at (\angm:\R) {$m_{\k}$};
    }
    
    \draw[red, thick] (M1)--(m1)--(M3)--(m3)--(M5)--(m5)--(M7)--(m7)--(M9)--(m9)--(M11)--(m11)--(M1);
\end{tikzpicture}
\end{minipage}
\caption{The decomposition of the 12-TET $\Delta=3$ Tonnetz into two independent cycles. The blue cycle contains all chords rooted on even integers, and the red cycle contains all chords rooted on odd integers.}
\label{fig:levi_delta3_split}
\end{figure}
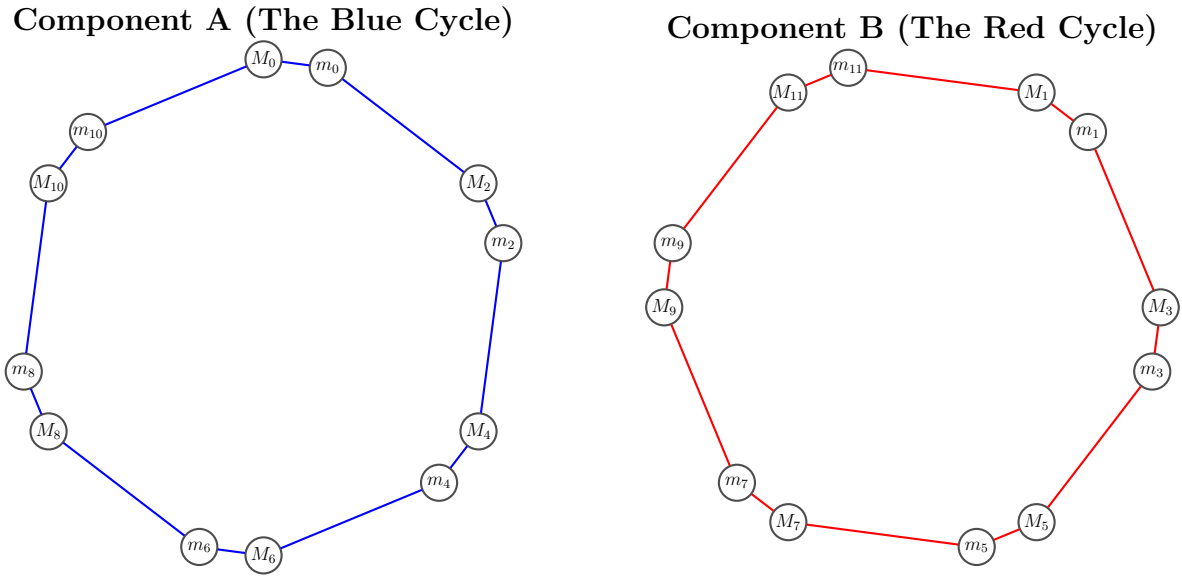

\subsection{The Functional Perspective: Restoring Connectivity in 3D}

Our analysis of ``modal degeneracy'' in Theorem 2 relied on a strict set-theoretical perspective: since the pitch-class sets of $M_k$ and $m_{k+5}$ are identical, we treated them as the same vertex in the Tonnetz. This reduced the graph valence to 2 and caused the topological splitting described in Section 6.4.

However, from a functional (Riemannian) perspective, a chord is defined not just by its pitch content, but by its root. In this view, the Major chord $M_0 = \{0, 5, 7\}$ rooted at 0 ($Csus4$) is distinct from the Minor chord $m_5 = \{5, 7, 0\}$ rooted at 5 ($Fsus2$), despite comprising the same notes. The ``Leading-Tone'' transformation ($L$) is no longer an identity map, but a bijection -- \emph{a wormhole} -- mapping a chord from the ``Even Universe'' to the ``Odd Universe'' (since $k$ and $k+5$ have opposite parity).

If we adopt this functional differentiation, the vertex set doubles from 12 unique sets to 24 functional chords. The topology undergoes a dimensional lift:
\begin{enumerate}
    \item The ``Parallel'' ($P$) and ``Relative'' ($R$) transformations continue to link chords within the same parity cycle (forming the Blue and Red rings derived in Section 6.4).
    \item The ``Leading-Tone'' ($L$) transformation now acts as a transversal edge, connecting every node $M_k$ in the Blue cycle to a node $m_{k+5}$ in the Red cycle (and vice versa).
\end{enumerate}

We explicitly list these restored functional connections in the tables below. Unlike the degenerate case, the $L$-column now provides valid, distinct outputs that link the previously disjoint cycles.

\begin{figure}[H]
  \centering
   \scriptsize
    \renewcommand{\arraystretch}{1.1.}
    \begin{minipage}{0.48\textwidth}
        \centering
        \renewcommand{\arraystretch}{1.0}
        \scalebox{0.85}{
        \begin{tabular}{|c||c|c|c|}
        \hline
        \textbf{Input} & $P$ & $L$ (Wormhole) & $R$ \\
        \hline
        $M_0$ & $m_0$ & $m_5$ & $m_{10}$ \\
        $M_1$ & $m_1$ & $m_6$ & $m_{11}$ \\
        $M_2$ & $m_2$ & $m_7$ & $m_0$ \\
        $M_3$ & $m_3$ & $m_8$ & $m_1$ \\
        $M_4$ & $m_4$ & $m_9$ & $m_2$ \\
        $M_5$ & $m_5$ & $m_{10}$ & $m_3$ \\
        $M_6$ & $m_6$ & $m_{11}$ & $m_4$ \\
        $M_7$ & $m_7$ & $m_0$ & $m_5$ \\
        $M_8$ & $m_8$ & $m_1$ & $m_6$ \\
        $M_9$ & $m_9$ & $m_2$ & $m_7$ \\
        $M_{10}$ & $m_{10}$ & $m_3$ & $m_8$ \\
        $M_{11}$ & $m_{11}$ & $m_4$ & $m_9$ \\
        \hline
        \end{tabular}
        }
        \caption{\scriptsize Functional Transformations on Major Chords. The $L$ transformation is now active, mapping $M_k$ to the functionally distinct $m_{k+5}$.}
    \end{minipage}
    \hfill
    \begin{minipage}{0.48\textwidth}
        \centering
        \renewcommand{\arraystretch}{1.0}
        \scalebox{0.85}{
        \begin{tabular}{|c||c|c|c|}
        \hline
        \textbf{Input} & $P$ & $L$ (Wormhole) & $R$ \\
        \hline
        $m_0$ & $M_0$ & $M_7$ & $M_2$ \\
        $m_1$ & $M_1$ & $M_8$ & $M_3$ \\
        $m_2$ & $M_2$ & $M_9$ & $M_4$ \\
        $m_3$ & $M_3$ & $M_{10}$ & $M_5$ \\
        $m_4$ & $M_4$ & $M_{11}$ & $M_6$ \\
        $m_5$ & $M_5$ & $M_0$ & $M_7$ \\
        $m_6$ & $M_6$ & $M_1$ & $M_8$ \\
        $m_7$ & $M_7$ & $M_2$ & $M_9$ \\
        $m_8$ & $M_8$ & $M_3$ & $M_{10}$ \\
        $m_9$ & $M_9$ & $M_4$ & $M_{11}$ \\
        $m_{10}$ & $M_{10}$ & $M_5$ & $M_0$ \\
        $m_{11}$ & $M_{11}$ & $M_6$ & $M_1$ \\
        \hline
        \end{tabular}
        }
        \caption{\scriptsize Functional Transformations on Minor Chords. The $L$ transformation is now active, mapping $m_k$ to the functionally distinct $M_{k+7}$.}
    \end{minipage}
\end{figure}

Consequently, the graph valence is restored to 3, and the disconnected components merge into a single, connected 3-dimensional structure. The two disjoint cycles of the 2D projection become the bases of a **prism graph** (or cylindrical lattice) in 3D space.

We visualize this ``Functional Cylinder'' below. The Blue and Red cycles form the top and bottom bases. To enhance readability, we apply a rotational offset to the Upper (Red) ring, creating a ``twisted'' cylinder where the functional $L$-connections form the bridges between the universes.

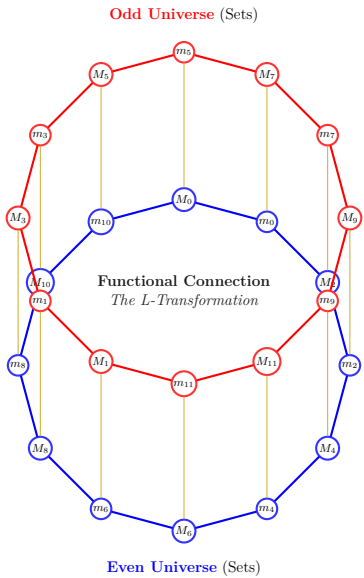
\begin{figure}[H]
\centering
\begin{tikzpicture}[scale=0.49, transform shape]
    \def\R{4.5} 
    \def\h{4.0} 
    \def\angShift{90} 
    \def\step{30} 
    
    \foreach \i/\label [count=\idx from 0] in {
        0/M_0, 1/m_0, 
        2/M_2, 3/m_2, 
        4/M_4, 5/m_4, 
        6/M_6, 7/m_6, 
        8/M_8, 9/m_8, 
        10/M_{10}, 11/m_{10}} {
            \pgfmathsetmacro{\currAng}{\angShift - \idx*\step}
            \node[circle, draw=blue!80, fill=white, thick, inner sep=1pt, font=\scriptsize] (L\idx) at (\currAng:\R) {$\label$};
    }

    \draw[blue, thick] (L0)--(L1)--(L2)--(L3)--(L4)--(L5)--(L6)--(L7)--(L8)--(L9)--(L10)--(L11)--(L0);

    \begin{scope}[yshift=\h cm]
        \foreach \i/\label [count=\idx from 0] in {
            0/m_5, 1/M_7, 
            2/m_7, 3/M_9, 
            4/m_9, 5/M_{11}, 
            6/m_{11}, 7/M_1, 
            8/m_1, 9/M_3, 
            10/m_3, 11/M_5} {
                \pgfmathsetmacro{\currAng}{\angShift - \idx*\step}
                \node[circle, draw=red!80, fill=white, thick, inner sep=1pt, font=\scriptsize] (U\idx) at (\currAng:\R) {$\label$};
        }
        
        \draw[red, thick] (U0)--(U1)--(U2)--(U3)--(U4)--(U5)--(U6)--(U7)--(U8)--(U9)--(U10)--(U11)--(U0);
    \end{scope}

    \foreach \i in {0,...,11} {
        \draw[mygold, thin] (L\i) -- (U\i);
    }

    \node at (0, -\R-1.0) {\textcolor{blue}{\textbf{Even Universe}} (Sets)};
    \node at (0, \h+\R+1.0) {\textcolor{red}{\textbf{Odd Universe}} (Sets)};
    \node[align=center, fill=white, inner sep=2pt, opacity=0.9] at (0, \h/2) {\textbf{Functional Connection}\\\textit{The L-Transformation}};

\end{tikzpicture}
\caption{\scriptsize The ``Functional Cylinder'' of the Wide Thirds System ($\Delta=3$). The Blue and Red cycles form the geometric bases, representing the invariant pitch-class sets. The gold lines represent the Leading-Tone ($L$) transformation, acting as a wormhole between the disjoint universes of Even and Odd roots.}
\label{fig:cylinder_delta3}
\end{figure}

\subsection{Musical Implications: The Phantom Modulation}

This functional restoration of connectivity in 3D has profound musical implications, extending the concept of ``Janus-faced'' triads introduced in Section 6.1.1.

In classical 12-TET harmony, modulation typically involves changing pitch classes (e.g., introducing a sharp or flat). However, in the $\Delta=3$ system, the $L$-transformation represented by the vertical edges of the cylinder corresponds to a \textbf{phantom modulation}. The physical sound remains identical (the pitch-class set is invariant), but the listener is instantaneously transported from the ``Even Universe'' to the ``Odd Universe'' purely by re-interpreting the root.

For example, traversing the wormhole $M_0 \xrightarrow{L} m_5$ keeps the notes $\{C, F, G\}$ constant but shifts the tonal center from $C$ to $F$ (a distance of 5 semitones). The cylinder acts as a topological map of this duality, unifying the disjoint set-theoretical cycles into a single functional object.

\subsection{The Combinatorial Structure: The Cubic Graph}
\label{sec:restored_cubic}

By explicitly distinguishing between $M_k$ and $m_{k+5}$ and treating the identity relation as a valid edge ($L$), the system reveals a highly regular internal geometry. The interactions of $P$, $L$, and $R$ generate a unified graph that can be analyzed using the language of graph theory.

\subsubsection{The Geometric Structure} \label{gs}
In this restored view, every Major chord $M_k$ connects to exactly three Minor chords:
\begin{itemize}
    \item $m_k$ via the Parallel operation $P$ (intra-universe),
    \item $m_{k-2}$ via the Relative operation $R$ (intra-universe),
    \item $m_{k+5}$ via the Wormhole operation $L$ (inter-universe).
\end{itemize}
Symmetrically, every Minor chord connects to exactly three Major chords. This transforms our graph into a connected, cubic (3-valent) bipartite graph with $V=24$ vertices ($12$ Majors and $12$ Minors).

This structure constitutes a \textbf{connected cubic bipartite graph}. We can interpret the 12 Minor chords as ``Points'' and the 12 Major chords as ``Lines''. The incidence relation is defined by our harmonic connections: each Line (Major chord) passes through exactly 3 Points (the Minors it connects to), and each Point (Minor chord) lies on exactly 3 Lines.

However, it should be mentioned that these are not geometric points and lines in the strict sense. As we will see in the next section, in this interpretation the graph admits that two points lie on two different lines and two lines can meet in more than one point.

\subsubsection{Visualization: The Levi Graphs}
The incidence graph of a configuration is formally known as a \textbf{Levi graph}. In Figure \ref{fig:levi_concentric_delta3}, we present this graph in a layout corresponding to the spectral analysis from the previous sections. The Blue and Red universes form two concentric rings, with the $P$ and $R$ operations creating the perimeter of each ring. The $L$ operations (gold lines) act as radial wormholes fusing the two worlds together.

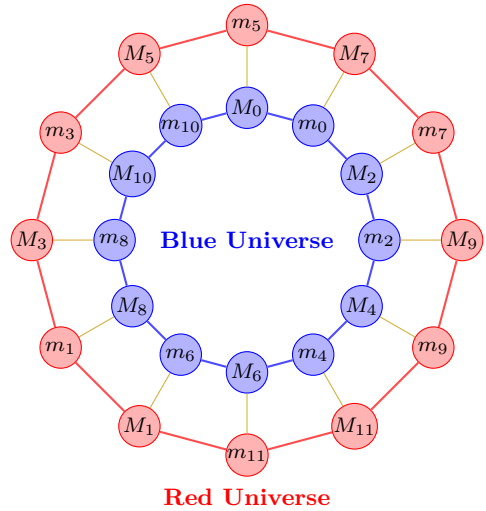
\begin{figure}[H]
    \centering
    \begin{tikzpicture}[scale=0.55, >=stealth]
        \def\Rout{5.2} 
        \def\Rin{3.2}  
        \definecolor{mygold}{RGB}{212, 175, 55} 
        \def\step{30}

        \tikzset{
            mynode/.style={circle, inner sep=0pt, minimum size=0.55cm, font=\tiny}
        }

        \foreach \i/\nid/\lab [count=\xi] in {
            0/m_5/m_{5}, 1/M_7/M_{7}, 
            2/m_7/m_{7}, 3/M_9/M_{9}, 
            4/m_9/m_{9}, 5/M_11/M_{11}, 
            6/m_11/m_{11}, 7/M_1/M_{1}, 
            8/m_1/m_{1}, 9/M_3/M_{3}, 
            10/m_3/m_{3}, 11/M_5/M_{5}} 
        {
            \pgfmathsetmacro{\angle}{90 - \xi*\step + \step}
            \coordinate (RedPos\xi) at (\angle:\Rout);
            \node[mynode, draw=red, fill=red!30, alias=RN_\nid] (RedNode\i) at (RedPos\xi) {$\lab$};
        }
        
        \foreach \i in {0,...,11} {
            \pgfmathsetmacro{\next}{int(mod(\i+1,12))}
            \draw[red!70, thick] (RedNode\i) -- (RedNode\next);
        }

        \foreach \j/\nid/\lab [count=\xj] in {
            0/M_0/M_{0}, 1/m_0/m_{0}, 
            2/M_2/M_{2}, 3/m_2/m_{2}, 
            4/M_4/M_{4}, 5/m_4/m_{4}, 
            6/M_6/M_{6}, 7/m_6/m_{6}, 
            8/M_8/M_{8}, 9/m_8/m_{8}, 
            10/M_10/M_{10}, 11/m_10/m_{10}}
        {
            \pgfmathsetmacro{\angle}{90 - \xj*\step + \step}
            \coordinate (BluePos\xj) at (\angle:\Rin);
            \node[mynode, draw=blue, fill=blue!30, alias=BN_\nid] (BlueNode\j) at (BluePos\xj) {$\lab$};
        }

        \foreach \j in {0,...,11} {
            \pgfmathsetmacro{\next}{int(mod(\j+1,12))}
            \draw[blue!70, thick] (BlueNode\j) -- (BlueNode\next);
        }

        \foreach \k in {0, 2, ..., 10} {
             \pgfmathsetmacro{\targetm}{int(mod(\k+5,12))}
             \draw[thin, color=mygold] (BN_M_\k) -- (RN_m_\targetm);
        }
        \foreach \k in {0, 2, ..., 10} {
             \pgfmathsetmacro{\targetM}{int(mod(\k+7,12))}
             \draw[ thin, color=mygold] (BN_m_\k) -- (RN_M_\targetM);
        }

        \node at (0, 0) {\textcolor{blue}{\textbf{\scriptsize Blue Universe}}};
        \node at (0, -\Rout-1.0) {\textcolor{red}{\textbf{\scriptsize Red Universe}}};

    \end{tikzpicture}
    \caption{The Levi graph representing the \texorpdfstring{$\Delta=3$}{Delta=3} configuration. The Blue and Red universes are visualized as concentric rings formed by \texorpdfstring{$P$}{P} and \texorpdfstring{$R$}{R} connections, connected by radial \texorpdfstring{$L$}{L}-wormholes (gold lines), uniting the system.}
    \label{fig:levi_concentric_delta3}
\end{figure}

In Figure \ref{fig:levi_circular_delta3}, we present two visualizations of the graph. On the left, chords are arranged in the standard sequential notation, highlighting the local alternating nature. On the right, the chords are arranged according to the Hamiltonian cycle, revealing the global symmetry.

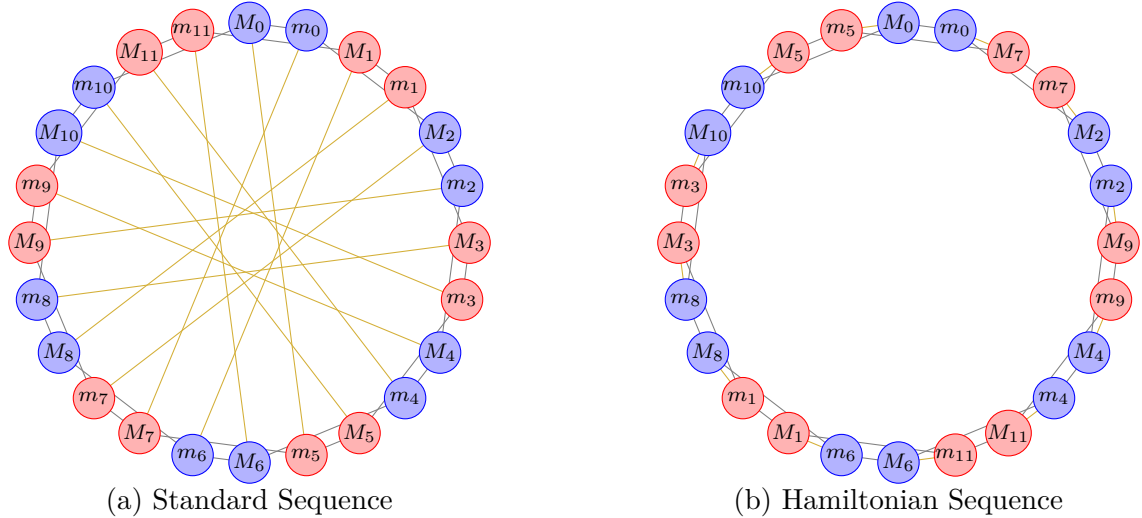
\begin{figure}[H]
    \centering
    \tikzset{
        mynode/.style={circle, inner sep=0pt, minimum size=0.55cm, font=\tiny}
    }

    \begin{minipage}{0.48\textwidth}
        \centering
        \begin{tikzpicture}[scale=0.65]
            \def\R{4.5}
            \definecolor{mygold}{RGB}{212, 175, 55}

            \foreach \i in {0,1,...,11} {
                \pgfmathsetmacro{\angM}{90 - \i*30}
                \pgfmathsetmacro{\angm}{90 - \i*30 - 15}
                \coordinate (cM\i) at (\angM:\R);
                \coordinate (cm\i) at (\angm:\R);
            }

            \foreach \i in {0,1,...,11} {
                \draw[gray, thin] (cM\i) -- (cm\i);
                \pgfmathsetmacro{\idxR}{int(mod(\i+10,12))}
                \draw[gray, thin] (cM\i) -- (cm\idxR);
                \pgfmathsetmacro{\idxL}{int(mod(\i+5,12))}
                \draw[thin, color=mygold] (cM\i) -- (cm\idxL);
            }

            \foreach \i in {0,1,...,11} {
                \pgfmathsetmacro{\isOdd}{int(mod(\i,2))}
                \ifnum\isOdd=0
                    \node[mynode, draw=blue, fill=blue!30] at (cM\i) {$M_{\i}$};
                    \node[mynode, draw=blue, fill=blue!30] at (cm\i) {$m_{\i}$};
                \else
                    \node[mynode, draw=red, fill=red!30] at (cM\i) {$M_{\i}$};
                    \node[mynode, draw=red, fill=red!30] at (cm\i) {$m_{\i}$};
                \fi
            }
            \node at (0, -\R-0.8) {\small (a) Standard Sequence};
        \end{tikzpicture}
    \end{minipage}%
    \hfill
    \begin{minipage}{0.48\textwidth}
        \centering
        \begin{tikzpicture}[scale=0.65]
            \def\R{4.5}
            \definecolor{mygold}{RGB}{212, 175, 55}

            \foreach \k/\pos in {0/0, 7/1, 2/2, 9/3, 4/4, 11/5, 6/6, 1/7, 8/8, 3/9, 10/10, 5/11} {
                \pgfmathsetmacro{\angM}{90 - \pos*30}
                \pgfmathsetmacro{\angm}{90 - \pos*30 - 15}
                \coordinate (cM\k) at (\angM:\R);
                \coordinate (cm\k) at (\angm:\R);
            }

            \foreach \i in {0,1,...,11} {
                \draw[gray, thin] (cM\i) -- (cm\i);
                \pgfmathsetmacro{\idxR}{int(mod(\i+10,12))}
                \draw[gray, thin] (cM\i) -- (cm\idxR);
                \pgfmathsetmacro{\idxL}{int(mod(\i+5,12))}
                \draw[thin, color=mygold] (cM\i) -- (cm\idxL);
            }

            \foreach \i in {0,1,...,11} {
                \pgfmathsetmacro{\isOdd}{int(mod(\i,2))}
                \ifnum\isOdd=0
                    \node[mynode, draw=blue, fill=blue!30] at (cM\i) {$M_{\i}$};
                    \node[mynode, draw=blue, fill=blue!30] at (cm\i) {$m_{\i}$};
                \else
                    \node[mynode, draw=red, fill=red!30] at (cM\i) {$M_{\i}$};
                    \node[mynode, draw=red, fill=red!30] at (cm\i) {$m_{\i}$};
                \fi
            }
            \node at (0, -\R-0.8) {\small (b) Hamiltonian Sequence};
        \end{tikzpicture}
    \end{minipage}
    
    \caption{The Connected Levi Graph visualized in two arrangements. (a) Standard sequential notation. (b) Hamiltonian perimeter arrangement following the \textbf{``Cycle of Fifths''}. The node size is reduced to reveal the delicate gold Wormholes (\texorpdfstring{$L$}{L}) connecting them.}
    \label{fig:levi_circular_delta3}
\end{figure}

\subsubsection{Minimal Chiral Cycles}
The introduction of wormholes significantly reduces the girth (shortest cycle length) of the graph. While the intra-universe graphs have a girth of 12, the restored cubic graph allows for tight loops of length 4.

\begin{theorem}
In the restored cubic system, there are exactly two isomorphism classes of minimal cycles (of length 4) under the rotational symmetry of \texorpdfstring{$\mathbb{Z}_{12}$}{Z12}. These classes are distinguished by their chirality (direction of winding through the wormhole).
\end{theorem}

\begin{proof}
A cycle must alternate between Major and Minor chords. The minimal length is 4. A 4-cycle must involve at least two wormhole crossings (\texorpdfstring{$L$}{L}) to leave and return to the starting universe.
The two minimal solutions correspond to the two directions of traversing the wormhole:

\begin{enumerate}
    \item \textbf{The Violet Cycle (Right-handed):} $ S = L_{+5} \cdot P \cdot L_{+7} \cdot R_{+2} $.\\
    Tracing the index change starting from $M_0$: 
    $M_0 \xrightarrow{L} m_5 \xrightarrow{P} M_5 \xrightarrow{L} m_{10} \xrightarrow{R} M_0$.
    The total displacement is $+5 + 0 + 5 + 2 = +12 \equiv 0 \pmod{12}$.
    
    \item \textbf{The Orange Cycle (Left-handed):} $ S' = R_{-2} \cdot L_{+7} \cdot P \cdot L_{+5} $. \\
    We trace the cycle $M_4 \to m_2 \to M_9 \to m_9 \to M_4$:
    \begin{itemize}
        \item $M_4 \xrightarrow{R} m_2$ (displacement $-2$)
        \item $m_2 \xrightarrow{L} M_9$ (displacement $+7$)
        \item $M_9 \xrightarrow{P} m_9$ (displacement $0$)
        \item $m_9 \xrightarrow{L} M_4$ (displacement $+7$. Note: $9+7=16 \equiv 4$)
    \end{itemize}
    Total displacement: $-2 + 7 + 0 + 7 = +12 \equiv 0 \pmod{12}$.
\end{enumerate}
These two sequences are non-isomorphic because they traverse the graph with opposite chirality.
\end{proof}

Figure \ref{fig:chiral_cycles_straight_delta3} illustrates these two fundamental cycles embedded within the Levi graph, depicted with straight edges and directional arrows to emphasize the geometric path.

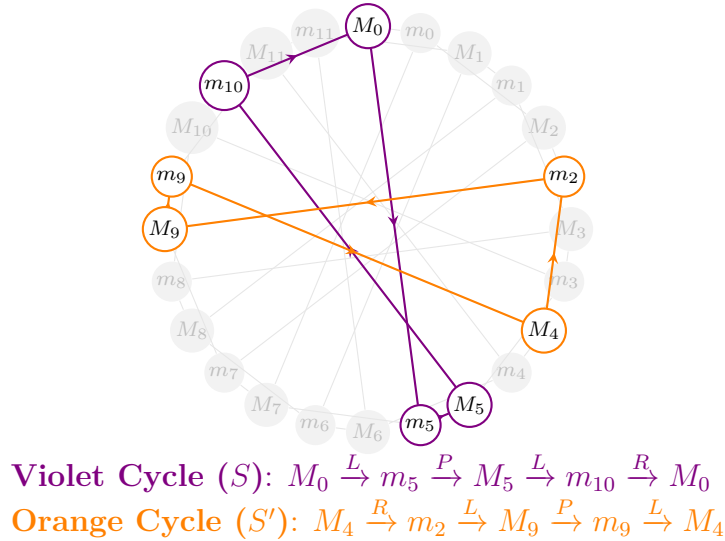
\begin{figure}[H]
    \centering
    \begin{tikzpicture}[
        scale=0.6, 
        >=stealth,
        midarrow/.style={
            postaction={decorate},
            decoration={
                markings,
                mark=at position 0.5 with {\arrow{stealth}}
            }
        },
        mynode/.style={circle, inner sep=1pt, fill=gray!10, text=gray!50, font=\tiny}
    ]
        \def\R{4.5}
        
        \foreach \i in {0,1,...,11} {
            \pgfmathsetmacro{\angM}{90 - \i*30}
            \pgfmathsetmacro{\angm}{90 - \i*30 - 15}
            \coordinate (M\i) at (\angM:\R);
            \coordinate (m\i) at (\angm:\R);
            
            \node[mynode] (N_M\i) at (M\i) {$M_{\i}$};
            \node[mynode] (N_m\i) at (m\i) {$m_{\i}$};
        }
        
        \foreach \i in {0,1,...,11} {
            \pgfmathsetmacro{\idxR}{int(mod(\i+10,12))}
            \pgfmathsetmacro{\idxL}{int(mod(\i+5,12))}
            
            \draw[gray!20, thin] (N_M\i) -- (N_m\i); 
            \draw[gray!20, thin] (N_M\i) -- (N_m\idxR); 
            \draw[gray!20, thin] (N_M\i) -- (N_m\idxL); 
        }

        \draw[violet, thick, midarrow] (N_M0) -- (N_m5); 
        \draw[violet, thick, midarrow] (N_m5) -- (N_M5); 
        \draw[violet, thick, midarrow] (N_M5) -- (N_m10); 
        \draw[violet, thick, midarrow] (N_m10) -- (N_M0); 
        
        \node[circle, draw=violet, fill=white, inner sep=1pt, font=\tiny, thick] at (M0) {$M_0$};
        \node[circle, draw=violet, fill=white, inner sep=1pt, font=\tiny, thick] at (m5) {$m_5$};
        \node[circle, draw=violet, fill=white, inner sep=1pt, font=\tiny, thick] at (M5) {$M_5$};
        \node[circle, draw=violet, fill=white, inner sep=1pt, font=\tiny, thick] at (m10) {$m_{10}$};

        \draw[orange, thick, midarrow] (N_M4) -- (N_m2); 
        \draw[orange, thick, midarrow] (N_m2) -- (N_M9); 
        \draw[orange, thick, midarrow] (N_M9) -- (N_m9); 
        \draw[orange, thick, midarrow] (N_m9) -- (N_M4); 

        \node[circle, draw=orange, fill=white, inner sep=1pt, font=\tiny, thick] at (M4) {$M_4$};
        \node[circle, draw=orange, fill=white, inner sep=1pt, font=\tiny, thick] at (m2) {$m_2$};
        \node[circle, draw=orange, fill=white, inner sep=1pt, font=\tiny, thick] at (M9) {$M_9$};
        \node[circle, draw=orange, fill=white, inner sep=1pt, font=\tiny, thick] at (m9) {$m_{9}$};

        \node[anchor=north] at (0, -\R-0.2) {
            \begin{tabular}{l}
            \textcolor{violet}{\textbf{Violet Cycle ($S$)}: $M_0 \xrightarrow{L} m_5 \xrightarrow{P} M_5 \xrightarrow{L} m_{10} \xrightarrow{R} M_0$} \\
            \textcolor{orange}{\textbf{Orange Cycle ($S' $)}: $M_4 \xrightarrow{R} m_2 \xrightarrow{L} M_{9} \xrightarrow{P} m_{9} \xrightarrow{L} M_4$}
            \end{tabular}
        };

    \end{tikzpicture}
    \caption{Visualization of the two chiral minimal cycles. Straight edges connect the nodes, with arrows placed at the midpoint of each segment to indicate direction. The Violet cycle starts at \texorpdfstring{$M_0$}{M0}, and the Orange cycle (its chiral opposite) is visualized starting at \texorpdfstring{$M_4$}{M4}.}
    \label{fig:chiral_cycles_straight_delta3}
\end{figure}
\subsection{Algebraic Characterization and Graph Invariants}
\label{sec:invariants_delta303}

To properly classify the functional 12-TET graph within the family of cubic graphs, we must first establish its fundamental algebraic invariants. By inspection of the graph constructed in the previous section, we identify the following properties:

\begin{enumerate}
    \item \textbf{Order and Valence:} The graph has $V=24$ vertices (12 Major, 12 Minor) and is $3$-regular (cubic).
    
    \item \textbf{Bipartiteness:} The graph is bipartite (Levi graph), with the two partitions corresponding to Major chords ($M$) and Minor chords ($m$). Every edge connects a vertex from the Major set to a vertex from the Minor set.
    
    \item \textbf{Girth:} As proven in Theorem 3, the graph has a girth of $g=4$. It contains fundamental 4-cycles formed by the interaction of the three operations, such as:
    \[ M_0 \xrightarrow{L} m_5 \xrightarrow{P} M_5 \xrightarrow{L} m_{10} \xrightarrow{R} M_0 \]
    
    \textit{Remark:} The existence of these 4-cycles implies that in the point-line interpretation, two distinct ``lines'' (e.g., $M_0$ and $M_5$) intersect at two distinct ``points'' (e.g., $m_5$ and $m_{10}$). Consequently, two distinct points can determine more than one line. This violates the axioms of a linear configuration, distinguishing this graph from the rigid geometric structures of the standard system.
    
    \item \textbf{Hamiltonicity:} The graph is Hamiltonian. A complete tour of all 24 chords can be constructed by alternating the Wormhole ($L$) and Relative ($R$) operations, or by alternating Wormhole ($L$) and Parallel ($P$) operations.
\end{enumerate}

\subsubsection{Hamiltonian Connectivity: Two Paths}
The graph is Hamiltonian, meaning there exists a path that visits every chord exactly once before returning to the start. We identify two distinct Hamiltonian cycles generated by alternating operations, each revealing a different aspect of the system's symmetry.

\paragraph{The Perimeter Cycle ($L \cdot P$).}
The first cycle corresponds to the perimeter arrangement shown previously in Figure \ref{fig:levi_circular_delta3}(b). It is generated by alternating the Parallel ($P$) and Wormhole ($L$) transformations:
\[ C_{LP} = (L \cdot P)^{12} \]
Starting from $M_0$, the sequence proceeds: $M_0 \xrightarrow{P} m_0 \xrightarrow{L} M_7 \xrightarrow{P} m_7 \xrightarrow{L} M_2 \dots$. This path traces the local inversion ($P$) followed by a global shift ($L$), naturally unfolding the graph's maximal symmetry (the Cycle of Fifths) into a ring.

\paragraph{The Grand Tour ($L \cdot R$).}
We can construct a second, distinct Hamiltonian cycle by alternating the Wormhole ($L$) and Relative ($R$) operations:
\[ C_{LR} = (L \cdot R)^{12} \]
This path effectively transposes the chord by a Perfect Fourth (interval class 5) at each full step ($M \to M$):
\[ M_0 \xrightarrow{R} m_{10} \xrightarrow{L} M_5 \xrightarrow{R} m_3 \xrightarrow{L} M_{10} \dots \]
While the algebraic formula is as simple as the perimeter cycle, its geometric realization in the standard sequential layout is strikingly different. Instead of stepping between neighbors, the Grand Tour weaves a complex, star-like pattern, jumping across the circle to stitch the entire manifold together.

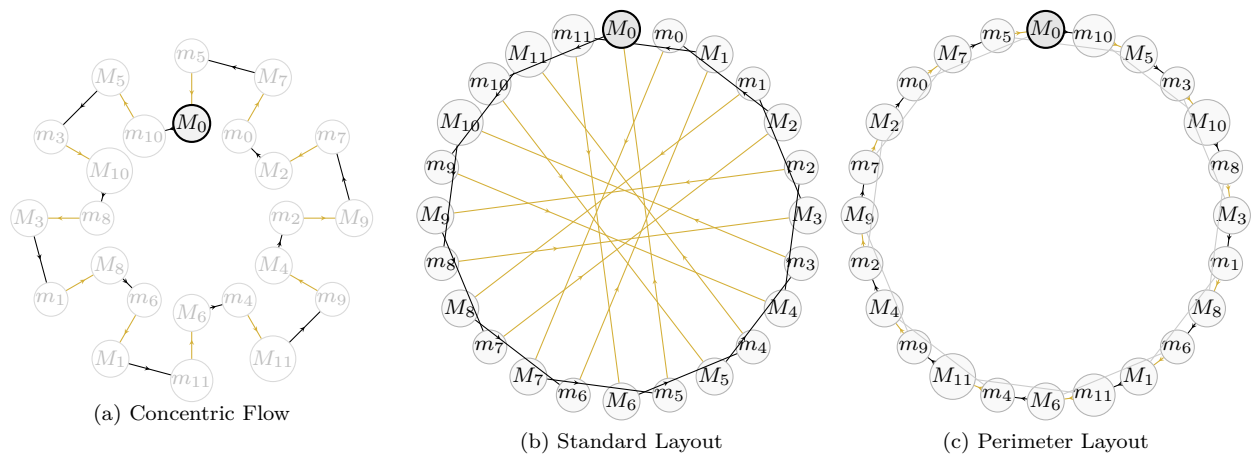
\begin{figure}[H]
    \centering
    \definecolor{mygold}{RGB}{212, 175, 55}

    \tikzset{quarterarrow/.style={postaction={decorate}, decoration={markings,mark=at position 0.7 with {\arrow[scale=0.6]{stealth}}}}}
    \tikzset{smallmidarrow/.style={postaction={decorate}, decoration={markings,mark=at position 0.5 with {\arrow[scale=0.6]{stealth}}}}}
    \tikzset{thirdarrow/.style={postaction={decorate}, decoration={markings,mark=at position 0.30 with {\arrow[scale=0.6]{stealth}}}}}
    \tikzset{mynode/.style={circle, inner sep=0pt, minimum size=0.45cm, font=\tiny}}

    \begin{minipage}{0.32\textwidth}
        \centering
        \begin{tikzpicture}[scale=0.45, >=stealth]
            \def\Rout{4.8} \def\Rin{2.8} \def\step{30}
            \foreach \i/\lab [count=\xi] in {0/m_{5}, 1/M_{7}, 2/m_{7}, 3/M_{9}, 4/m_{9}, 5/M_{11}, 6/m_{11}, 7/M_{1}, 8/m_{1}, 9/M_{3}, 10/m_{3}, 11/M_{5}} {
                \pgfmathsetmacro{\angle}{90 - \xi*\step + \step}
                \node[mynode, draw=gray!30, fill=white, text=gray!50, alias=RN_\lab] (RedNode\i) at (\angle:\Rout) {$\lab$};
            }
            \foreach \j/\lab [count=\xj] in {0/M_{0}, 1/m_{0}, 2/M_{2}, 3/m_{2}, 4/M_{4}, 5/m_{4}, 6/M_{6}, 7/m_{6}, 8/M_{8}, 9/m_{8}, 10/M_{10}, 11/m_{10}} {
                \pgfmathsetmacro{\angle}{90 - \xj*\step + \step}
                \node[mynode, draw=gray!30, fill=white, text=gray!50, alias=BN_\lab] (BlueNode\j) at (\angle:\Rin) {$\lab$};
            }
            
            
            \foreach \start/\end/\type in {
                BN_M_{0}/BN_m_{10}/R,   BN_m_{10}/RN_M_{5}/L, 
                RN_M_{5}/RN_m_{3}/R,    RN_m_{3}/BN_M_{10}/L,
                BN_M_{10}/BN_m_{8}/R,   BN_m_{8}/RN_M_{3}/L,
                RN_M_{3}/RN_m_{1}/R,    RN_m_{1}/BN_M_{8}/L,
                BN_M_{8}/BN_m_{6}/R,    BN_m_{6}/RN_M_{1}/L,
                RN_M_{1}/RN_m_{11}/R,   RN_m_{11}/BN_M_{6}/L,
                BN_M_{6}/BN_m_{4}/R,    BN_m_{4}/RN_M_{11}/L,
                RN_M_{11}/RN_m_{9}/R,   RN_m_{9}/BN_M_{4}/L,
                BN_M_{4}/BN_m_{2}/R,    BN_m_{2}/RN_M_{9}/L,
                RN_M_{9}/RN_m_{7}/R,    RN_m_{7}/BN_M_{2}/L,
                BN_M_{2}/BN_m_{0}/R,    BN_m_{0}/RN_M_{7}/L,
                RN_M_{7}/RN_m_{5}/R,    RN_m_{5}/BN_M_{0}/L
            } {
                \ifx\type\undefined\else
                    \if\type R
                        \draw[black, thin, smallmidarrow] (\start) -- (\end);
                    \else
                        \draw[thin, color=mygold, quarterarrow] (\start) -- (\end);
                    \fi
                \fi
            }
            
            \node[mynode, draw=black, fill=gray!10, thick] at (90:\Rin) {$M_0$};
            \node[align=center, font=\tiny] at (0, -\Rout-1) {(a) Concentric Flow};
        \end{tikzpicture}
    \end{minipage}
    \hfill
    \begin{minipage}{0.32\textwidth}
        \centering
        \begin{tikzpicture}[scale=0.55, >=stealth]
            \def\R{4.5} \def\step{30}
            \foreach \i in {0,...,11} {
                \pgfmathsetmacro{\angM}{90 - \i*\step}
                \pgfmathsetmacro{\angm}{90 - \i*\step - 15}
                \node[mynode, draw=gray!60, fill=gray!5] (M_\i) at (\angM:\R) {$M_{\i}$};
                \node[mynode, draw=gray!60, fill=gray!5] (m_\i) at (\angm:\R) {$m_{\i}$};
            }
            
            \foreach \u/\v/\type in {
                M_0/m_10/R, m_10/M_5/L, 
                M_5/m_3/R,  m_3/M_10/L,
                M_10/m_8/R, m_8/M_3/L, 
                M_3/m_1/R,  m_1/M_8/L,
                M_8/m_6/R,  m_6/M_1/L, 
                M_1/m_11/R, m_11/M_6/L,
                M_6/m_4/R,  m_4/M_11/L, 
                M_11/m_9/R, m_9/M_4/L,
                M_4/m_2/R,  m_2/M_9/L, 
                M_9/m_7/R,  m_7/M_2/L,
                M_2/m_0/R,  m_0/M_7/L, 
                M_7/m_5/R,  m_5/M_0/L
            } {
                \if\type R
                    \draw[black, thin, thirdarrow] (\u) -- (\v);
                \else
                    \draw[thin, color=mygold, thirdarrow] (\u) -- (\v);
                \fi
            }
            
            \node[mynode, draw=black, fill=gray!20, thick] at (90:\R) {$M_0$};
            \node[align=center, font=\tiny] at (0, -\R-1) {(b) Standard Layout};
        \end{tikzpicture}
    \end{minipage}
    \hfill
    \begin{minipage}{0.32\textwidth}
        \centering
        \begin{tikzpicture}[scale=0.55, >=stealth]
            \def\R{4.5}
            
            \foreach \id/\lab [count=\i from 0] in {
                M_0/M_{0}, m_10/m_{10}, M_5/M_{5}, m_3/m_{3}, 
                M_10/M_{10}, m_8/m_{8}, M_3/M_{3}, m_1/m_{1}, 
                M_8/M_{8}, m_6/m_{6}, M_1/M_{1}, m_11/m_{11}, 
                M_6/M_{6}, m_4/m_{4}, M_11/M_{11}, m_9/m_{9}, 
                M_4/M_{4}, m_2/m_{2}, M_9/M_{9}, m_7/m_{7}, 
                M_2/M_{2}, m_0/m_{0}, M_7/M_{7}, m_5/m_{5}} 
            {
                \pgfmathsetmacro{\angle}{90 - \i * (360/24)}
                \node[mynode, draw=gray!60, fill=gray!5, alias=P_\id] (Node\i) at (\angle:\R) {$\lab$};
            }
            
            \foreach \k in {0,...,11} {
                \draw[gray!40, thin] (P_M_\k) -- (P_m_\k);
            }

            \foreach \i in {0,...,23} {
                \pgfmathsetmacro{\next}{int(mod(\i+1,24))}
                \pgfmathsetmacro{\isR}{int(mod(\i,2))} 
                \ifnum\isR=0 
                    \draw[black, thin, smallmidarrow] (Node\i) -- (Node\next);
                \else 
                    \draw[thin, color=mygold, smallmidarrow] (Node\i) -- (Node\next);
                \fi
            }
            \node[mynode, draw=black, fill=gray!20, thick] at (90:\R) {$M_0$};
            \node[align=center, font=\tiny] at (0, -\R-1) {(c) Perimeter Layout};
        \end{tikzpicture}
    \end{minipage}

    \caption{The ``Grand Tour'' Hamiltonian Cycle $(L \cdot R)^{12}$. (a) Concentric flow zig-zagging between Blue and Red rings. (b) Star-like pattern in standard layout. (c) Perimeter layout: the nodes are re-indexed such that the $(L \cdot R)^{12}$ Hamiltonian cycle forms the outer boundary. The remaining $P$ operations appear as internal chords of constant length, forming a spirograph pattern.}
    \label{fig:hamiltonian_combined_delta3}
\end{figure}

\section{The ``Far Narrow Thirds'' System: \texorpdfstring{$\Delta=1, (t,s)=(10,9)$}{Delta=1, (t,s)=(10,9)}}
\label{sec:far_thirds}

We now turn our attention to the second solution of the Diophantine equation $t+s=7$ for the case $\Delta=1$. This is the partition $(t,s)=(10,9)$, which we designate as the \textbf{"Far Narrow Thirds"} system.

Here, the generator $q=7$ is split into a "Major Third" of 10 steps (a Minor Seventh in standard terminology) and a "Minor Third" of 9 steps (a Major Sixth). While the difference $\Delta = t-s = 1$ is identical to the standard Western system, the resulting topology is dramatically different due to the magnitude of the intervals.

\subsection{Chord Definitions}
The Major chord $M_k$ and Minor chord $m_k$ are defined by the standard formulas:
\begin{align}
    M_k &= \{k, k+t, k+q\} = \{k, k+10, k+7\}, \\
    m_k &= \{k, k+s, k+q\} = \{k, k+9, k+7\}.
\end{align}
Unlike the $\Delta=3$ case, there is no modal degeneracy here. The Major chord $M_0 = \{0, 10, 7\}$ and the Minor chord $m_0 = \{0, 9, 7\}$ are distinct pitch sets that cannot be mapped to one another by translation; they are mirror images (chiral copies) of one another.

\subsection{The Tonnetz and Voice Leading}
The topology of the system is determined by the common-tone connections between chords. We designate these transformations using the standard $P, L, R$ nomenclature. 

For the system $(t,s)=(10,9)$, the transformation rules are derived from the shared intervals $q$ (for $P$), $s$ (for $L$), and $t$ (for $R$). Note the specific indices for $L$ and $R$ shifts in this system:
\begin{align*}
    P(M_k) &= m_k & P(m_k) &= M_k \\
    L(M_k) &= m_{k+10} & L(m_k) &= M_{k+2} \\
    R(M_k) &= m_{k+3} & R(m_k) &= M_{k+9}
\end{align*}

The complete set of voice-leading connections is provided in Tables \ref{tab:far_thirds_major} and \ref{tab:far_thirds_minor}.

\begin{table}[H]
\centering
\scriptsize
\renewcommand{\arraystretch}{1.2}
\begin{minipage}{0.48\textwidth}
\centering
\begin{tabular}{|c||c|c|c|}
\hline
\textbf{Source} ($M_k$) & \textbf{P} ($m_k$) & \textbf{L} ($m_{k+10}$) & \textbf{R} ($m_{k+3}$) \\
\hline
$M_0$ & $m_0$ & $m_{10}$ & $m_3$ \\
$M_1$ & $m_1$ & $m_{11}$ & $m_4$ \\
$M_2$ & $m_2$ & $m_{0}$ & $m_5$ \\
$M_3$ & $m_3$ & $m_{1}$ & $m_6$ \\
$M_4$ & $m_4$ & $m_{2}$ & $m_7$ \\
$M_5$ & $m_5$ & $m_{3}$ & $m_8$ \\
$M_6$ & $m_6$ & $m_{4}$ & $m_9$ \\
$M_7$ & $m_7$ & $m_{5}$ & $m_{10}$ \\
$M_8$ & $m_8$ & $m_{6}$ & $m_{11}$ \\
$M_9$ & $m_9$ & $m_{7}$ & $m_{0}$ \\
$M_{10}$ & $m_{10}$ & $m_{8}$ & $m_{1}$ \\
$M_{11}$ & $m_{11}$ & $m_{9}$ & $m_{2}$ \\
\hline
\end{tabular}
\caption{Voice Leading: Major Chords $M_k$.}
\label{tab:far_thirds_major}
\end{minipage}
\hfill
\begin{minipage}{0.48\textwidth}
\centering
\begin{tabular}{|c||c|c|c|}
\hline
\textbf{Source} ($m_k$) & \textbf{P} ($M_k$) & \textbf{L} ($M_{k+2}$) & \textbf{R} ($M_{k+9}$) \\
\hline
$m_0$ & $M_0$ & $M_2$ & $M_9$ \\
$m_1$ & $M_1$ & $M_3$ & $M_{10}$ \\
$m_2$ & $M_2$ & $M_4$ & $M_{11}$ \\
$m_3$ & $M_3$ & $M_5$ & $M_{0}$ \\
$m_4$ & $M_4$ & $M_6$ & $M_{1}$ \\
$m_5$ & $M_5$ & $M_7$ & $M_{2}$ \\
$m_6$ & $M_6$ & $M_8$ & $M_{3}$ \\
$m_7$ & $M_7$ & $M_9$ & $M_{4}$ \\
$m_8$ & $M_8$ & $M_{10}$ & $M_{5}$ \\
$m_9$ & $M_9$ & $M_{11}$ & $M_{6}$ \\
$m_{10}$ & $M_{10}$ & $M_{0}$ & $M_{7}$ \\
$m_{11}$ & $M_{11}$ & $M_{1}$ & $M_{8}$ \\
\hline
\end{tabular}
\caption{Voice Leading: Minor Chords $m_k$.}
\label{tab:far_thirds_minor}
\end{minipage}
\end{table}

\subsection{Topological Properties}
The Levi graph of this system is a connected cubic graph on 24 vertices. Its structure is defined by the interaction of the three fundamental involutions.

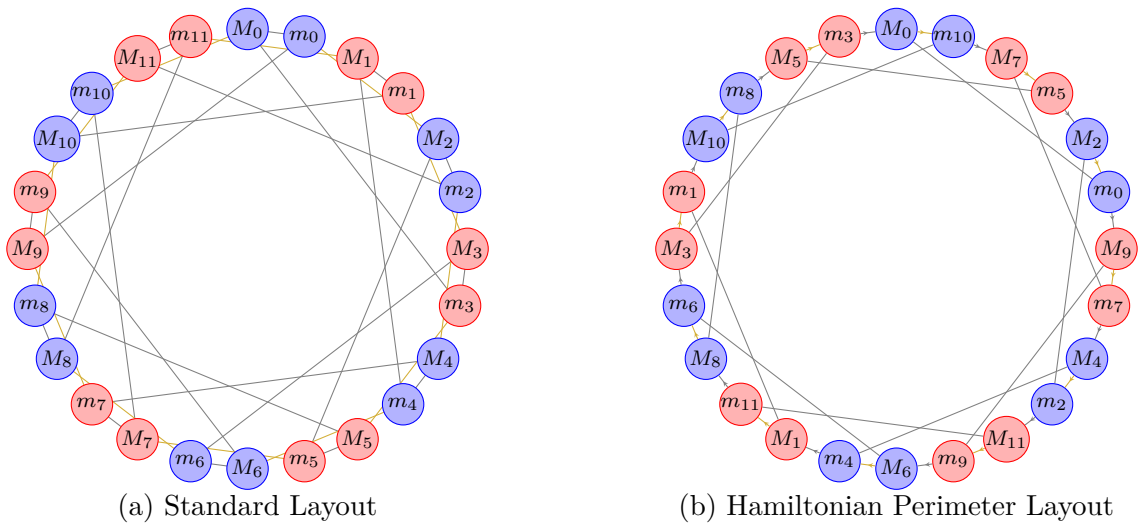
\begin{figure}[H]
    \centering
    \tikzset{
        mynode/.style={circle, inner sep=0pt, minimum size=0.55cm, font=\tiny}
    }
    \tikzset{smallmidarrow/.style={postaction={decorate}, decoration={markings,mark=at position 0.5 with {\arrow[scale=0.6]{stealth}}}}}

    \begin{minipage}{0.48\textwidth}
        \centering
        \begin{tikzpicture}[scale=0.65]
            \def\R{4.5}
            \definecolor{mygold}{RGB}{212, 175, 55}

            \foreach \i in {0,1,...,11} {
                \pgfmathsetmacro{\angM}{90 - \i*30}
                \pgfmathsetmacro{\angm}{90 - \i*30 - 15}
                \coordinate (cM\i) at (\angM:\R);
                \coordinate (cm\i) at (\angm:\R);
            }

            \foreach \i in {0,1,...,11} {
                \draw[gray, thin] (cM\i) -- (cm\i);
                \pgfmathsetmacro{\idxR}{int(mod(\i+3,12))}
                \draw[gray, thin] (cM\i) -- (cm\idxR);
                \pgfmathsetmacro{\idxL}{int(mod(\i+10,12))}
                \draw[thin, color=mygold] (cM\i) -- (cm\idxL);
            }

            \foreach \i in {0,1,...,11} {
                \pgfmathsetmacro{\isOdd}{int(mod(\i,2))}
                \ifnum\isOdd=0
                    \node[mynode, draw=blue, fill=blue!30] at (cM\i) {$M_{\i}$};
                    \node[mynode, draw=blue, fill=blue!30] at (cm\i) {$m_{\i}$};
                \else
                    \node[mynode, draw=red, fill=red!30] at (cM\i) {$M_{\i}$};
                    \node[mynode, draw=red, fill=red!30] at (cm\i) {$m_{\i}$};
                \fi
            }
            \node at (0, -\R-0.8) {\small (a) Standard Layout};
        \end{tikzpicture}
    \end{minipage}%
    \hfill
    \begin{minipage}{0.48\textwidth}
        \centering
        \begin{tikzpicture}[scale=0.65, >=stealth]
            \def\R{4.5}
            \definecolor{mygold}{RGB}{212, 175, 55}
            
            \foreach \k/\id/\lab [count=\i from 0] in {
                0/M_0/M_{0}, 10/m_10/m_{10}, 7/M_7/M_{7}, 5/m_5/m_{5}, 
                2/M_2/M_{2}, 0/m_0/m_{0}, 9/M_9/M_{9}, 7/m_7/m_{7}, 
                4/M_4/M_{4}, 2/m_2/m_{2}, 11/M_11/M_{11}, 9/m_9/m_{9}, 
                6/M_6/M_{6}, 4/m_4/m_{4}, 1/M_1/M_{1}, 11/m_11/m_{11}, 
                8/M_8/M_{8}, 6/m_6/m_{6}, 3/M_3/M_{3}, 1/m_1/m_{1}, 
                10/M_10/M_{10}, 8/m_8/m_{8}, 5/M_5/M_{5}, 3/m_3/m_{3}} 
            {
                \pgfmathsetmacro{\angle}{90 - \i * (360/24)}
                
                \pgfmathsetmacro{\isOdd}{int(mod(\k,2))}
                \ifnum\isOdd=0
                    \node[mynode, draw=blue, fill=blue!30, alias=P_\id] (Node\i) at (\angle:\R) {$\lab$};
                \else
                    \node[mynode, draw=red, fill=red!30, alias=P_\id] (Node\i) at (\angle:\R) {$\lab$};
                \fi
            }
            
            \foreach \k in {0,...,11} {
                \draw[gray, thin] (P_M_\k) -- (P_m_\k);
            }

            \foreach \i in {0,...,23} {
                \pgfmathsetmacro{\next}{int(mod(\i+1,24))}
                \pgfmathsetmacro{\isR}{int(mod(\i,2))} 
                \ifnum\isR=0 
                    \draw[color=mygold, thin, smallmidarrow] (Node\i) -- (Node\next);
                \else 
                    \draw[gray, thin, smallmidarrow] (Node\i) -- (Node\next);
                \fi
            }
            
            \node at (0, -\R-0.8) {\small (b) Hamiltonian Perimeter Layout};
        \end{tikzpicture}
    \end{minipage}
    
    \caption{The Connected Levi Graph for $(t,s)=(10,9)$. (a) Standard layout with radial pairing. (b) Hamiltonian perimeter layout with mirrored coloring. Even nodes are Blue, Odd nodes are Red. The perimeter alternates between $L$-edges (Gold) and $R$-edges (Gray), while the internal chords represent $P$-connections (Gray).}
    \label{fig:levi_circular_delta10_9}
\end{figure}

\subsubsection{Minimal Circles and Girth}
An investigation of the cyclic paths in this graph reveals the fundamental "shapes" of the harmonic space.

\begin{itemize}
    \item \textbf{The $PL$-cycle:} The sequence $M_k \xrightarrow{P} m_k \xrightarrow{L} M_{k+2}$ induces a shift of $+2$. Since $\gcd(2, 12)=2$, this generates two disjoint cycles of length 12 (dodecagons).
    \item \textbf{The $PR$-cycle:} The sequence $M_k \xrightarrow{P} m_k \xrightarrow{R} M_{k+9}$ induces a shift of $+9 \equiv -3$. Since $12/\gcd(3,12)=4$, this generates three disjoint cycles of length 8 (octagons).
    \item \textbf{The Minimal Cycle:} The most compact loop in the system is neither of these. It is a "twisted" ternary cycle formed by the sequence $(P \cdot L \cdot R)^2$.
\end{itemize}

\begin{theorem}[Girth of the Far Narrow Thirds System]
The Levi graph of the $(10,9)$ system has girth 6. It is tiled by minimal hexagonal cycles generated by the composite transformation sequence $P \to R \to L \to P \to R \to L$.
\end{theorem}

\begin{proof}
Let us trace the path starting from an arbitrary Major chord $M_k$ under the sequence of transformations $S = P \circ R \circ L \circ P \circ R \circ L$. Note that we apply the operators from left to right as voice-leading steps:
\begin{enumerate}
    \item $M_k \xrightarrow{P} m_k$ (Shift: $0$)
    \item $m_k \xrightarrow{R} M_{k+9}$ (Shift: $+9$)
    \item $M_{k+9} \xrightarrow{L} m_{(k+9)+10} = m_{k+7}$ (Shift: $+10$)
    \item $m_{k+7} \xrightarrow{P} M_{k+7}$ (Shift: $0$)
    \item $M_{k+7} \xrightarrow{R} m_{(k+7)+3} = m_{k+10}$ (Shift: $+3$)
    \item $m_{k+10} \xrightarrow{L} M_{(k+10)+2} = M_{k+12} \equiv M_k$ (Shift: $+2$)
\end{enumerate}
The total index shift is $0 + 9 + 10 + 0 + 3 + 2 = 24 \equiv 0 \pmod{12}$. The path returns to the starting chord $M_k$ after exactly 6 distinct edges. Since the graph is bipartite, odd cycles are impossible, and since no 4-cycles (digons) exist (as shown by the larger lengths of binary cycles $PL$ and $PR$), the minimal cycle length is 6.
\end{proof}

Since the girth is 6, the graph contains no 4-cycles. In the language of incidence geometry, this means there are no two points connected by more than one line. Thus, the $(10,9)$ harmonic system forms a valid \textbf{geometric configuration of 12 points and 12 lines} ($12_3$).

\begin{figure}[H]
    \centering
    \begin{tikzpicture}[
        scale=0.6, 
        >=stealth,
        midarrow/.style={
            postaction={decorate},
            decoration={
                markings,
                mark=at position 0.5 with {\arrow{stealth}}
            }
        },
        mynode/.style={circle, inner sep=1pt, fill=gray!10, text=gray!50, font=\tiny}
    ]
        \def\R{4.5}
        
        \foreach \i in {0,1,...,11} {
            \pgfmathsetmacro{\angM}{90 - \i*30}
            \pgfmathsetmacro{\angm}{90 - \i*30 - 15}
            \coordinate (M\i) at (\angM:\R);
            \coordinate (m\i) at (\angm:\R);
            
            \node[mynode] (N_M\i) at (M\i) {$M_{\i}$};
            \node[mynode] (N_m\i) at (m\i) {$m_{\i}$};
        }
        
        \foreach \i in {0,1,...,11} {
            \pgfmathsetmacro{\idxL}{int(mod(\i+10,12))}
            \pgfmathsetmacro{\idxR}{int(mod(\i+3,12))}
            \draw[gray!20, thin] (N_M\i) -- (N_m\i); 
            \draw[gray!20, thin] (N_M\i) -- (N_m\idxR); 
            \draw[gray!20, thin] (N_M\i) -- (N_m\idxL); 
        }

        \draw[violet, thick, midarrow] (N_m7) -- (N_M7); 
        \draw[violet, thick, midarrow] (N_M7) -- (N_m5); 
        \draw[violet, thick, midarrow] (N_m5) -- (N_M2); 
        \draw[violet, thick, midarrow] (N_M2) -- (N_m2); 
        \draw[violet, thick, midarrow] (N_m2) -- (N_M4); 
        \draw[violet, thick, midarrow] (N_M4) -- (N_m7); 
        
        \foreach \n/\lab in {m7/m_{7}, M7/M_{7}, m5/m_{5}, M2/M_{2}, m2/m_{2}, M4/M_{4}} {
             \node[circle, draw=violet, fill=white, inner sep=1pt, font=\tiny, thick] at (N_\n) {$\lab$};
        }

        \draw[orange, thick, midarrow] (N_M1) -- (N_m1); 
        \draw[orange, thick, midarrow] (N_m1) -- (N_M10); 
        \draw[orange, thick, midarrow] (N_M10) -- (N_m8); 
        \draw[orange, thick, midarrow] (N_m8) -- (N_M8); 
        \draw[orange, thick, midarrow] (N_M8) -- (N_m11); 
        \draw[orange, thick, midarrow] (N_m11) -- (N_M1); 

        \foreach \n/\lab in {M1/M_{1}, m1/m_{1}, M10/M_{10}, m8/m_{8}, M8/M_{8}, m11/m_{11}} {
             \node[circle, draw=orange, fill=white, inner sep=1pt, font=\tiny, thick] at (N_\n) {$\lab$};
        }

        \node[anchor=north] at (0, -\R-0.2) {
            \begin{tabular}{l}
            \textcolor{violet}{\textbf{Violet Hexagon}: The minimal cycle starting at $m_7$} \\
            \textcolor{orange}{\textbf{Orange Hexagon}: The disjoint minimal cycle starting at $M_1$}
            \end{tabular}
        };

    \end{tikzpicture}
    \caption{Visualization of two disjoint minimal cycles in the Far Narrow Thirds system. The true girth of the system is 6, realized by "twisted hexagons" involving all three operations $P, R, L$.}
    \label{fig:minimal_hexagons_delta10_9}
\end{figure}

\subsection{Algebraic characterization and graph invariants}
To properly classify the functional graph of the "Far Narrow Thirds" system within the family of cubic graphs, we must first establish its fundamental algebraic invariants. By inspection of the graph constructed in the previous section, we identify the following properties:

\begin{enumerate}
    \item \textbf{Order and Valence:} The graph has $V=24$ vertices (12 Major, 12 Minor) and is $3$-regular (cubic).
    
    \item \textbf{Bipartiteness:} The graph is bipartite (Levi graph), with the two partitions corresponding to Major chords ($M$) and Minor chords ($m$). Every edge connects a vertex from the Major set to a vertex from the Minor set.
    
    \item \textbf{Girth:} As proven in Theorem 3, the graph has a girth of $g=6$. It contains fundamental 6-cycles formed by the ternary interaction of operations ($P, L, R$). The absence of 4-cycles implies that this graph forms a valid $12_3$ geometric configuration.
    
    \item \textbf{Hamiltonicity:} The graph is Hamiltonian. A complete tour of all 24 chords can be constructed by alternating the Wormhole ($L$) and Relative ($R$) operations.
\end{enumerate}

\subsubsection{Hamiltonian Connectivity: The Two Paths}
While the local binary cycles ($PL$ and $PR$) fracture the universe into disjoint rings (dodecagons and octagons), the graph possesses a global connectivity generated by the interaction of $L$ and $R$. We identify two distinct directed Hamiltonian paths that traverse the entire harmonic manifold.

\paragraph{The Grand Tour (Circle of Fifths).}
The primary Hamiltonian cycle is generated by alternating the Wormhole ($L$) and Relative ($R$) transformations:
\[ C_{LR} = (L \cdot R)^{12} \]
Starting from $M_0$, the sequence proceeds: $M_0 \xrightarrow{L} m_{10} \xrightarrow{R} M_7 \xrightarrow{L} m_5 \xrightarrow{R} M_2 \dots$.
This path induces a root shift of $+7$ at each full step ($M \to M$). Geometrically, it stitches the disconnected local neighborhoods together by jumping strictly along the interval of the Perfect Fifth.

\paragraph{The Reverse Tour (Circle of Fourths).}
The chiral opposite of the Grand Tour is generated by reversing the alternation order (or traversing the cycle backwards):
\[ C_{RL} = (R \cdot L)^{12} \]
This path induces a root shift of $-7 \equiv +5$ (the Perfect Fourth). While it traces the same set of edges as the Grand Tour, it represents the distinct harmonic experience of "descending" through the cycle of fifths.

\textit{Remark on the Perimeter Cycle:} It is important to note a key distinction from the "Wide Thirds" ($\Delta=3$) system. In that system, the Perimeter Cycle ($L \cdot P$) is Hamiltonian. In the present $(10,9)$ system, the $(L \cdot P)$ sequence induces a root shift of $+2$, which (since $\gcd(2,12)=2$) decomposes the graph into two disjoint, non-communicating dodecagonal rings. Thus, in the "Far Narrow Thirds" universe, only the $L/R$ interaction provides a true Hamiltonian wormhole.

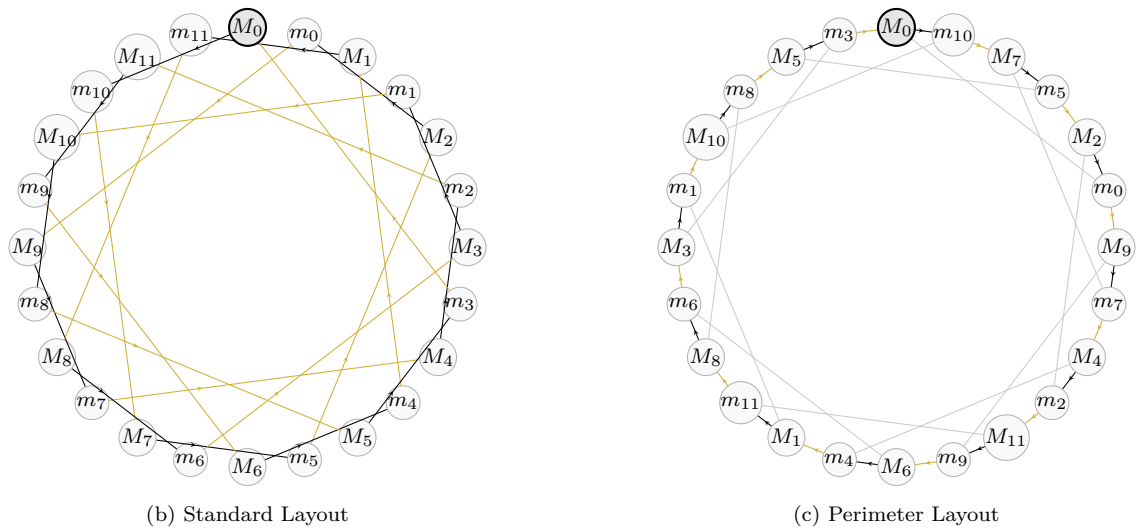
\begin{figure}[H]
    \centering
    \definecolor{mygold}{RGB}{212, 175, 55}

    \tikzset{quarterarrow/.style={postaction={decorate}, decoration={markings,mark=at position 0.7 with {\arrow[scale=0.6]{stealth}}}}}
    \tikzset{smallmidarrow/.style={postaction={decorate}, decoration={markings,mark=at position 0.5 with {\arrow[scale=0.6]{stealth}}}}}
    \tikzset{thirdarrow/.style={postaction={decorate}, decoration={markings,mark=at position 0.30 with {\arrow[scale=0.6]{stealth}}}}}
    \tikzset{mynode/.style={circle, inner sep=0pt, minimum size=0.45cm, font=\tiny}}
    \begin{minipage}{0.48\textwidth}
        \centering
        \begin{tikzpicture}[scale=0.65, >=stealth]
            \def\R{4.5} \def\step{30}
            \foreach \i in {0,...,11} {
                \pgfmathsetmacro{\angM}{90 - \i*\step}
                \pgfmathsetmacro{\angm}{90 - \i*\step - 15}
                \node[mynode, draw=gray!60, fill=gray!5] (M_\i) at (\angM:\R) {$M_{\i}$};
                \node[mynode, draw=gray!60, fill=gray!5] (m_\i) at (\angm:\R) {$m_{\i}$};
            }
            
            \foreach \u/\v/\type in {
                M_0/m_10/L, m_10/M_7/R, 
                M_7/m_5/L,  m_5/M_2/R, 
                M_2/m_0/L,  m_0/M_9/R, 
                M_9/m_7/L,  m_7/M_4/R, 
                M_4/m_2/L,  m_2/M_11/R, 
                M_11/m_9/L, m_9/M_6/R, 
                M_6/m_4/L,  m_4/M_1/R, 
                M_1/m_11/L, m_11/M_8/R,
                M_8/m_6/L,  m_6/M_3/R,
                M_3/m_1/L,  m_1/M_10/R,
                M_10/m_8/L, m_8/M_5/R,
                M_5/m_3/L,  m_3/M_0/R
            } {
                \if\type L
                    \draw[black, thin, thirdarrow] (\u) -- (\v);
                \else
                    \draw[thin, color=mygold, thirdarrow] (\u) -- (\v);
                \fi
            }
            
            \node[mynode, draw=black, fill=gray!20, thick] at (90:\R) {$M_0$};
            \node[align=center, font=\tiny] at (0, -\R-1) {(b) Standard Layout};
        \end{tikzpicture}
    \end{minipage}
    \hfill
    \begin{minipage}{0.48\textwidth}
        \centering
        \begin{tikzpicture}[scale=0.65, >=stealth]
            \def\R{4.5}
            
            \foreach \id/\lab [count=\i from 0] in {
                M_0/M_{0}, m_10/m_{10}, M_7/M_{7}, m_5/m_{5}, 
                M_2/M_{2}, m_0/m_{0}, M_9/M_{9}, m_7/m_{7}, 
                M_4/M_{4}, m_2/m_{2}, M_11/M_{11}, m_9/m_{9}, 
                M_6/M_{6}, m_4/m_{4}, M_1/M_{1}, m_11/m_{11}, 
                M_8/M_{8}, m_6/m_{6}, M_3/M_{3}, m_1/m_{1}, 
                M_10/M_{10}, m_8/m_{8}, M_5/M_{5}, m_3/m_{3}} 
            {
                \pgfmathsetmacro{\angle}{90 - \i * (360/24)}
                \node[mynode, draw=gray!60, fill=gray!5, alias=P_\id] (Node\i) at (\angle:\R) {$\lab$};
            }
            
            \foreach \k in {0,...,11} {
                \draw[gray!40, thin] (P_M_\k) -- (P_m_\k);
            }

            \foreach \i in {0,...,23} {
                \pgfmathsetmacro{\next}{int(mod(\i+1,24))}
                \pgfmathsetmacro{\isR}{int(mod(\i,2))} 
                \ifnum\isR=0 
                    \draw[black, thin, smallmidarrow] (Node\i) -- (Node\next);
                \else 
                    \draw[thin, color=mygold, smallmidarrow] (Node\i) -- (Node\next);
                \fi
            }
            \node[mynode, draw=black, fill=gray!20, thick] at (90:\R) {$M_0$};
            \node[align=center, font=\tiny] at (0, -\R-1) {(c) Perimeter Layout};
        \end{tikzpicture}
    \end{minipage}

    \caption{The ``Grand Tour'' Hamiltonian Cycle $(L \cdot R)^{12}$ for the Far Narrow Thirds system. (b) In the standard layout, the cycle forms a spirograph pattern. (c) In the perimeter layout (ordered by the Circle of Fifths), the Hamiltonian path forms the outer boundary, while the $P$ relations form internal chords.}
    \label{fig:hamiltonian_combined_delta10_9}
\end{figure}

\section{The ``Far Wide Thirds'' System: \texorpdfstring{$\Delta=3, (t,s)=(11,8)$}{Delta=3, (t,s)=(11,8)}}
\label{sec:far_wide}

We conclude our exploration of the $\Delta=3$ family with the partition $(t,s)=(11,8)$, which we designate as the \textbf{"Far Wide Thirds"} system.

Here, the generator $q=7$ is split into a "Major Third" of 11 steps (a Major Seventh) and a "Minor Third" of 8 steps (a Minor Sixth). Although this system shares the parameter $\Delta=3$ with the "Wide Thirds" system of Section 5, it exhibits none of that system's pathologies. The large interval sizes prevent the overlap that caused modal degeneracy, resulting in a fully connected, non-degenerate harmonic universe.

\subsection{Chord Definitions}
The Major chord $M_k$ and Minor chord $m_k$ are defined by the standard formulas:
\begin{align}
    M_k &= \{k, k+t, k+q\} = \{k, k+11, k+7\}, \\
    m_k &= \{k, k+s, k+q\} = \{k, k+8, k+7\}.
\end{align}
Unlike the $(5,2)$ case, there is no modal degeneracy here. The Major and Minor chords are distinct pitch sets and form chiral pairs. The complete list of the 24 triads in this system is provided in Tables \ref{tab:far_wide_major_def} and \ref{tab:far_wide_minor_def}.

\begin{table}[H]
\centering
\scriptsize
\renewcommand{\arraystretch}{1.0}
\begin{minipage}{0.48\textwidth}
\centering
\begin{tabular}{|c|c|}
\hline
\textbf{Root} ($k$) & \textbf{Major Chord} $M_k = \{k, k+11, k+7\}$ \\
\hline
0 & \{0, 11, 7\} \\
1 & \{1, 0, 8\} \\
2 & \{2, 1, 9\} \\
3 & \{3, 2, 10\} \\
4 & \{4, 3, 11\} \\
5 & \{5, 4, 0\} \\
6 & \{6, 5, 1\} \\
7 & \{7, 6, 2\} \\
8 & \{8, 7, 3\} \\
9 & \{9, 8, 4\} \\
10 & \{10, 9, 5\} \\
11 & \{11, 10, 6\} \\
\hline
\end{tabular}
\caption{\scriptsize The 12 Major Triads of the Far Wide Thirds System.}
\label{tab:far_wide_major_def}
\end{minipage}
\hfill
\begin{minipage}{0.48\textwidth}
\centering
\begin{tabular}{|c|c|}
\hline
\textbf{Root} ($k$) & \textbf{Minor Chord} $m_k = \{k, k+8, k+7\}$ \\
\hline
0 & \{0, 8, 7\} \\
1 & \{1, 9, 8\} \\
2 & \{2, 10, 9\} \\
3 & \{3, 11, 10\} \\
4 & \{4, 0, 11\} \\
5 & \{5, 1, 0\} \\
6 & \{6, 2, 1\} \\
7 & \{7, 3, 2\} \\
8 & \{8, 4, 3\} \\
9 & \{9, 5, 4\} \\
10 & \{10, 6, 5\} \\
11 & \{11, 7, 6\} \\
\hline
\end{tabular}
\caption{\scriptsize The 12 Minor Triads of the Far Wide Thirds System.}
\label{tab:far_wide_minor_def}
\end{minipage}
\end{table}

\subsection{The Tonnetz and Voice Leading}
The topology is determined by the three fundamental transformations. Note the specific index shifts derived from the $(11,8)$ geometry:

\begin{align*}
    P(M_k) &= m_k & P(m_k) &= M_k \\
    L(M_k) &= m_{k+11} & L(m_k) &= M_{k+1} \\
    R(M_k) &= m_{k+4} & R(m_k) &= M_{k+8}
\end{align*}

The complete set of connections is provided in Tables \ref{tab:far_wide_major} and \ref{tab:far_wide_minor}.

\begin{table}[H]
\centering
\scriptsize
\renewcommand{\arraystretch}{1.2}
\begin{minipage}{0.48\textwidth}
\centering
\begin{tabular}{|c||c|c|c|}
\hline
\textbf{Source} ($M_k$) & \textbf{P} ($m_k$) & \textbf{L} ($m_{k+11}$) & \textbf{R} ($m_{k+4}$) \\
\hline
$M_0$ & $m_0$ & $m_{11}$ & $m_4$ \\
$M_1$ & $m_1$ & $m_{0}$ & $m_5$ \\
$M_2$ & $m_2$ & $m_{1}$ & $m_6$ \\
$M_3$ & $m_3$ & $m_{2}$ & $m_7$ \\
$M_4$ & $m_4$ & $m_{3}$ & $m_8$ \\
$M_5$ & $m_5$ & $m_{4}$ & $m_9$ \\
$M_6$ & $m_6$ & $m_{5}$ & $m_{10}$ \\
$M_7$ & $m_7$ & $m_{6}$ & $m_{11}$ \\
$M_8$ & $m_8$ & $m_{7}$ & $m_{0}$ \\
$M_9$ & $m_9$ & $m_{8}$ & $m_{1}$ \\
$M_{10}$ & $m_{10}$ & $m_{9}$ & $m_{2}$ \\
$M_{11}$ & $m_{11}$ & $m_{10}$ & $m_{3}$ \\
\hline
\end{tabular}
\caption{\scriptsize Voice Leading: Major Chords $M_k$.}
\label{tab:far_wide_major}
\end{minipage}
\hfill
\begin{minipage}{0.48\textwidth}
\centering
\begin{tabular}{|c||c|c|c|}
\hline
\textbf{Source} ($m_k$) & \textbf{P} ($M_k$) & \textbf{L} ($M_{k+1}$) & \textbf{R} ($M_{k+8}$) \\
\hline
$m_0$ & $M_0$ & $M_1$ & $M_8$ \\
$m_1$ & $M_1$ & $M_2$ & $M_9$ \\
$m_2$ & $M_2$ & $M_3$ & $M_{10}$ \\
$m_3$ & $M_3$ & $M_4$ & $M_{11}$ \\
$m_4$ & $M_4$ & $M_5$ & $M_{0}$ \\
$m_5$ & $M_5$ & $M_6$ & $M_{1}$ \\
$m_6$ & $M_6$ & $M_7$ & $M_{2}$ \\
$m_7$ & $M_7$ & $M_8$ & $M_{3}$ \\
$m_8$ & $M_8$ & $M_9$ & $M_{4}$ \\
$m_9$ & $M_9$ & $M_{10}$ & $M_{5}$ \\
$m_{10}$ & $M_{10}$ & $M_{11}$ & $M_{6}$ \\
$m_{11}$ & $M_{11}$ & $M_{0}$ & $M_{7}$ \\
\hline
\end{tabular}
\caption{\scriptsize Voice Leading: Minor Chords $m_k$.}
\label{tab:far_wide_minor}
\end{minipage}
\end{table}

\subsection{Topological Properties}
The Levi graph of this system is a connected cubic graph on 24 vertices. Its structure is particularly rich, possessing two distinct Hamiltonian cycles and a clean tiling by hexagons.

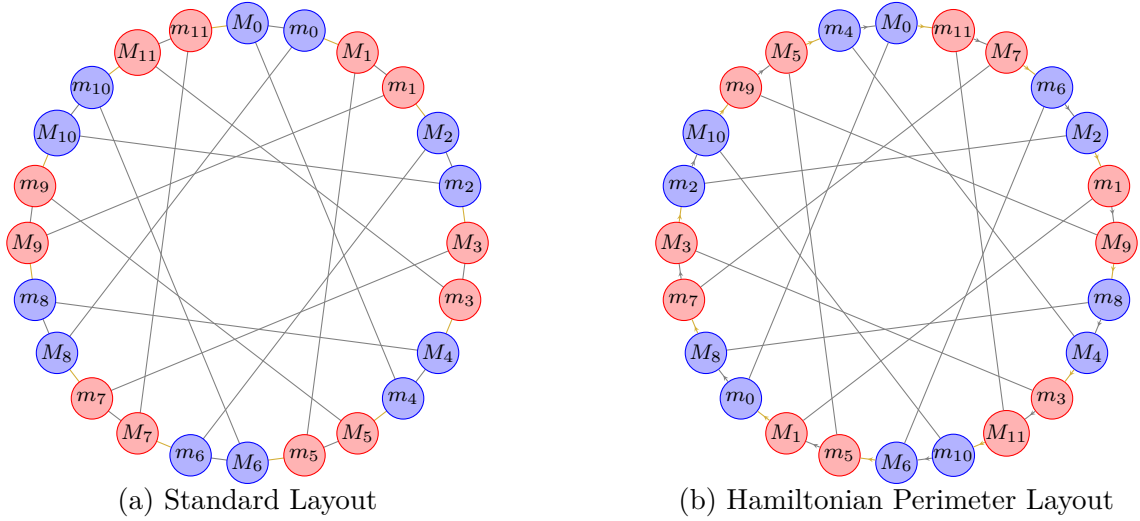
\begin{figure}[H]
    \centering
    \tikzset{
        mynode/.style={circle, inner sep=0pt, minimum size=0.55cm, font=\tiny}
    }
    \tikzset{smallmidarrow/.style={postaction={decorate}, decoration={markings,mark=at position 0.5 with {\arrow[scale=0.6]{stealth}}}}}

    \begin{minipage}{0.48\textwidth}
        \centering
        \begin{tikzpicture}[scale=0.65]
            \def\R{4.5}
            \definecolor{mygold}{RGB}{212, 175, 55}

            \foreach \i in {0,1,...,11} {
                \pgfmathsetmacro{\angM}{90 - \i*30}
                \pgfmathsetmacro{\angm}{90 - \i*30 - 15}
                \coordinate (cM\i) at (\angM:\R);
                \coordinate (cm\i) at (\angm:\R);
            }

            \foreach \i in {0,1,...,11} {
                \draw[gray, thin] (cM\i) -- (cm\i);
                \pgfmathsetmacro{\idxR}{int(mod(\i+4,12))}
                \draw[gray, thin] (cM\i) -- (cm\idxR);
                \pgfmathsetmacro{\idxL}{int(mod(\i+11,12))}
                \draw[thin, color=mygold] (cM\i) -- (cm\idxL);
            }

            \foreach \i in {0,1,...,11} {
                \pgfmathsetmacro{\isOdd}{int(mod(\i,2))}
                \ifnum\isOdd=0
                    \node[mynode, draw=blue, fill=blue!30] at (cM\i) {$M_{\i}$};
                    \node[mynode, draw=blue, fill=blue!30] at (cm\i) {$m_{\i}$};
                \else
                    \node[mynode, draw=red, fill=red!30] at (cM\i) {$M_{\i}$};
                    \node[mynode, draw=red, fill=red!30] at (cm\i) {$m_{\i}$};
                \fi
            }
            \node at (0, -\R-0.8) {\small (a) Standard Layout};
        \end{tikzpicture}
    \end{minipage}%
    \hfill
    \begin{minipage}{0.48\textwidth}
        \centering
        \begin{tikzpicture}[scale=0.65, >=stealth]
            \def\R{4.5}
            \definecolor{mygold}{RGB}{212, 175, 55}
            
            
            \foreach \k/\id/\lab [count=\i from 0] in {
                0/M_0/M_{0}, 11/m_11/m_{11}, 7/M_7/M_{7}, 6/m_6/m_{6}, 
                2/M_2/M_{2}, 1/m_1/m_{1}, 9/M_9/M_{9}, 8/m_8/m_{8}, 
                4/M_4/M_{4}, 3/m_3/m_{3}, 11/M_11/M_{11}, 10/m_10/m_{10}, 
                6/M_6/M_{6}, 5/m_5/m_{5}, 1/M_1/M_{1}, 0/m_0/m_{0}, 
                8/M_8/M_{8}, 7/m_7/m_{7}, 3/M_3/M_{3}, 2/m_2/m_{2}, 
                10/M_10/M_{10}, 9/m_9/m_{9}, 5/M_5/M_{5}, 4/m_4/m_{4}} 
            {
                \pgfmathsetmacro{\angle}{90 - \i * (360/24)}
                
                \pgfmathsetmacro{\isOdd}{int(mod(\k,2))}
                \ifnum\isOdd=0
                    \node[mynode, draw=blue, fill=blue!30, alias=P_\id] (Node\i) at (\angle:\R) {$\lab$};
                \else
                    \node[mynode, draw=red, fill=red!30, alias=P_\id] (Node\i) at (\angle:\R) {$\lab$};
                \fi
            }
            
            \foreach \k in {0,...,11} {
                \draw[gray, thin] (P_M_\k) -- (P_m_\k);
            }

            \foreach \i in {0,...,23} {
                \pgfmathsetmacro{\next}{int(mod(\i+1,24))}
                \pgfmathsetmacro{\isR}{int(mod(\i,2))} 
                \ifnum\isR=0 
                    \draw[color=mygold, thin, smallmidarrow] (Node\i) -- (Node\next);
                \else 
                    \draw[gray, thin, smallmidarrow] (Node\i) -- (Node\next);
                \fi
            }
            
            \node at (0, -\R-0.8) {\small (b) Hamiltonian Perimeter Layout};
        \end{tikzpicture}
    \end{minipage}
    \caption{The Connected Levi Graph for $(t,s)=(11,8)$. (a) Standard layout. (b) Hamiltonian perimeter layout following the Circle of Fifths ($L \cdot R$). Note that, unlike in the previous system, the $PL$ cycle is also Hamiltonian; it corresponds to the chromatic perimeter sequence visible in the standard layout (a).}
    \label{fig:levi_circular_delta11_8}
\end{figure}

\subsubsection{Minimal Circles and Girth}
An analysis of the cyclic paths reveals a key distinction from the $(10,9)$ system: the presence of binary hexagons.

\begin{itemize}
    \item \textbf{The $PL$-cycle:} The sequence $M_k \xrightarrow{P} m_k \xrightarrow{L} M_{k+1}$ induces a root shift of $+1$. Since $\gcd(1, 12)=1$, this generates a Hamiltonian cycle of length 24 (the Chromatic Circle).
    \item \textbf{The $LR$-cycle:} The sequence $M_k \xrightarrow{L} m_{k+11} \xrightarrow{R} M_{k+7}$ induces a root shift of $+7$. Since $\gcd(7,12)=1$, this also generates a Hamiltonian cycle of length 24 (the Circle of Fifths).
    \item \textbf{The $PR$-cycle:} The sequence $M_k \xrightarrow{P} m_k \xrightarrow{R} M_{k+8}$ induces a root shift of $+8$. Since $12/\gcd(8,12)=3$, this generates four disjoint cycles of length 6 (Hexagons).
\end{itemize}

\begin{theorem}[Girth of the Far Wide Thirds System]
The Levi graph of the $(11,8)$ system has girth 6. It is tiled by four disjoint $PR$-hexagons.
\end{theorem}

\begin{proof}
Consider the sequence $S = P \cdot R$. The path proceeds:
$M_k \xrightarrow{P} m_k \xrightarrow{R} M_{k+8} \xrightarrow{P} m_{k+8} \xrightarrow{R} M_{k+16} \xrightarrow{P} m_{k+16} \xrightarrow{R} M_{k+24}$.
Since $24 \equiv 0 \pmod{12}$, the path closes after 6 edges.
\[ M_0 \leftrightarrow m_0 \leftrightarrow M_8 \leftrightarrow m_8 \leftrightarrow M_4 \leftrightarrow m_4 \leftrightarrow M_0 \]
Since the graph is bipartite, there are no odd cycles. A girth of 4 would require $M_k$ to share two neighbors with some $m_j$. The neighbors of $M_0$ are $\{m_0, m_{11}, m_4\}$. None of these share another Major neighbor with $M_0$ via $P, L$, or $R$ in fewer than 6 steps. Thus, the girth is 6.
\end{proof}

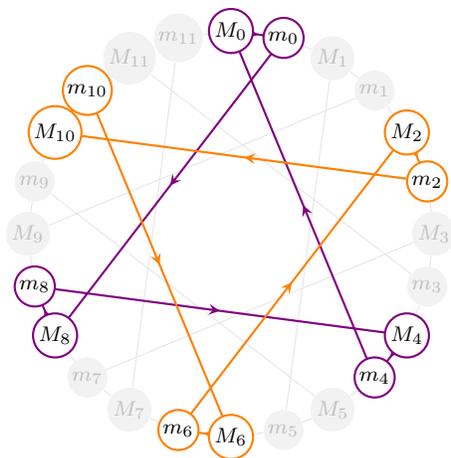
\begin{figure}[H]
    \centering
    \begin{tikzpicture}[
        scale=0.6, 
        >=stealth,
        midarrow/.style={
            postaction={decorate},
            decoration={
                markings,
                mark=at position 0.5 with {\arrow{stealth}}
            }
        },
        mynode/.style={circle, inner sep=1pt, fill=gray!10, text=gray!50, font=\tiny}
    ]
        \def\R{4.5}
        
        \foreach \i in {0,1,...,11} {
            \pgfmathsetmacro{\angM}{90 - \i*30}
            \pgfmathsetmacro{\angm}{90 - \i*30 - 15}
            \coordinate (M\i) at (\angM:\R);
            \coordinate (m\i) at (\angm:\R);
            
            \node[mynode] (N_M\i) at (M\i) {$M_{\i}$};
            \node[mynode] (N_m\i) at (m\i) {$m_{\i}$};
        }
        
        \foreach \i in {0,1,...,11} {
            \pgfmathsetmacro{\idxL}{int(mod(\i+11,12))}
            \pgfmathsetmacro{\idxR}{int(mod(\i+4,12))}
            \draw[gray!20, thin] (N_M\i) -- (N_m\i); 
            \draw[gray!20, thin] (N_M\i) -- (N_m\idxR); 
            \draw[gray!20, thin] (N_M\i) -- (N_m\idxL); 
        }

        \draw[violet, thick, midarrow] (N_M0) -- (N_m0); 
        \draw[violet, thick, midarrow] (N_m0) -- (N_M8); 
        \draw[violet, thick, midarrow] (N_M8) -- (N_m8); 
        \draw[violet, thick, midarrow] (N_m8) -- (N_M4); 
        \draw[violet, thick, midarrow] (N_M4) -- (N_m4); 
        \draw[violet, thick, midarrow] (N_m4) -- (N_M0); 
        
        \foreach \n/\lab in {M0/M_{0}, m0/m_{0}, M8/M_{8}, m8/m_{8}, M4/M_{4}, m4/m_{4}} {
             \node[circle, draw=violet, fill=white, inner sep=1pt, font=\tiny, thick] at (N_\n) {$\lab$};
        }

        \draw[orange, thick, midarrow] (N_M2) -- (N_m2); 
        \draw[orange, thick, midarrow] (N_m2) -- (N_M10); 
        \draw[orange, thick, midarrow] (N_M10) -- (N_m10); 
        \draw[orange, thick, midarrow] (N_m10) -- (N_M6); 
        \draw[orange, thick, midarrow] (N_M6) -- (N_m6); 
        \draw[orange, thick, midarrow] (N_m6) -- (N_M2); 

        \foreach \n/\lab in {M2/M_{2}, m2/m_{2}, M10/M_{10}, m10/m_{10}, M6/M_{6}, m6/m_{6}} {
             \node[circle, draw=orange, fill=white, inner sep=1pt, font=\tiny, thick] at (N_\n) {$\lab$};
        }

        \node[anchor=north] at (0, -\R-0.2) {
            \begin{tabular}{l}
            \textcolor{violet}{\textbf{Violet Hexagon}: The minimal $PR$-cycle starting at $M_0$} \\
            \textcolor{orange}{\textbf{Orange Hexagon}: The disjoint minimal $PR$-cycle starting at $M_2$}
            \end{tabular}
        };

    \end{tikzpicture}
    \caption{Visualization of two disjoint minimal cycles in the Far Wide Thirds system. The girth is 6, realized by simpler binary hexagons $(PR)^3$ rather than the ternary loops of the previous system.}
    \label{fig:minimal_hexagons_delta11_8}
\end{figure}

\subsection{Algebraic characterization and graph invariants}
The "Far Wide Thirds" system is algebraically cleaner than its $\Delta=3$ cousin (the Wide system) and topologically distinct from the $(10,9)$ system.

\begin{enumerate}
    \item \textbf{Configuration:} Like the $(10,9)$ system, this graph forms a valid $12_3$ geometric configuration (Girth 6). However, its tiling is simpler, relying on binary $PR$-hexagons rather than ternary $PLR$-hexagons.
    
    \item \textbf{Dual Hamiltonicity:} A unique feature of this system is that it possesses \textit{two} purely binary Hamiltonian cycles:
    \begin{itemize}
        \item \textbf{The Chromatic Tour ($P \cdot L$):} Alternating $P$ and $L$ steps through neighbors $M_k \to M_{k+1}$. This provides a chromatic traversal of the universe.
        \item \textbf{The Circle of Fifths Tour ($L \cdot R$):} Alternating $L$ and $R$ steps through $M_k \to M_{k+7}$. This stitches the universe together via the Perfect Fifth.
    \end{itemize}
\end{enumerate}
This duality suggests that the $(11,8)$ system, despite its "exotic" interval sizes, is one of the most highly connected and symmetric harmonic spaces available in 12-TET.

\section{The Case $\Delta=5$, $(t,s)=(6,1)$: The ``Tritone'' System}

In this section, we examine the specific harmonic configuration defined by the parameter $\Delta=5$. Following the classification in Section 3, this corresponds to the solution pair $(t, s) = (6, 1)$, where the ``Major Third'' is a tritone ($t=6$) and the ``Minor Third'' is a semitone ($s=1$).

\subsection{Major and Minor Chords}

The Major chord $M_k$ is defined by the set $\{k, k+6, k+7\}$, and the Minor chord $m_k$ by the set $\{k, k+1, k+7\}$. Table \ref{tab:delta5} enumerates the 24 polychords constituting this harmonic universe. Since these chord structures do not correspond to standard Western triads, we retain the abstract notation $M_k$ and $m_k$ to designate the chords rooted at $k$.

\begin{table}[H]
   \centering
   \scriptsize 
    \renewcommand{\arraystretch}{0.7}
    \begin{tabular}{|c|c|c|c|c|}
        \hline
        \textbf{Root ($k$)} & \textbf{Major Chord Set} & \textbf{Symbol} & \textbf{Minor Chord Set} & \textbf{Symbol} \\
        \hline
        0 & $\{0, 6, 7\}$ & $M_0$ & $\{0, 1, 7\}$ & $m_0$ \\
        1 & $\{1, 7, 8\}$ & $M_1$ & $\{1, 2, 8\}$ & $m_1$ \\
        2 & $\{2, 8, 9\}$ & $M_2$ & $\{2, 3, 9\}$ & $m_2$ \\
        3 & $\{3, 9, 10\}$ & $M_3$ & $\{3, 4, 10\}$ & $m_3$ \\
        4 & $\{4, 10, 11\}$ & $M_4$ & $\{4, 5, 11\}$ & $m_4$ \\
        5 & $\{5, 11, 0\}$ & $M_5$ & $\{5, 6, 0\}$ & $m_5$ \\
        6 & $\{6, 0, 1\}$ & $M_6$ & $\{6, 7, 1\}$ & $m_6$ \\
        7 & $\{7, 1, 2\}$ & $M_7$ & $\{7, 8, 2\}$ & $m_7$ \\
        8 & $\{8, 2, 3\}$ & $M_8$ & $\{8, 9, 3\}$ & $m_8$ \\
        9 & $\{9, 3, 4\}$ & $M_9$ & $\{9, 10, 4\}$ & $m_9$ \\
        10 & $\{10, 4, 5\}$ & $M_{10}$ & $\{10, 11, 5\}$ & $m_{10}$ \\
        11 & $\{11, 5, 6\}$ & $M_{11}$ & $\{11, 0, 6\}$ & $m_{11}$ \\
        \hline
    \end{tabular}
    \caption{The complete set of harmonic triads generated by the solution $(t,s)=(6,1)$ in 12-TET. The Major chords consist of a root, tritone, and fifth, while Minor chords consist of a root, semitone, and fifth.}
    \label{tab:delta5}
\end{table}
\subsection{The ``Tritone'' System Tonnetz}

We now establish the connectivity of the harmonic space for the ``Tritone'' system by applying the three standard involutions $P$, $L$, and $R$. In this configuration ($\Delta=5$), the transformations map the chords as follows: $P$ connects chords with the same root; $L$ connects a Major chord to the Minor chord a tritone away ($t=6$); and $R$ connects a Major chord to the Minor chord a semitone down ($s=1$).

Tables \ref{tab:neighbors_M_6_1} and \ref{tab:neighbors_m_6_1} explicitly list the $P, L, R$ neighbors for every Major and Minor chord in the system.

\begin{table}[H]
    \centering
    \scriptsize
    \renewcommand{\arraystretch}{1.2}
    
    \begin{minipage}[t]{0.48\textwidth}
        \centering
        \begin{tabular}{|c||c|c|c|}
            \hline
            \textbf{Chord} & \textbf{P-neigh.} & \textbf{L-neigh.} & \textbf{R-neigh.} \\[-0.5ex]
            $(M_k)$ & $(m_k)$ & $(m_{k+6})$ & $(m_{k-1})$ \\
            \hline
            $M_0$ & $m_0$ & $m_6$ & $m_{11}$ \\
            $M_1$ & $m_1$ & $m_7$ & $m_0$ \\
            $M_2$ & $m_2$ & $m_8$ & $m_1$ \\
            $M_3$ & $m_3$ & $m_9$ & $m_2$ \\
            $M_4$ & $m_4$ & $m_{10}$ & $m_3$ \\
            $M_5$ & $m_5$ & $m_{11}$ & $m_4$ \\
            $M_6$ & $m_6$ & $m_0$ & $m_5$ \\
            $M_7$ & $m_7$ & $m_1$ & $m_6$ \\
            $M_8$ & $m_8$ & $m_2$ & $m_7$ \\
            $M_9$ & $m_9$ & $m_3$ & $m_8$ \\
            $M_{10}$ & $m_{10}$ & $m_4$ & $m_9$ \\
            $M_{11}$ & $m_{11}$ & $m_5$ & $m_{10}$ \\
            \hline
        \end{tabular}
        \caption{Neighbors of Major chords.}
        \label{tab:neighbors_M_6_1}
    \end{minipage}
    \hfill
    \begin{minipage}[t]{0.48\textwidth}
        \centering
        \begin{tabular}{|c||c|c|c|}
            \hline
            \textbf{Chord} & \textbf{P-neigh.} & \textbf{L-neigh.} & \textbf{R-neigh.} \\[-0.5ex]
            $(m_k)$ & $(M_k)$ & $(M_{k+6})$ & $(M_{k+1})$ \\
            \hline
            $m_0$ & $M_0$ & $M_6$ & $M_1$ \\
            $m_1$ & $M_1$ & $M_7$ & $M_2$ \\
            $m_2$ & $M_2$ & $M_8$ & $M_3$ \\
            $m_3$ & $M_3$ & $M_9$ & $M_4$ \\
            $m_4$ & $M_4$ & $M_{10}$ & $M_5$ \\
            $m_5$ & $M_5$ & $M_{11}$ & $M_6$ \\
            $m_6$ & $M_6$ & $M_0$ & $M_7$ \\
            $m_7$ & $M_7$ & $M_1$ & $M_8$ \\
            $m_8$ & $M_8$ & $M_2$ & $M_9$ \\
            $m_9$ & $M_9$ & $M_3$ & $M_{10}$ \\
            $m_{10}$ & $M_{10}$ & $M_4$ & $M_{11}$ \\
            $m_{11}$ & $M_{11}$ & $M_5$ & $M_0$ \\
            \hline
        \end{tabular}
        \caption{Neighbors of Minor chords.}
        \label{tab:neighbors_m_6_1}
    \end{minipage}
\end{table}

\subsection{The Levi graph, Hamiltonian and minimal cycles}

Figure \ref{fig:levi_wide_cycles} presents the topological structure of the Levi graph for the $\Delta=5$ system.

The vertices are arranged in the chromatic sequence $(M_0, m_0, M_1, m_1, \dots)$ around the circle. This arrangement reveals that the graph is Hamiltonian: the perimeter forms a continuous cycle alternating between $P$ (connecting $M_k$ to $m_k$) and $R$ (connecting $m_k$ to $M_{k+1}$).

Crucially, this graph differs from the $\Delta=1$ case in its girth. While the standard Tonnetz contains hexagonal cycles (girth 6), the ``Tritone'' system contains cycles of length 4, indicating a girth of 4. Figure \ref{fig:levi_wide_cycles} highlights two such minimal cycles with opposite chirality:
\begin{itemize}
    \item A \textbf{clockwise} cycle (violet) involving the sequence $M_0 \xrightarrow{P} m_0 \xrightarrow{L} M_6 \xrightarrow{P} m_6 \xrightarrow{L} M_0$.
    \item A \textbf{counter-clockwise} cycle (gold) involving the sequence $m_9 \xrightarrow{P} M_9 \xrightarrow{L} m_3 \xrightarrow{P} M_3 \xrightarrow{L} m_9$.
\end{itemize}
The $L$-transformations (internal chords) connect diametrically opposed pairs, facilitating these "bowtie" shaped 4-cycles that cross the harmonic center.

\begin{figure}[H]
    \centering
    \definecolor{mygold}{rgb}{0.85, 0.65, 0.13}
    \definecolor{violet}{rgb}{0.5, 0.0, 0.5}
    
    \begin{tikzpicture}[scale=0.5, 
        every node/.style={transform shape},
        arrow_mid/.style={decoration={markings, mark=at position 0.70 with {\arrow{Latex[length=3mm, width=1mm]}}}, postaction={decorate}}
    ]
        \def\R{5.5} 
        
        \foreach \k in {0,...,11} {
            \pgfmathsetmacro{\angM}{90 - \k*30}
            \pgfmathsetmacro{\angm}{90 - \k*30 - 15}
            \coordinate (M\k) at (\angM:\R);
            \coordinate (m\k) at (\angm:\R);
            
            \pgfmathsetmacro{\idxL}{int(mod(\k+6,12))}
            \pgfmathsetmacro{\angmL}{90 - \idxL*30 - 15}
            \coordinate (mL) at (\angmL:\R);
            \draw[mygold, thin] (M\k) -- (mL);
            
            \draw[red, thin] (M\k) -- (m\k);
            
            \pgfmathsetmacro{\idxNext}{int(mod(\k+1,12))}
            \pgfmathsetmacro{\angMnext}{90 - \idxNext*30}
            \coordinate (Mnext) at (\angMnext:\R);
            \draw[blue, thin] (m\k) -- (Mnext);
        }
        
        
        \draw[violet, line width=1.0pt, arrow_mid] (M0) -- (m0);
        \draw[violet, line width=1.0pt, arrow_mid] (m0) -- (M6);
        \draw[violet, line width=1.0pt, arrow_mid] (M6) -- (m6);
        \draw[violet, line width=1.0pt, arrow_mid] (m6) -- (M0);
        
        \draw[orange, line width=1.0pt, arrow_mid] (m9) -- (M9);
        \draw[orange, line width=1.0pt, arrow_mid] (M9) -- (m3);
        \draw[orange, line width=1.0pt, arrow_mid] (m3) -- (M3);
        \draw[orange, line width=1.0pt, arrow_mid] (M3) -- (m9);

        \foreach \k in {0,...,11} {
            \pgfmathsetmacro{\angM}{90 - \k*30}
            \pgfmathsetmacro{\angm}{90 - \k*30 - 15}
            
            \node[circle, fill=blue!30, draw=blue, thick, 
                  inner sep=0pt, minimum size=0.9cm, text=black, font=\small] 
                  at (\angM:\R) {$M_{\k}$};
            
            \node[circle, fill=red!30, draw=red, thick, 
                  inner sep=0pt, minimum size=0.9cm, text=black, font=\small] 
                  at (\angm:\R) {$m_{\k}$};
        }
    \end{tikzpicture}
    \caption{The Levi graph of the Tritone System ($\Delta=5$) arranged chromatically. The perimeter forms a Hamiltonian cycle. The diagram specifically highlights two minimal cycles (girth 4) of opposite chirality: a clockwise cycle in violet ($M_0$-$m_0$-$M_6$-$m_6$) and a counter-clockwise cycle in gold ($m_9$-$M_9$-$m_3$-$M_3$). These 4-cycles correspond to the commutation of operations that is absent in the standard Tonnetz.}
    \label{fig:levi_wide_cycles}
\end{figure}
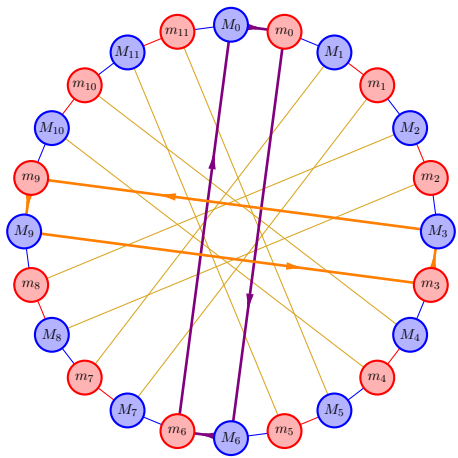

\subsection{Properties of the ``Tritone'' System  Graph}

We now characterize the topology of the graph generated by the $\Delta=5$ system, contrasting it with the properties of the standard $\Delta=1$ system established in Section 4.

\subsubsection{Girth and the Absence of a Geometric Configuration}
The most fundamental topological difference between the standard system and the ``Wide'' system lies in the girth of their respective Levi graphs.

In Section 4, we established that the Levi graph of the standard system has a girth of 6. The smallest cycles are hexagons (e.g., the sequence $PLRPLR$), which allowed us to interpret the graph as the incidence structure of a valid $12_3$ configuration.

In contrast, the graph of the Tritone system has a girth of 4. As illustrated in Figure \ref{fig:levi_wide_cycles}, the commutation of the $P$ and $L$ transformations creates quadrangular cycles:
\begin{equation}
    M_k \xrightarrow{P} m_k \xrightarrow{L} M_{k+6} \xrightarrow{P} m_{k+6} \xrightarrow{L} M_k.
\end{equation}
From the perspective of incidence geometry, a cycle of length 4 implies that two distinct points (chords) are connected by two distinct lines (transformations), or conversely, that two lines intersect at two points. This violates the axioms of a linear configuration, where two points must determine a unique line. Therefore, unlike the standard case, the Tritone system \textbf{does not} constitute a valid $12_3$ geometric configuration.

\subsubsection{Hamiltonian Structure}
While it fails to form a geometric configuration, the graph possesses a high degree of symmetry and connectivity. As demonstrated by the chromatic arrangement in Figure \ref{fig:levi_wide_cycles}, the graph is Hamiltonian. The perimeter, defined by the alternating sequence of $P$ and $R$ transformations, forms a Hamiltonian cycle of length 24:
\begin{equation}
    M_0 \xrightarrow{P} m_0 \xrightarrow{R} M_1 \xrightarrow{P} m_1 \xrightarrow{R} \dots \xrightarrow{R} M_{11} \xrightarrow{P} m_{11} \xrightarrow{R} M_0.
\end{equation}
This cycle corresponds to a complete chromatic traversal of the 12 roots, visiting every Major and Minor chord exactly once. This contrasts with the standard system, where the Hamiltonian path is typically constructed using the $L$ and $R$ operations (following the Circle of Fifths).

\subsubsection{Connectivity}
The graph can be described as a cubic, vertex-transitive graph of order 24. Structurally, it consists of the 24-cycle perimeter (formed by $P$ and $R$) with "cross-chords" formed by the $L$-transformation. Since $L$ connects $M_k$ to $m_{k+6}$, it links regions of the graph that are diametrically opposite in the chromatic circle. This high connectivity "shrinks" the graph, allowing for rapid traversal between harmonically distant chords via the tritone intervals.

\section{Three exotic musical systems}

So far, we have explored systems generated by the chromatic step ($q=1$ or $11$) and the perfect fifth ($q=7$). While these values of $q$ yield the most familiar musical structures, our generalized theoretical framework allows for any generator $q$ and any partition $(t,s)$ such that $t+s=q$.

To fully understand the landscape of possible harmonic universes in 12-TET, we must look beyond the standard generator. We wish to investigate other possibilities for $q$ to see if they yield coherent musical systems. Specifically, we will examine three "exotic" systems that correspond to the remaining valid classes of affine offsets in $\mathbb{Z}_{12}$:
\begin{itemize}
    \item The ``Tritonian Fifth'' System with $q=6$ and $(t,s)=(4,2)$.
    \item The ``Tritone-Wide'' System with $q=10$ and $(t,s)=(8,2)$.
    \item The ``Equally Spread'' System with $q=9$ and  $(t,s)=(6,3)$.
\end{itemize}

We begin with the most symmetric and topologically distinct of these: the system generated by the Tritone.

\subsection{The ``Tritonian Fifth'' System: \texorpdfstring{$(t,s)=(4,2)$}{t=4, s=2}}
In this system, the generator $q=6$ (the Tritone) is partitioned into a Major Third ($t=4$) and a Major Second ($s=2$). The parameter $\Delta = t-s = 2$.

\subsubsection{Chord Definitions}
The Major chord $M_k$ and Minor chord $m_k$ are defined as:
\begin{align}
    M_k &= \{k, k+4, k+6\} \\
    m_k &= \{k, k+2, k+6\}
\end{align}
Note that the interval vector for both chords is $[2, 4, 6]$. Unlike the standard triadic system, these chords are inversionally symmetric sets (specifically, set class 3-8 in Forte classification).

The full list of pitch sets for all 24 chords is provided in Tables \ref{tab:tritone_major_def} and \ref{tab:tritone_minor_def}.

\begin{table}[H]
\centering
\scriptsize
\renewcommand{\arraystretch}{1.2}
\begin{minipage}{0.48\textwidth}
\centering
\begin{tabular}{|c|c|}
\hline
\textbf{Root} ($k$) & \textbf{Major Chord} $M_k = \{k, k+4, k+6\}$ \\
\hline
0 & \{0, 4, 6\} \\
1 & \{1, 5, 7\} \\
2 & \{2, 6, 8\} \\
3 & \{3, 7, 9\} \\
4 & \{4, 8, 10\} \\
5 & \{5, 9, 11\} \\
6 & \{6, 10, 0\} \\
7 & \{7, 11, 1\} \\
8 & \{8, 0, 2\} \\
9 & \{9, 1, 3\} \\
10 & \{10, 2, 4\} \\
11 & \{11, 3, 5\} \\
\hline
\end{tabular}
\caption{Major Triads in the Tritone System.}
\label{tab:tritone_major_def}
\end{minipage}
\hfill
\begin{minipage}{0.48\textwidth}
\centering
\begin{tabular}{|c|c|}
\hline
\textbf{Root} ($k$) & \textbf{Minor Chord} $m_k = \{k, k+2, k+6\}$ \\
\hline
0 & \{0, 2, 6\} \\
1 & \{1, 3, 7\} \\
2 & \{2, 4, 8\} \\
3 & \{3, 5, 9\} \\
4 & \{4, 6, 10\} \\
5 & \{5, 7, 11\} \\
6 & \{6, 8, 0\} \\
7 & \{7, 9, 1\} \\
8 & \{8, 10, 2\} \\
9 & \{9, 11, 3\} \\
10 & \{10, 0, 4\} \\
11 & \{11, 1, 5\} \\
\hline
\end{tabular}
\caption{Minor Triads in the Tritone System.}
\label{tab:tritone_minor_def}
\end{minipage}
\end{table}

\subsubsection{The Tonnetz and Voice Leading}
The voice-leading transformations are derived by preserving the common intervals of size $q=6$, $t=4$, and $s=2$. Due to the geometry of the system, the transformations map as follows:
\begin{align*}
    P(M_k) &= m_k & P(m_k) &= M_k \\
    L(M_k) &= m_{k-2} & L(m_k) &= M_{k+2} \\
    R(M_k) &= m_{k+4} & R(m_k) &= M_{k-4}
\end{align*}
The complete voice-leading map is presented in Tables \ref{tab:tritone_major_vl} and \ref{tab:tritone_minor_vl}.

\begin{table}[H]
\centering
\scriptsize
\renewcommand{\arraystretch}{1.2}
\begin{minipage}{0.48\textwidth}
\centering
\begin{tabular}{|c||c|c|c|}
\hline
\textbf{Source} ($M_k$) & \textbf{P} ($m_k$) & \textbf{L} ($m_{k-2}$) & \textbf{R} ($m_{k+4}$) \\
\hline
$M_0$ & $m_0$ & $m_{10}$ & $m_4$ \\
$M_1$ & $m_1$ & $m_{11}$ & $m_5$ \\
$M_2$ & $m_2$ & $m_{0}$ & $m_6$ \\
$M_3$ & $m_3$ & $m_{1}$ & $m_7$ \\
$M_4$ & $m_4$ & $m_{2}$ & $m_8$ \\
$M_5$ & $m_5$ & $m_{3}$ & $m_9$ \\
$M_6$ & $m_6$ & $m_{4}$ & $m_{10}$ \\
$M_7$ & $m_7$ & $m_{5}$ & $m_{11}$ \\
$M_8$ & $m_8$ & $m_{6}$ & $m_{0}$ \\
$M_9$ & $m_9$ & $m_{7}$ & $m_{1}$ \\
$M_{10}$ & $m_{10}$ & $m_{8}$ & $m_{2}$ \\
$M_{11}$ & $m_{11}$ & $m_{9}$ & $m_{3}$ \\
\hline
\end{tabular}
\caption{Voice Leading: Major Chords $M_k$.}
\label{tab:tritone_major_vl}
\end{minipage}
\hfill
\begin{minipage}{0.48\textwidth}
\centering
\begin{tabular}{|c||c|c|c|}
\hline
\textbf{Source} ($m_k$) & \textbf{P} ($M_k$) & \textbf{L} ($M_{k+2}$) & \textbf{R} ($M_{k-4}$) \\
\hline
$m_0$ & $M_0$ & $M_2$ & $M_8$ \\
$m_1$ & $M_1$ & $M_3$ & $M_9$ \\
$m_2$ & $M_2$ & $M_4$ & $M_{10}$ \\
$m_3$ & $M_3$ & $M_5$ & $M_{11}$ \\
$m_4$ & $M_4$ & $M_6$ & $M_{0}$ \\
$m_5$ & $M_5$ & $M_7$ & $M_{1}$ \\
$m_6$ & $M_6$ & $M_8$ & $M_{2}$ \\
$m_7$ & $M_7$ & $M_9$ & $M_{3}$ \\
$m_8$ & $M_8$ & $M_{10}$ & $M_{4}$ \\
$m_9$ & $M_9$ & $M_{11}$ & $M_{5}$ \\
$m_{10}$ & $M_{10}$ & $M_{0}$ & $M_{6}$ \\
$m_{11}$ & $M_{11}$ & $M_{1}$ & $M_{7}$ \\
\hline
\end{tabular}
\caption{Voice Leading: Minor Chords $m_k$.}
\label{tab:tritone_minor_vl}
\end{minipage}
\end{table}

\subsubsection{Topological Disconnectedness}
A crucial feature of this system is immediately apparent from the transformation formulas. All root shifts are even integers:
\begin{itemize}
    \item $P$: shift 0
    \item $L$: shift $\pm 2$
    \item $R$: shift $\pm 4$
\end{itemize}
Consequently, a Major chord $M_k$ with an even root $k$ can only transform into Minor chords with even roots ($m_k, m_{k-2}, m_{k+4}$). Conversely, chords with odd roots connect only to other odd-rooted chords.

This causes the Levi graph to fracture into **two disjoint, isomorphic components**: the "Even Universe" and the "Odd Universe."

\begin{figure}[H]
    \centering
    \begin{tikzpicture}[scale=0.75]
        \def\R{2.8}
        \definecolor{mygold}{RGB}{212, 175, 55}
        \tikzset{mynode/.style={circle, draw, thin, inner sep=1pt, font=\scriptsize}}
        
        \begin{scope}[xshift=-4cm]
            
            \foreach \k [count=\i] in {0,2,4,6,8,10} {
                \pgfmathsetmacro{\angM}{90 - (\i-1)*60}
                \pgfmathsetmacro{\angm}{\angM - 30}
                
                \node[mynode, fill=blue!10] (M\k) at (\angM:\R) {$M_{\k}$};
                \node[mynode, fill=red!10] (m\k) at (\angm:\R) {$m_{\k}$};
            }
            
            \foreach \k in {0,2,4,6,8,10} {
                \draw[red, thick] (M\k) -- (m\k);
                
                \pgfmathsetmacro{\prev}{int(mod(\k-2+12,12))}
                \draw[blue, thick] (M\k) -- (m\prev);
                
                \pgfmathsetmacro{\next}{int(mod(\k+4,12))}
                \draw[mygold, thick] (M\k) -- (m\next);
            }
            \node at (0, -\R-0.8) {\textbf{(a) The Even Universe}};
        \end{scope}

        \begin{scope}[xshift=4cm]
             \foreach \k [count=\i] in {1,3,5,7,9,11} {
                \pgfmathsetmacro{\angM}{90 - (\i-1)*60}
                \pgfmathsetmacro{\angm}{\angM - 30}
                
                \node[mynode, fill=blue!10] (M\k) at (\angM:\R) {$M_{\k}$};
                \node[mynode, fill=red!10] (m\k) at (\angm:\R) {$m_{\k}$};
            }
            
            \foreach \k in {1,3,5,7,9,11} {
                \draw[red, thick] (M\k) -- (m\k);
                
                \pgfmathsetmacro{\prev}{int(mod(\k-2+12,12))}
                \draw[blue, thick] (M\k) -- (m\prev);
                
                \pgfmathsetmacro{\next}{int(mod(\k+4,12))}
                \draw[mygold, thick] (M\k) -- (m\next);
            }
            \node at (0, -\R-0.8) {\textbf{(b) The Odd Universe}};
        \end{scope}
        
        \node[anchor=north] at (0, -\R-1.5) {
            \scriptsize
            \begin{tabular}{ccc}
            \textcolor{red}{\textbf{--- P (Red)}} & \textcolor{mygold}{\textbf{--- R (Gold)}} & \textcolor{blue}{\textbf{--- L (Blue)}}
            \end{tabular}
        };
    \end{tikzpicture}
    \caption{The Disconnected Levi Graph of the $(6,4,2)$ system. The harmonic world splits into two disjoint components (Even and Odd roots). Each component is Hamiltonian, with the alternating $PL$ cycle (Red-Blue) forming the perimeter, while the $R$ transformations (Gold) act as internal passages connecting $M_i$ to $m_{i+4}$.}
    \label{fig:tritone_disconnected}
\end{figure}
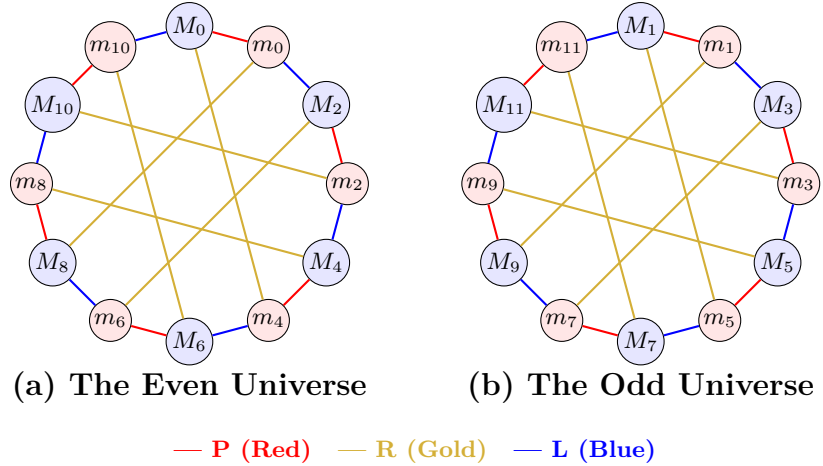
\subsubsection{Minimal Cycles and Girth}
Within each component, the graph structure is highly regular. The system's girth is 4, defined by the $LR$-cycle.

\begin{itemize}
    \item \textbf{The Square ($LR$-cycle):} $M_k \xrightarrow{L} m_{k-2} \xrightarrow{R} M_{k-6} \xrightarrow{L} m_{k-8} \xrightarrow{R} M_{k-12} \equiv M_k$.
    Since $6 \equiv -6 \pmod{12}$, this cycle has length 4: $M_0 \leftrightarrow m_{10} \leftrightarrow M_6 \leftrightarrow m_4 \leftrightarrow M_0$.
 \end{itemize}

\begin{figure}[H]
    \centering
    \begin{tikzpicture}[scale=0.8, >=stealth]
        \def\R{3.0}
        \definecolor{mygold}{RGB}{212, 175, 55}
        \tikzset{mynode/.style={circle, draw, thin, inner sep=1pt, font=\scriptsize}}
 \begin{scope}[xshift=-4cm]
        \foreach \k [count=\i] in {0,2,4,6,8,10} {
            \pgfmathsetmacro{\angM}{90 - (\i-1)*60}
            \pgfmathsetmacro{\angm}{\angM - 30}
            
            \node[mynode, fill=blue!10] (M\k) at (\angM:\R) {$M_{\k}$};
            \node[mynode, fill=red!10] (m\k) at (\angm:\R) {$m_{\k}$};
        }

        \foreach \k in {0,2,4,6,8,10} {
             \pgfmathsetmacro{\prev}{int(mod(\k-2+12,12))}
             \pgfmathsetmacro{\next}{int(mod(\k+4,12))}
             \draw[gray!30, thin] (M\k) -- (m\k);      
             \draw[gray!30, thin] (M\k) -- (m\prev);  
             \draw[gray!30, thin] (M\k) -- (m\next);  
        }

        \draw[blue, very thick] (M0) -- (m10);
        \draw[mygold, very thick] (m10) -- (M6);
        \draw[blue, very thick] (M6) -- (m4);
        \draw[mygold, very thick] (m4) -- (M0);
 \end{scope}
 
 \begin{scope}[xshift=4cm]
        \foreach \k [count=\i] in {1,3,5,7,9,11} {
            \pgfmathsetmacro{\angM}{90 - (\i-1)*60}
            \pgfmathsetmacro{\angm}{\angM - 30}
            
            \node[mynode, fill=blue!10] (M\k) at (\angM:\R) {$M_{\k}$};
            \node[mynode, fill=red!10] (m\k) at (\angm:\R) {$m_{\k}$};
        }

        \foreach \k in {1,3,5,7,9,11} {
             \pgfmathsetmacro{\prev}{int(mod(\k-2+12,12))}
             \pgfmathsetmacro{\next}{int(mod(\k+4,12))}
             \draw[gray!30, thin] (M\k) -- (m\k);      
             \draw[gray!30, thin] (M\k) -- (m\prev);  
             \draw[gray!30, thin] (M\k) -- (m\next);  
        }

        \draw[blue, very thick] (M3) -- (m1);
        \draw[mygold, very thick] (m1) -- (M9);
        \draw[blue, very thick] (M9) -- (m7);
        \draw[mygold, very thick] (m7) -- (M3);
    \end{scope}
        \node[anchor=north] at (0., -\R-0.5) {
            \begin{tabular}{l}
            \textbf{Square ($LR$)}: Alternating \textcolor{blue}{\textbf{L (Blue)}} and \textcolor{mygold}{\textbf{R (Gold)}} edges. 
                 \end{tabular}
        };
    \end{tikzpicture}
    \caption{Minimal cycles in both Even and Odd Universes visualized on the Levi graph with 12 nodes. Both graphs have girth 4 due to the presence of $LR$-squares.}
    \label{fig:tritone_minimal}
\end{figure}
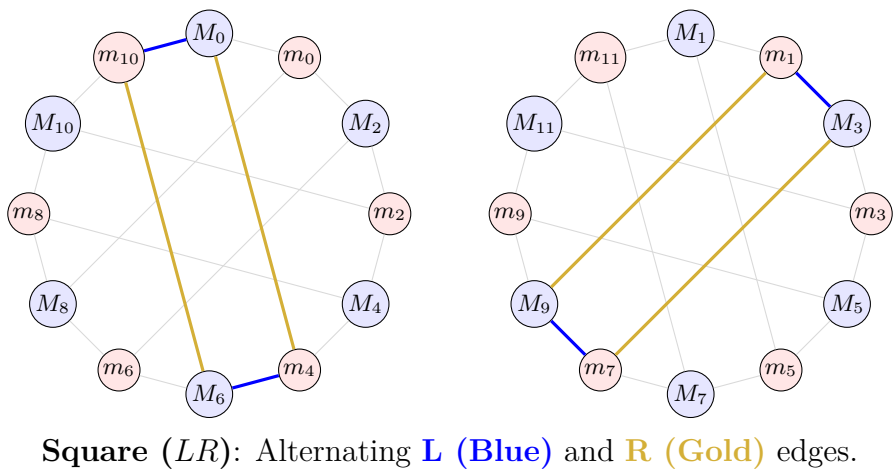
\subsubsection{Hamiltonian Cycles}
It is visible on Figure \ref{fig:tritone_disconnected} that each connected component of the Levi graph of the full system is Hamiltonian. The $PL$-cycle provides a complete tour of the component (length 12).
The sequence $M_k \xrightarrow{P} m_k \xrightarrow{L} M_{k+2}$ shifts the root by $+2$ and produces a cycle: $M_0 \to m_0 \to M_2 \to m_2 \to M_4 \to m_4 \to M_6 \to m_6 \to M_8 \to m_8 \to M_{10} \to m_{10} \to M_0$. See Figure \ref{fig:tritone_disconnected}.

\subsection{The ``Tritone-Wide'' System: \texorpdfstring{$(t,s)=(8,2)$}{t=8, s=2}}

In this system, the generator $q=10$ (equivalent to a Minor Seventh, or an inverted Major Second) is partitioned into a Minor Sixth ($t=8$) and a Major Second ($s=2$). The difference $\Delta = t-s = 6$ is once again the Tritone, suggesting a structural affinity with the previous system.

\subsubsection{Chord Definitions}
The Major chord $M_k$ and Minor chord $m_k$ are defined as:
\begin{align}
    M_k &= \{k, k+8, k+10\} \\
    m_k &= \{k, k+2, k+10\}
\end{align}
These trichords belong to set class 3-6 (Forte), which is inversionally symmetric and consists of the intervals $[2, 8, 10]$ (or normal order $[0, 2, 4]$).

The chords are listed in Tables \ref{tab:wide_major_def} and \ref{tab:wide_minor_def}.

\begin{table}[H]
\centering
\scriptsize
\renewcommand{\arraystretch}{1.2}
\begin{minipage}{0.48\textwidth}
\centering
\begin{tabular}{|c|c|}
\hline
\textbf{Root} ($k$) & \textbf{Major Chord} $M_k = \{k, k+8, k+10\}$ \\
\hline
0 & \{0, 8, 10\} \\
1 & \{1, 9, 11\} \\
2 & \{2, 10, 0\} \\
3 & \{3, 11, 1\} \\
4 & \{4, 0, 2\} \\
5 & \{5, 1, 3\} \\
6 & \{6, 2, 4\} \\
7 & \{7, 3, 5\} \\
8 & \{8, 4, 6\} \\
9 & \{9, 5, 7\} \\
10 & \{10, 6, 8\} \\
11 & \{11, 7, 9\} \\
\hline
\end{tabular}
\caption{Major Triads in the Tritone-Wide System.}
\label{tab:wide_major_def}
\end{minipage}
\hfill
\begin{minipage}{0.48\textwidth}
\centering
\begin{tabular}{|c|c|}
\hline
\textbf{Root} ($k$) & \textbf{Minor Chord} $m_k = \{k, k+2, k+10\}$ \\
\hline
0 & \{0, 2, 10\} \\
1 & \{1, 3, 11\} \\
2 & \{2, 4, 0\} \\
3 & \{3, 5, 1\} \\
4 & \{4, 6, 2\} \\
5 & \{5, 7, 3\} \\
6 & \{6, 8, 4\} \\
7 & \{7, 9, 5\} \\
8 & \{8, 10, 6\} \\
9 & \{9, 11, 7\} \\
10 & \{10, 0, 8\} \\
11 & \{11, 1, 9\} \\
\hline
\end{tabular}
\caption{Minor Triads in the Tritone-Wide System.}
\label{tab:wide_minor_def}
\end{minipage}
\end{table}

\subsubsection{Voice Leading and Graph Topology}
The voice-leading transformations preserve the structural intervals of the system ($q=10, t=8, s=2$). The transformation rules are:
\begin{align*}
    P(M_k) &= m_k & P(m_k) &= M_k \\
    L(M_k) &= m_{k-2} & L(m_k) &= M_{k+2} \\
    R(M_k) &= m_{k+8} & R(m_k) &= M_{k-8} \equiv M_{k+4}
\end{align*}

The root shifts ($\Delta P=0, \Delta L=\pm 2, \Delta R=\pm 4$) are all even. Consequently, this system also fractures into two disjoint **Even** and **Odd** universes. When compared to the $q=6$ with $(t,s)=(4,2)$ system, the current system has quite different internal connections. Now, the $R$ transformation connects $M_k$ to $m_{k+8}$ (a span of +8 semitones) rather than $m_{k+4}$.

\begin{figure}[H]
    \centering
    \begin{tikzpicture}[scale=0.75]
        \def\R{2.8}
        \definecolor{mygold}{RGB}{212, 175, 55}
        \tikzset{mynode/.style={circle, draw, thin, inner sep=1pt, font=\scriptsize}}
        
        \begin{scope}[xshift=-4cm]
            \foreach \k [count=\i] in {0,2,4,6,8,10} {
                \pgfmathsetmacro{\angM}{90 - (\i-1)*60}
                \pgfmathsetmacro{\angm}{\angM - 30}
                \node[mynode, fill=blue!10] (M\k) at (\angM:\R) {$M_{\k}$};
                \node[mynode, fill=red!10] (m\k) at (\angm:\R) {$m_{\k}$};
            }
            \foreach \k in {0,2,4,6,8,10} {
                \draw[red, thick] (M\k) -- (m\k);
                \pgfmathsetmacro{\prev}{int(mod(\k-2+12,12))}
                \draw[blue, thick] (M\k) -- (m\prev);
                \pgfmathsetmacro{\next}{int(mod(\k+8,12))}
                \draw[mygold, thick] (M\k) -- (m\next);
            }
            \node at (0, -\R-0.8) {\textbf{(a) The Even Universe}};
        \end{scope}

        \begin{scope}[xshift=4cm]
             \foreach \k [count=\i] in {1,3,5,7,9,11} {
                \pgfmathsetmacro{\angM}{90 - (\i-1)*60}
                \pgfmathsetmacro{\angm}{\angM - 30}
                \node[mynode, fill=blue!10] (M\k) at (\angM:\R) {$M_{\k}$};
                \node[mynode, fill=red!10] (m\k) at (\angm:\R) {$m_{\k}$};
            }
            \foreach \k in {1,3,5,7,9,11} {
                \draw[red, thick] (M\k) -- (m\k);
                \pgfmathsetmacro{\prev}{int(mod(\k-2+12,12))}
                \draw[blue, thick] (M\k) -- (m\prev);
                \pgfmathsetmacro{\next}{int(mod(\k+8,12))}
                \draw[mygold, thick] (M\k) -- (m\next);
            }
            \node at (0, -\R-0.8) {\textbf{(b) The Odd Universe}};
        \end{scope}
        
        \node[anchor=north] at (0, -\R-1.5) {
            \scriptsize
            \begin{tabular}{ccc}
            \textcolor{red}{\textbf{--- P (Red)}} & \textcolor{blue}{\textbf{--- L (Blue)}} & \textcolor{mygold}{\textbf{--- R (Gold)}}
            \end{tabular}
        };
    \end{tikzpicture}
    \caption{The Disconnected Levi Graph of the $(10,8,2)$ ``Tritone-Wide'' system. Like the previous system, it splits into Even and Odd components. The $P$ and $L$ connections are identical, but the $R$ transformation (Gold) now connects $M_k$ to $m_{k+8}$, creating a different internal chordal geometry.}
    \label{fig:wide_disconnected}
\end{figure}
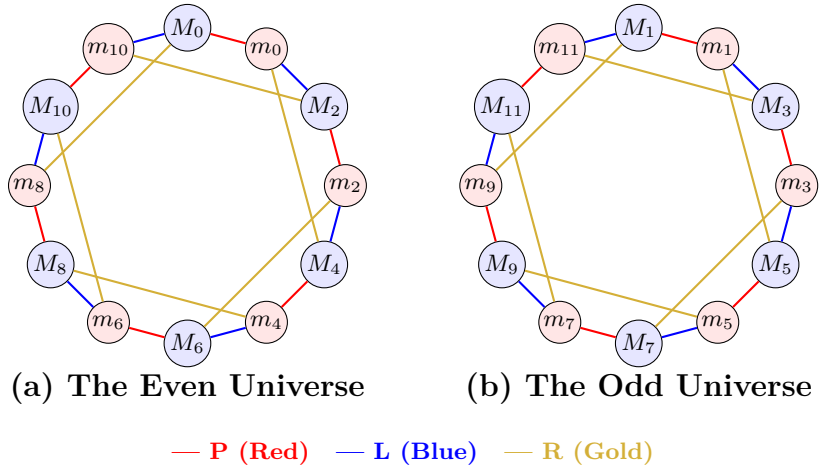

\subsubsection{Minimal Cycles and Girth}
In the ``Tritone-Wide'' system, the graph topology maintains a girth of 4, but the structure of the minimal cycles is fundamentally different from the previous system. Unlike the simple $LR$-squares found in the $q=6$, $(t,s)=(4,2)$ system, the minimal cycle here involves all three transformations: $P$, $L$, and $R$.

Specifically, the cycle is formed by the sequence $LPLR$ (or its retrograde). Starting from a major chord $M_k$, we see that the sequence 
\[
M_k \xrightarrow{L} m_{k-2} \xrightarrow{P} M_{k-2} \xrightarrow{L} m_{k-4} \xrightarrow{R} M_k
\]
gives a closed loop of 4 edges:

\begin{figure}[H]
    \centering
    \begin{tikzpicture}[scale=0.8, >=stealth]
        \def\R{3.0}
        \definecolor{mygold}{RGB}{212, 175, 55}
        \tikzset{mynode/.style={circle, draw, thin, inner sep=1pt, font=\scriptsize}}

        \begin{scope}[xshift=-4cm]
            \foreach \k [count=\i] in {0,2,4,6,8,10} {
                \pgfmathsetmacro{\angM}{90 - (\i-1)*60}
                \pgfmathsetmacro{\angm}{\angM - 30}
                \node[mynode, fill=blue!10] (M\k) at (\angM:\R) {$M_{\k}$};
                \node[mynode, fill=red!10] (m\k) at (\angm:\R) {$m_{\k}$};
            }
            \foreach \k in {0,2,4,6,8,10} {
                 \pgfmathsetmacro{\prev}{int(mod(\k-2+12,12))}
                 \pgfmathsetmacro{\next}{int(mod(\k+8,12))}
                 \draw[gray!20, thin] (M\k) -- (m\k);
                 \draw[gray!20, thin] (M\k) -- (m\prev);
                 \draw[gray!20, thin] (M\k) -- (m\next);
            }
            \draw[blue, very thick] (M0) -- (m10);   
            \draw[red, very thick] (m10) -- (M10);   
            \draw[blue, very thick] (M10) -- (m8);   
            \draw[mygold, very thick] (m8) -- (M0);  
            
            \node at (0, -\R-0.8) {\textbf{(a) Even: $M_0$-$m_{10}$-$M_{10}$-$m_8$}};
        \end{scope}

        \begin{scope}[xshift=4cm]
            \foreach \k [count=\i] in {1,3,5,7,9,11} {
                \pgfmathsetmacro{\angM}{90 - (\i-1)*60}
                \pgfmathsetmacro{\angm}{\angM - 30}
                \node[mynode, fill=blue!10] (M\k) at (\angM:\R) {$M_{\k}$};
                \node[mynode, fill=red!10] (m\k) at (\angm:\R) {$m_{\k}$};
            }
            \foreach \k in {1,3,5,7,9,11} {
                 \pgfmathsetmacro{\prev}{int(mod(\k-2+12,12))}
                 \pgfmathsetmacro{\next}{int(mod(\k+8,12))}
                 \draw[gray!20, thin] (M\k) -- (m\k);
                 \draw[gray!20, thin] (M\k) -- (m\prev);
                 \draw[gray!20, thin] (M\k) -- (m\next);
            }
            \draw[blue, very thick] (M3) -- (m1);    
            \draw[red, very thick] (m1) -- (M1);     
            \draw[blue, very thick] (M1) -- (m11);   
            \draw[mygold, very thick] (m11) -- (M3); 
            
            \node at (0, -\R-0.8) {\textbf{(b) Odd: $M_3$-$m_{1}$-$M_{1}$-$m_{11}$}};
        \end{scope}

        \node[anchor=north] at (0, -\R-1.5) {
            \begin{tabular}{l}
            \textbf{Minimal Cycle ($LPLR$)}: \textcolor{blue}{\textbf{L}} $\to$ \textcolor{red}{\textbf{P}} $\to$ \textcolor{blue}{\textbf{L}} $\to$ \textcolor{mygold}{\textbf{R}}
            \end{tabular}
        };
    \end{tikzpicture}
    \caption{Visualization of the minimal cycles (girth = 4) in the Tritone-Wide system. Unlike the previous system's $LR$-squares, these cycles involve all three transformations ($L$-$P$-$L$-$R$).}
    \label{fig:wide_minimal}
\end{figure}
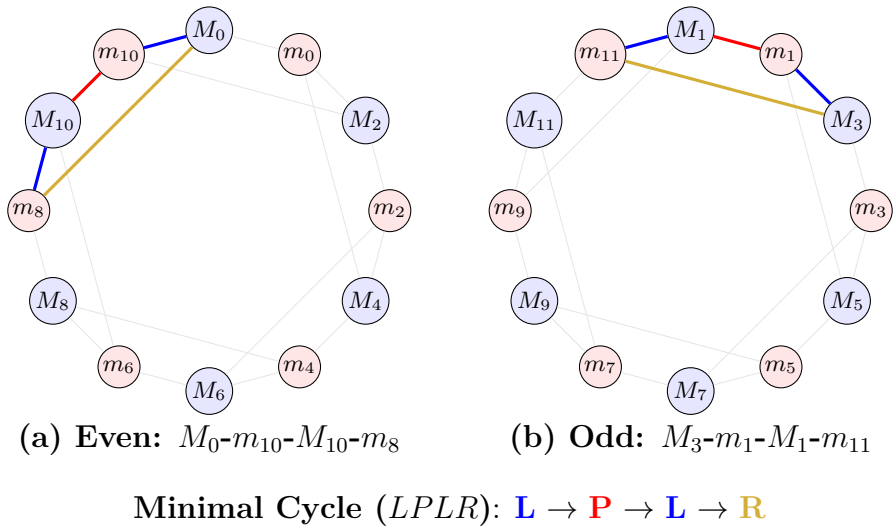

\subsubsection{Hamiltonian Cycles}
Despite the topological split, each universe (Even and Odd) remains Hamiltonian. As with the previous system, the simplest Hamiltonian cycle is formed by alternating $P$ and $L$ transformations along the perimeter of the graph (see perimeter cycles in Figure \ref{fig:wide_disconnected}).  
This confirms that while the system is disconnected globally, its subsystems retain high connectivity and navigability.

\subsection{The Case \texorpdfstring{$\Delta=3$}{Delta=3}: The ``Equally Spaced'' System \texorpdfstring{$(t,s)=(6,3)$}{t=6, s=3}}

We now arrive at the most symmetrical case in our study, defined by the generator $q=9$ (a Major Sixth) partitioned into a Tritone ($t=6$) and a Minor Third ($s=3$). The structural parameter is $\Delta = 6 - 3 = 3$, identical to the ``Wide'' system discussed previously, yet the resulting topology is distinct.

\subsubsection{Explicit Enumeration and the Discovery of Identity}
The Major chord $M_k$ is defined by the set $\{k, k+6, k+9\}$, and the Minor chord $m_k$ by the set $\{k, k+3, k+9\}$. Both are subsets of the diminished seventh chord. To inspect the structure, we list the pitch sets explicitly in Figure \ref{fig:chords_delta3_equal_explicit}.

\begin{figure}[H]
\centering
\scriptsize
\begin{minipage}{0.45\textwidth}
\centering
\textbf{Major Chords ($M_k$)}
\begin{tabular}{ccc}
\toprule
Name & Root & \textbf{Set} \\
\midrule
$M_0$ & 0 & \textbf{\{0, 6, 9\}} \\
$M_1$ & 1 & \{1, 7, 10\} \\
$M_2$ & 2 & \{2, 8, 11\} \\
$M_3$ & 3 & \{3, 9, 0\} \\
$M_4$ & 4 & \{4, 10, 1\} \\
$M_5$ & 5 & \{5, 11, 2\} \\
$M_6$ & 6 & \{6, 0, 3\} \\
$M_7$ & 7 & \{7, 1, 4\} \\
$M_8$ & 8 & \{8, 2, 5\} \\
$M_9$ & 9 & \{9, 3, 6\} \\
$M_{10}$ & 10 & \{10, 4, 7\} \\
$M_{11}$ & 11 & \{11, 5, 8\} \\
\bottomrule
\end{tabular}
\end{minipage}
\hfill
\begin{minipage}{0.45\textwidth}
\centering
\textbf{Minor Chords ($m_k$)}
\begin{tabular}{ccc}
\toprule
Name & Root & \textbf{Set} \\
\midrule
$m_0$ & 0 & \{0, 3, 9\} \\
$m_1$ & 1 & \{1, 4, 10\} \\
$m_2$ & 2 & \{2, 5, 11\} \\
$m_3$ & 3 & \{3, 6, 0\} \\
$m_4$ & 4 & \{4, 7, 1\} \\
$m_5$ & 5 & \{5, 8, 2\} \\
$m_6$ & 6 & \{6, 9, 3\} \\
$m_7$ & 7 & \{7, 10, 4\} \\
$m_8$ & 8 & \{8, 11, 5\} \\
$m_9$ & 9 & \textbf{\{9, 0, 6\}} \\
$m_{10}$ & 10 & \{10, 1, 7\} \\
$m_{11}$ & 11 & \{11, 2, 8\} \\
\bottomrule
\end{tabular}
\end{minipage}
\caption{Explicit pitch sets of the ``Equally Spaced'' system ($q=9$). Note the highlighted entries: $M_0$ and $m_9$ are identical sets.}
\label{fig:chords_delta3_equal_explicit}
\end{figure}

\subsubsection{Theorem of Modal Degeneracy}
Just as in the $(5,2)$ system, inspection reveals a collision of sets.
\begin{itemize}
    \item $M_0 = \{0, 6, 9\}$
    \item $m_9 = \{9, 0, 6\}$
\end{itemize}
This identity ($L$-transformation) is general.

\begin{theorem}[Modal Identity in $q=9$]
In the 12-TET system with $(t=6, s=3)$, the set of pitch classes forming a major chord $M_k$ is identical to the set forming the minor chord $m_{k-3}$ (or $m_{k+9}$).
\[ M_k \equiv m_{k-3} \pmod{12}. \]
\end{theorem}

\begin{proof}
$M_k = \{k, k+6, k+9\}$.
$m_{k-3} = \{(k-3), (k-3)+3, (k-3)+9\} = \{k-3, k, k+6\}$.
Since $k-3 \equiv k+9 \pmod{12}$, the sets are identical.
\end{proof}

\subsubsection{Musical Interpretation of Modal Degeneracy}
This degeneracy creates an ambiguity similar to the one observed in the $(5,2)$ case, but even more symmetric due to the properties of the diminished seventh chord.
\begin{itemize}
    \item **The Diminished Subset:** Every chord in this universe is a subset of size 3 of a Diminished Seventh chord (size 4). Specifically, $M_k$ is the diminished triad with the ``missing'' note at $k+3$, while $m_k$ is the diminished triad with the missing note at $k+6$.
    \item **Enharmonic Pivot:** The $L$-transformation corresponds to shifting the root by a Minor Third ($-3$). Since the chord is symmetric, this shift is physically imperceptible without context.
\end{itemize}

\subsubsection{The Degenerate Tonnetz and the Six Universes}
We establish the connectivity. The transformations are:
\begin{itemize}
    \item **Parallel ($P$):** Preserves $q=9$. $M_k \leftrightarrow m_k$.
    \item **Relative ($R$):** Preserves $t=6$. $M_k \leftrightarrow m_{k+6}$.
    \item **Leading-Tone ($L$):** Preserves $s=3$. $M_k \leftrightarrow m_{k-3}$.
\end{itemize}
Because $M_k \equiv m_{k-3}$, the $L$ transformation acts as an identity loop in the strict set-theoretical graph.

Tables \ref{tab:neighbors_M_6_3} and \ref{tab:neighbors_m_6_3} list the neighbors. The $L$ column is crossed out.

\begin{table}[H]
    \centering
    \scriptsize
    \renewcommand{\arraystretch}{1.3}
    
    \begin{minipage}[t]{0.48\textwidth}
        \centering
        \begin{tabular}{|c||c|c|c|}
            \hline
            \textbf{Chord} & \textbf{P-neigh.} & \textbf{L-neigh.} & \textbf{R-neigh.} \\[-0.5ex]
             & $(m_k)$ & $(m_{k-3})$ & $(m_{k+6})$ \\
            \hline
            $M_0$ & $m_0$ & \tikz[baseline=-3pt]{\node[inner sep=1pt] (A) {$m_9$}; \draw[red, thick] (A.north west) -- (A.south east) (A.north east) -- (A.south west);} & $m_6$ \\
            $M_1$ & $m_1$ & \tikz[baseline=-3pt]{\node[inner sep=1pt] (A) {$m_{10}$}; \draw[red, thick] (A.north west) -- (A.south east) (A.north east) -- (A.south west);} & $m_7$ \\
            $M_2$ & $m_2$ & \tikz[baseline=-3pt]{\node[inner sep=1pt] (A) {$m_{11}$}; \draw[red, thick] (A.north west) -- (A.south east) (A.north east) -- (A.south west);} & $m_8$ \\
            $M_3$ & $m_3$ & \tikz[baseline=-3pt]{\node[inner sep=1pt] (A) {$m_0$}; \draw[red, thick] (A.north west) -- (A.south east) (A.north east) -- (A.south west);} & $m_9$ \\
            $M_4$ & $m_4$ & \tikz[baseline=-3pt]{\node[inner sep=1pt] (A) {$m_1$}; \draw[red, thick] (A.north west) -- (A.south east) (A.north east) -- (A.south west);} & $m_{10}$ \\
            $M_5$ & $m_5$ & \tikz[baseline=-3pt]{\node[inner sep=1pt] (A) {$m_2$}; \draw[red, thick] (A.north west) -- (A.south east) (A.north east) -- (A.south west);} & $m_{11}$ \\
            $M_6$ & $m_6$ & \tikz[baseline=-3pt]{\node[inner sep=1pt] (A) {$m_3$}; \draw[red, thick] (A.north west) -- (A.south east) (A.north east) -- (A.south west);} & $m_0$ \\
            $M_7$ & $m_7$ & \tikz[baseline=-3pt]{\node[inner sep=1pt] (A) {$m_4$}; \draw[red, thick] (A.north west) -- (A.south east) (A.north east) -- (A.south west);} & $m_1$ \\
            $M_8$ & $m_8$ & \tikz[baseline=-3pt]{\node[inner sep=1pt] (A) {$m_5$}; \draw[red, thick] (A.north west) -- (A.south east) (A.north east) -- (A.south west);} & $m_2$ \\
            $M_9$ & $m_9$ & \tikz[baseline=-3pt]{\node[inner sep=1pt] (A) {$m_6$}; \draw[red, thick] (A.north west) -- (A.south east) (A.north east) -- (A.south west);} & $m_3$ \\
            $M_{10}$ & $m_{10}$ & \tikz[baseline=-3pt]{\node[inner sep=1pt] (A) {$m_7$}; \draw[red, thick] (A.north west) -- (A.south east) (A.north east) -- (A.south west);} & $m_4$ \\
            $M_{11}$ & $m_{11}$ & \tikz[baseline=-3pt]{\node[inner sep=1pt] (A) {$m_8$}; \draw[red, thick] (A.north west) -- (A.south east) (A.north east) -- (A.south west);} & $m_5$ \\
            \hline
        \end{tabular}
        \caption{Neighbors of Major chords ($q=9$). The $L$ column is crossed out.}
        \label{tab:neighbors_M_6_3}
    \end{minipage}
    \hfill
    \begin{minipage}[t]{0.48\textwidth}
        \centering
        \begin{tabular}{|c||c|c|c|}
            \hline
            \textbf{Chord} & \textbf{P-neigh.} & \textbf{L-neigh.} & \textbf{R-neigh.} \\[-0.5ex]
             & $(M_k)$ & $(M_{k+3})$ & $(M_{k+6})$ \\
            \hline
            $m_0$ & $M_0$ & \tikz[baseline=-3pt]{\node[inner sep=1pt] (A) {$M_3$}; \draw[red, thick] (A.north west) -- (A.south east) (A.north east) -- (A.south west);} & $M_6$ \\
            $m_1$ & $M_1$ & \tikz[baseline=-3pt]{\node[inner sep=1pt] (A) {$M_4$}; \draw[red, thick] (A.north west) -- (A.south east) (A.north east) -- (A.south west);} & $M_7$ \\
            $m_2$ & $M_2$ & \tikz[baseline=-3pt]{\node[inner sep=1pt] (A) {$M_5$}; \draw[red, thick] (A.north west) -- (A.south east) (A.north east) -- (A.south west);} & $M_8$ \\
            $m_3$ & $M_3$ & \tikz[baseline=-3pt]{\node[inner sep=1pt] (A) {$M_6$}; \draw[red, thick] (A.north west) -- (A.south east) (A.north east) -- (A.south west);} & $M_9$ \\
            $m_4$ & $M_4$ & \tikz[baseline=-3pt]{\node[inner sep=1pt] (A) {$M_7$}; \draw[red, thick] (A.north west) -- (A.south east) (A.north east) -- (A.south west);} & $M_{10}$ \\
            $m_5$ & $M_5$ & \tikz[baseline=-3pt]{\node[inner sep=1pt] (A) {$M_8$}; \draw[red, thick] (A.north west) -- (A.south east) (A.north east) -- (A.south west);} & $M_{11}$ \\
            $m_6$ & $M_6$ & \tikz[baseline=-3pt]{\node[inner sep=1pt] (A) {$M_9$}; \draw[red, thick] (A.north west) -- (A.south east) (A.north east) -- (A.south west);} & $M_0$ \\
            $m_7$ & $M_7$ & \tikz[baseline=-3pt]{\node[inner sep=1pt] (A) {$M_{10}$}; \draw[red, thick] (A.north west) -- (A.south east) (A.north east) -- (A.south west);} & $M_1$ \\
            $m_8$ & $M_8$ & \tikz[baseline=-3pt]{\node[inner sep=1pt] (A) {$M_{11}$}; \draw[red, thick] (A.north west) -- (A.south east) (A.north east) -- (A.south west);} & $M_2$ \\
            $m_9$ & $M_9$ & \tikz[baseline=-3pt]{\node[inner sep=1pt] (A) {$M_0$}; \draw[red, thick] (A.north west) -- (A.south east) (A.north east) -- (A.south west);} & $M_3$ \\
            $m_{10}$ & $M_{10}$ & \tikz[baseline=-3pt]{\node[inner sep=1pt] (A) {$M_1$}; \draw[red, thick] (A.north west) -- (A.south east) (A.north east) -- (A.south west);} & $M_4$ \\
            $m_{11}$ & $M_{11}$ & \tikz[baseline=-3pt]{\node[inner sep=1pt] (A) {$M_2$}; \draw[red, thick] (A.north west) -- (A.south east) (A.north east) -- (A.south west);} & $M_5$ \\
            \hline
        \end{tabular}
        \caption{Neighbors of Minor chords ($q=9$). The $L$ column is crossed out.}
        \label{tab:neighbors_m_6_3}
    \end{minipage}
\end{table}

The result of this degeneracy is far more severe than in the $\Delta=3$ case. Without the $L$-connection, the graph relies solely on $P$ and $R$.
\begin{itemize}
    \item $P$ connects $M_k \to m_k$.
    \item $R$ connects $m_k \to M_{k+6}$ (since $M_{k-6} \equiv M_{k+6}$).
\end{itemize}
This creates the cycle:
\[ M_k \xrightarrow{P} m_k \xrightarrow{R} M_{k+6} \xrightarrow{P} m_{k+6} \xrightarrow{R} M_{k+12} \equiv M_k \]
This is a closed loop of length 4. Since there are 24 functional chords, the universe fractures into **six disjoint components** (or ``universes''), each being a rectangle.

\begin{figure}[H]
    \centering
    \begin{tikzpicture}[scale=0.6, transform shape]
        \tikzset{majnode/.style={circle, draw=blue!80, fill=blue!5, thick, inner sep=1.5pt, minimum size=0.7cm, font=\scriptsize}}
        \tikzset{minnode/.style={circle, draw=red!80, fill=red!5, thick, inner sep=1.5pt, minimum size=0.7cm, font=\scriptsize}}
        \definecolor{mygold}{RGB}{212, 175, 55}
        
        \def\dx{8}  
        \def\dy{-6.0} 
        \def\w{1.0}   
        \def\h{1.8}   

        
        \begin{scope}[xshift=0cm, yshift=0cm, rotate=45]
            \coordinate (T) at (\w, \h);   
            \coordinate (L) at (-\w, \h);  
            \coordinate (B) at (-\w, -\h); 
            \coordinate (R) at (\w, -\h);  
            
            \draw[red, thick] (T) -- (L);
            \draw[blue, thick] (L) -- (B);
            \draw[red, thick] (B) -- (R);
            \draw[blue, thick] (R) -- (T);
            
            \node[rotate=-45, font=\tiny, fill=white, inner sep=0.5pt, text=black] at (0, \h) {P};
            \node[rotate=-45, font=\tiny, fill=white, inner sep=0.5pt, text=black] at (-\w, 0) {R};
            \node[rotate=-45, font=\tiny, fill=white, inner sep=0.5pt, text=black] at (0, -\h) {P};
            \node[rotate=-45, font=\tiny, fill=white, inner sep=0.5pt, text=black] at (\w, 0) {R};

            \node[minnode, rotate=-45] at (T) {$m_{10}$};
            \node[majnode, rotate=-45] at (L) {$M_{10}$};
            \node[minnode, rotate=-45] at (B) {$m_{4}$};
            \node[majnode, rotate=-45] at (R) {$M_{4}$};
        \end{scope}
        \node[font=\bfseries\scriptsize] at (0, 0) {Universe 1.b};

        \begin{scope}[xshift=0cm, yshift=\dy cm, rotate=-45]
            \coordinate (T) at (-\w, \h);  
            \coordinate (L) at (-\w, -\h); 
            \coordinate (B) at (\w, -\h);  
            \coordinate (R) at (\w, \h);   
            
            \draw[blue, thick] (T) -- (L);
            \draw[red, thick] (L) -- (B);
            \draw[blue, thick] (B) -- (R);
            \draw[red, thick] (R) -- (T);

            \node[rotate=45, font=\tiny, fill=white, inner sep=0.5pt, text=black] at (-\w, 0) {R};
            \node[rotate=45, font=\tiny, fill=white, inner sep=0.5pt, text=black] at (0, -\h) {P};
            \node[rotate=45, font=\tiny, fill=white, inner sep=0.5pt, text=black] at (\w, 0) {R};
            \node[rotate=45, font=\tiny, fill=white, inner sep=0.5pt, text=black] at (0, \h) {P};
            
            \node[majnode, rotate=45] at (T) {$M_{1}$};
            \node[minnode, rotate=45] at (L) {$m_{7}$};
            \node[majnode, rotate=45] at (B) {$M_{7}$};
            \node[minnode, rotate=45] at (R) {$m_{1}$};
        \end{scope}
        \node[font=\bfseries\scriptsize] at (0, \dy 0) {Universe 1.a};
        
        \node[font=\bfseries\small] at (0, \dy - 3.0) {Component 1};


        \begin{scope}[xshift=\dx cm, yshift=0cm, rotate=45]
            \coordinate (T) at (\w, \h);   
            \coordinate (L) at (-\w, \h);  
            \coordinate (B) at (-\w, -\h); 
            \coordinate (R) at (\w, -\h);  
            
            \draw[red, thick] (T) -- (L);
            \draw[blue, thick] (L) -- (B);
            \draw[red, thick] (B) -- (R);
            \draw[blue, thick] (R) -- (T);
            
            \node[rotate=-45, font=\tiny, fill=white, inner sep=0.5pt, text=black] at (0, \h) {P};
            \node[rotate=-45, font=\tiny, fill=white, inner sep=0.5pt, text=black] at (-\w, 0) {R};
            \node[rotate=-45, font=\tiny, fill=white, inner sep=0.5pt, text=black] at (0, -\h) {P};
            \node[rotate=-45, font=\tiny, fill=white, inner sep=0.5pt, text=black] at (\w, 0) {R};

            \node[minnode, rotate=-45] at (T) {$m_{9}$};
            \node[majnode, rotate=-45] at (L) {$M_{9}$};
            \node[minnode, rotate=-45] at (B) {$m_{3}$};
            \node[majnode, rotate=-45] at (R) {$M_{3}$};
        \end{scope}
        \node[font=\bfseries\scriptsize] at (\dx, 0) {Universe 0.b};

        \begin{scope}[xshift=\dx cm, yshift=\dy cm, rotate=-45]
            \coordinate (T) at (-\w, \h);  
            \coordinate (L) at (-\w, -\h); 
            \coordinate (B) at (\w, -\h);  
            \coordinate (R) at (\w, \h);   
            
            \draw[blue, thick] (T) -- (L);
            \draw[red, thick] (L) -- (B);
            \draw[blue, thick] (B) -- (R);
            \draw[red, thick] (R) -- (T);
            
            \node[rotate=45, font=\tiny, fill=white, inner sep=0.5pt, text=black] at (-\w, 0) {R};
            \node[rotate=45, font=\tiny, fill=white, inner sep=0.5pt, text=black] at (0, -\h) {P};
            \node[rotate=45, font=\tiny, fill=white, inner sep=0.5pt, text=black] at (\w, 0) {R};
            \node[rotate=45, font=\tiny, fill=white, inner sep=0.5pt, text=black] at (0, \h) {P};

            \node[majnode, rotate=45] at (T) {$M_{0}$};
            \node[minnode, rotate=45] at (L) {$m_{6}$};
            \node[majnode, rotate=45] at (B) {$M_{6}$};
            \node[minnode, rotate=45] at (R) {$m_{0}$};
        \end{scope}
        \node[font=\bfseries\scriptsize] at (\dx, \dy ) {Universe 0.a};

        \node[font=\bfseries\small] at (\dx, \dy - 3.0) {Component 0};


        \begin{scope}[xshift=2*\dx cm, yshift=0cm, rotate=45]
            \coordinate (T) at (\w, \h);   
            \coordinate (L) at (-\w, \h);  
            \coordinate (B) at (-\w, -\h); 
            \coordinate (R) at (\w, -\h);  
            
            \draw[red, thick] (T) -- (L);
            \draw[blue, thick] (L) -- (B);
            \draw[red, thick] (B) -- (R);
            \draw[blue, thick] (R) -- (T);
            
            \node[rotate=-45, font=\tiny, fill=white, inner sep=0.5pt, text=black] at (0, \h) {P};
            \node[rotate=-45, font=\tiny, fill=white, inner sep=0.5pt, text=black] at (-\w, 0) {R};
            \node[rotate=-45, font=\tiny, fill=white, inner sep=0.5pt, text=black] at (0, -\h) {P};
            \node[rotate=-45, font=\tiny, fill=white, inner sep=0.5pt, text=black] at (\w, 0) {R};

            \node[minnode, rotate=-45] at (T) {$m_{11}$};
            \node[majnode, rotate=-45] at (L) {$M_{11}$};
            \node[minnode, rotate=-45] at (B) {$m_{5}$};
            \node[majnode, rotate=-45] at (R) {$M_{5}$};
        \end{scope}
        \node[font=\bfseries\scriptsize] at (2*\dx, 0) {Universe 2.b};

        \begin{scope}[xshift=2*\dx cm, yshift=\dy cm, rotate=-45]
            \coordinate (T) at (-\w, \h);  
            \coordinate (L) at (-\w, -\h); 
            \coordinate (B) at (\w, -\h);  
            \coordinate (R) at (\w, \h);   
            
            \draw[blue, thick] (T) -- (L);
            \draw[red, thick] (L) -- (B);
            \draw[blue, thick] (B) -- (R);
            \draw[red, thick] (R) -- (T);
            
            \node[rotate=45, font=\tiny, fill=white, inner sep=0.5pt, text=black] at (-\w, 0) {R};
            \node[rotate=45, font=\tiny, fill=white, inner sep=0.5pt, text=black] at (0, -\h) {P};
            \node[rotate=45, font=\tiny, fill=white, inner sep=0.5pt, text=black] at (\w, 0) {R};
            \node[rotate=45, font=\tiny, fill=white, inner sep=0.5pt, text=black] at (0, \h) {P};

            \node[majnode, rotate=45] at (T) {$M_{2}$};
            \node[minnode, rotate=45] at (L) {$m_{8}$};
            \node[majnode, rotate=45] at (B) {$M_{8}$};
            \node[minnode, rotate=45] at (R) {$m_{2}$};
        \end{scope}
        \node[font=\bfseries\scriptsize] at (2*\dx, \dy 0) {Universe 2.a};

        \node[font=\bfseries\small] at (2*\dx, \dy - 3.0) {Component 2};

    \end{tikzpicture}
    \caption{The degenerate geometry of the $(t=6, s=3)$ system. The harmonic space fractures into six disjoint rectangles.}
    \label{fig:six_universes_rect_final_corrected}
\end{figure}
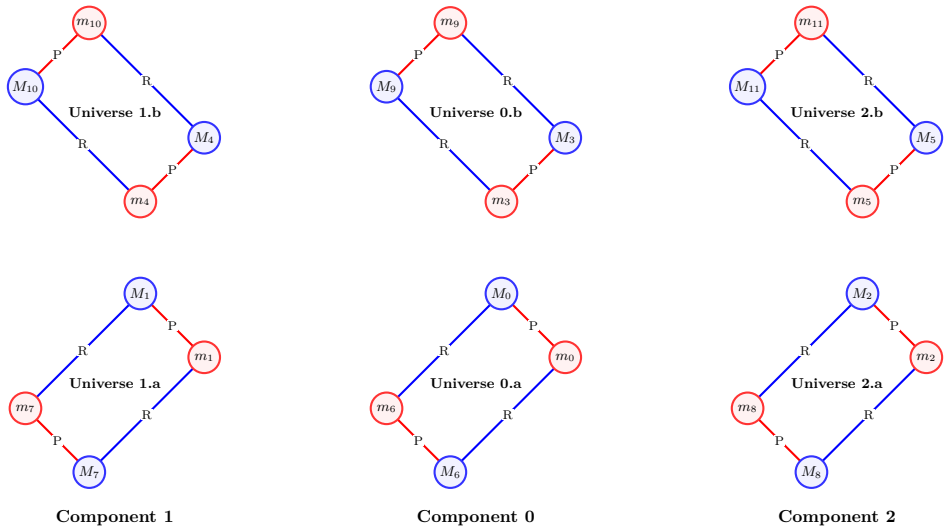

\subsubsection{Restoring the Connection}
To recover the full harmonic system, we must "uncross" the $L$-column in the neighbor tables. Restoring the $L$-transformation ($L(M_k) = m_{k-3}$) creates a bridge between these previously isolated universes. 

For example, $L$ connects $M_0$ (from **Universe 0.a**) to $m_9$ (from **Universe 0.b**). Since both $0$ and $9$ are multiples of 3, the restoration of $L$ merges the six small rectangles into larger structures based on root congruences modulo 3. This leads us directly to the tri-partitioned topology.

\begin{table}[H]
\centering
\scriptsize
\renewcommand{\arraystretch}{1.2}
\begin{minipage}{0.48\textwidth}
\centering
\begin{tabular}{|c||c|c|c|}
\hline
\textbf{Chord} ($M_k$) & \textbf{P} ($m_k$) & \textbf{L} ($m_{k-3}$) & \textbf{R} ($m_{k+6}$) \\
\hline
$M_0$ & $m_0$ & $m_{9}$ & $m_6$ \\
$M_1$ & $m_1$ & $m_{10}$ & $m_7$ \\
$M_2$ & $m_2$ & $m_{11}$ & $m_8$ \\
$M_3$ & $m_3$ & $m_{0}$ & $m_9$ \\
$M_4$ & $m_4$ & $m_{1}$ & $m_{10}$ \\
$M_5$ & $m_5$ & $m_{2}$ & $m_{11}$ \\
$M_6$ & $m_6$ & $m_{3}$ & $m_{0}$ \\
$M_7$ & $m_7$ & $m_{4}$ & $m_{1}$ \\
$M_8$ & $m_8$ & $m_{5}$ & $m_{2}$ \\
$M_9$ & $m_9$ & $m_{6}$ & $m_{3}$ \\
$M_{10}$ & $m_{7}$ & $m_{8}$ & $m_{4}$ \\
$M_{11}$ & $m_{8}$ & $m_{9}$ & $m_{5}$ \\
\hline
\end{tabular}
\caption{Voice Leading: Major Chords $M_k$.}
\end{minipage}
\hfill
\begin{minipage}{0.48\textwidth}
\centering
\begin{tabular}{|c||c|c|c|}
\hline
\textbf{Source} ($m_k$) & \textbf{P} ($M_k$) & \textbf{L} ($M_{k+3}$) & \textbf{R} ($M_{k-6}$) \\
\hline
$m_0$ & $M_0$ & $M_3$ & $M_6$ \\
$m_1$ & $M_1$ & $M_4$ & $M_7$ \\
$m_2$ & $M_2$ & $M_5$ & $M_{8}$ \\
$m_3$ & $M_3$ & $M_6$ & $M_{9}$ \\
$m_4$ & $M_4$ & $M_7$ & $M_{10}$ \\
$m_5$ & $M_5$ & $M_8$ & $M_{11}$ \\
$m_6$ & $M_6$ & $M_9$ & $M_{0}$ \\
$m_7$ & $M_7$ & $M_{10}$ & $M_{1}$ \\
$m_8$ & $M_8$ & $M_{11}$ & $M_{2}$ \\
$m_9$ & $M_9$ & $M_{0}$ & $M_{3}$ \\
$m_{10}$ & $M_{10}$ & $M_{1}$ & $M_{4}$ \\
$m_{11}$ & $M_{11}$ & $M_{2}$ & $M_{5}$ \\
\hline
\end{tabular}
\caption{Voice Leading: Minor Chords $m_k$.}
\end{minipage}
\end{table}

\subsubsection{Tri-Partition of the Graph}
Since all root shifts (0, -3, +6) are multiples of 3, the Levi graph fractures into \textbf{three disjoint components}, each corresponding to a residue class modulo 3.
\begin{itemize}
    \item \textbf{Component 0:} Roots $\{0, 3, 6, 9\}$ (The ``Diminished-0'' Universe)
    \item \textbf{Component 1:} Roots $\{1, 4, 7, 10\}$ (The ``Diminished-1'' Universe)
    \item \textbf{Component 2:} Roots $\{2, 5, 8, 11\}$ (The ``Diminished-2'' Universe)
\end{itemize}
Each component contains 8 nodes (4 Major, 4 Minor) and forms a highly symmetric graph.

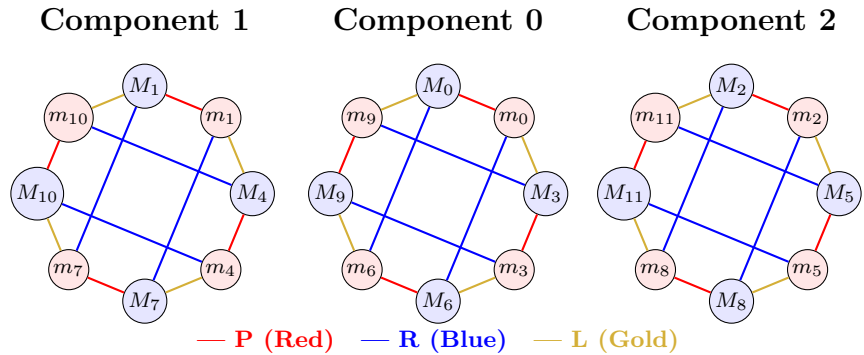
\begin{figure}[H]
    \centering
    \begin{tikzpicture}[scale=0.65]
        \tikzset{mynode/.style={circle, draw, thin, inner sep=1pt, font=\tiny}}
        \definecolor{mygold}{RGB}{212, 175, 55}

        \begin{scope}[xshift=0cm]
            \node at (0, 3.5) {\textbf{Component 0}};
            \def\R{2.2}
            \foreach \k [count=\i] in {0, 3, 6, 9} {
                \pgfmathsetmacro{\angM}{90 - (\i-1)*90}
                \pgfmathsetmacro{\angm}{\angM - 45}
                
                \node[mynode, fill=blue!10] (M\k) at (\angM:\R) {$M_{\k}$};
                \node[mynode, fill=red!10] (m\k) at (\angm:\R) {$m_{\k}$};
            }
            \foreach \k in {0, 3, 6, 9} {
                \draw[red, thick] (M\k) -- (m\k);
                \pgfmathsetmacro{\prev}{int(mod(\k-3+12,12))}
                \draw[mygold, thick] (M\k) -- (m\prev);
                \pgfmathsetmacro{\cross}{int(mod(\k+6,12))}
                \draw[blue, thick] (M\k) -- (m\cross);
            }
        \end{scope}

        \begin{scope}[xshift=-6cm]
            \node at (0, 3.5) {\textbf{Component 1}};
            \def\R{2.2}
            \foreach \k [count=\i] in {1, 4, 7, 10} {
                \pgfmathsetmacro{\angM}{90 - (\i-1)*90}
                \pgfmathsetmacro{\angm}{\angM - 45}
                \node[mynode, fill=blue!10] (M\k) at (\angM:\R) {$M_{\k}$};
                \node[mynode, fill=red!10] (m\k) at (\angm:\R) {$m_{\k}$};
            }
            \foreach \k in {1, 4, 7, 10} {
                \draw[red, thick] (M\k) -- (m\k);
                \pgfmathsetmacro{\prev}{int(mod(\k-3+12,12))}
                \draw[mygold, thick] (M\k) -- (m\prev);
                \pgfmathsetmacro{\cross}{int(mod(\k+6,12))}
                \draw[blue, thick] (M\k) -- (m\cross);
            }
        \end{scope}

        \begin{scope}[xshift=6cm]
            \node at (0, 3.5) {\textbf{Component 2}};
            \def\R{2.2}
            \foreach \k [count=\i] in {2, 5, 8, 11} {
                \pgfmathsetmacro{\angM}{90 - (\i-1)*90}
                \pgfmathsetmacro{\angm}{\angM - 45}
                \node[mynode, fill=blue!10] (M\k) at (\angM:\R) {$M_{\k}$};
                \node[mynode, fill=red!10] (m\k) at (\angm:\R) {$m_{\k}$};
            }
            \foreach \k in {2, 5, 8, 11} {
                \draw[red, thick] (M\k) -- (m\k);
                \pgfmathsetmacro{\prev}{int(mod(\k-3+12,12))}
                \draw[mygold, thick] (M\k) -- (m\prev);
                \pgfmathsetmacro{\cross}{int(mod(\k+6,12))}
                \draw[blue, thick] (M\k) -- (m\cross);
            }
        \end{scope}
        
        \node[anchor=north] at (0, -2.5) {
            \scriptsize
            \begin{tabular}{ccc}
            \textcolor{red}{\textbf{--- P (Red)}} & \textcolor{blue}{\textbf{--- R (Blue)}} & \textcolor{mygold}{\textbf{--- L (Gold)}}
            \end{tabular}
        };
    \end{tikzpicture}
    \caption{The Tri-Partitioned Levi Graph of the $(9,6,3)$ system. The graph splits into three isomorphic components, each containing 8 chords whose roots form a diminished seventh chord. The topology of each component is an octagonal ring formed by alternating $P$ and $L$ edges, with $R$ edges connecting opposite sides.}
    \label{fig:equal_split}
\end{figure}

\subsubsection{Minimal Cycles and Girth}
Within any single component, the girth is 4. This is defined by the rectangle formed by alternating $P$ and $R$ transformations:
\[ M_k \xrightarrow{P} m_k \xrightarrow{R} M_{k-6} \xrightarrow{P} m_{k-6} \xrightarrow{R} M_k \]
Since $k-12 \equiv k$, this loop closes in 4 steps. The $L$ transformation does not participate in 4-cycles, but contributes to form the perimeter.

\begin{figure}[H]
    \centering
    \begin{tikzpicture}[scale=0.8, >=stealth]
        \tikzset{mynode/.style={circle, draw, thin, inner sep=1pt, font=\scriptsize}}
        \definecolor{mygold}{RGB}{212, 175, 55}
        
        \node[mynode, fill=blue!10] (M0) at (0, 2) {$M_0$};
        \node[mynode, fill=red!10] (m0) at (2, 2) {$m_0$};
        \node[mynode, fill=blue!10] (M6) at (2, -2) {$M_6$};
        \node[mynode, fill=red!10] (m6) at (0, -2) {$m_6$};
        
        \draw[red, very thick] (M0) -- (m0) node[midway, above] {$P$};
        \draw[blue, very thick] (m0) -- (M6) node[midway, right] {$R$};
        \draw[red, very thick] (M6) -- (m6) node[midway, below] {$P$};
        \draw[blue, very thick] (m6) -- (M0) node[midway, left] {$R$};
        
        \node at (1, -3.5) {\textbf{The ${\color{red}P}{\color{blue}R}$-Rectangle}};
    \end{tikzpicture}
    \caption{The minimal cycle in the $(9,6,3)$ system. The $P$ and $R$ transformations create a commutative square (or digon in pitch-class space) that defines the graph's girth.}
    \label{fig:equal_minimal}
\end{figure}

\subsubsection{Hamiltonicity}
Each of the three components is trivially Hamiltonian. The alternating $PL$-cycle traces the perimeter of the component, visiting all 8 nodes in the sequence:
\[ M_k \xrightarrow{P} m_k \xrightarrow{L} M_{k+3} \xrightarrow{L} m_{k+3} \xrightarrow{L} M_{k+6} \xrightarrow{L} m_{k+6} \xrightarrow{L} M_{k+9} \xrightarrow{L} m_{k+9} \xrightarrow{L} M_{k} .\]
For Component 0, this cycle is $M_0 \to m_0 \to M_3 \to m_3 \to M_6 \to m_6 \to M_9 \to m_9 \to M_0$.

\section{Levi Graphs and Configuration Theory}\label{sec8} 

In this section, we provide definitions for the geometric and graph-theoretic notions used intuitively starting from Section \ref{sec43}. As the earlier exposition was designed to be accessible to musicians we sacrificed mathematical rigor there; here we formalize these structures to place our approach to 12-TET harmonic systems on a firm mathematical footing.

To ensure clarity, we analyze detailed examples of Levi graphs $L_n(a,b,c)$, specifically focusing on the toy model $L_6(a,b,c)$ and our primary object of study, $L_{12}(a,b,c)$. The subsequent subsections provide a full classification of the cyclic Levi graphs of $L_{12}(a,b,c)$, demonstrating that there exist exactly nine non-isomorphic classes, from which eight can be given an interpretation of a 12TET musical system. Crucially, every specific harmonic system encountered in previous Sections corresponds to one of these eight fundamental classes. This self--contained expository part is an adaptation from our earlier work \cite{Nurowski10TET} on 10TET systems. 

The classification of Levi graphs $L_{12}(a,b,c)$ presented below serves as the theoretical foundation for the final part of this work. In \textbf{Section \ref{sec9}}, we will utilize these results to characterize all musical 12-TET systems resulting from our generation scheme $t+s=q$ subject to the difference constraint $t-s=\Delta$.

\subsection{Connection Sets and Levi Graphs}

\begin{definition}[Offset]
Let $n$ be a positive integer. An \textbf{offset} (or connection set) $S$ is defined as a three-element subset of the ring of integers modulo $n$, denoted by $\mathbb{Z}_n$:
\[
S = \{a, b, c\} \subset \mathbb{Z}_n, \quad |S|=3.
\]
\end{definition}

\begin{definition}[Cyclic Cubic Bipartite Levi Graph]
Given an offset $S = \{a, b, c\}$, the \textbf{Levi graph} $L_n(S)$ is a bipartite graph with vertex set $V = \mathcal{P} \cup \mathcal{L}$, where:
\[
\mathcal{P} = \{P_0, P_1, \dots, P_{n-1}\} \quad \text{and} \quad \mathcal{L} = \{L_0, L_1, \dots, L_{n-1}\}.
\]
The edge set $E$ is defined constructively using all elements of $S$:
\[
E = \{ (P_i, L_{i+s}) \mid i \in \mathbb{Z}_n, \, \forall s \in S \}.
\]
In other words, every point $P_i$ is connected to the three lines $L_{i+a}, L_{i+b}, L_{i+c}$.
\end{definition}

Historically, the set $\mathcal P$ in a graph $L_n(S)$ is called the set of its \textbf{points}, and the set $\mathcal L$ is called the set of its \textbf{lines}.

\subsection{Cycles and the Girth}
A $k\geq 4$ passage $P_{i_0}\to L_{i_1}\to P_{i_2}\to L_{i_3}\to\dots\to L_{i_{k-1}}\to P_{i_k}=P_{i_0}$ in a graph $L_n(S)$ along a loop through its edges, is called a \textbf{cycle of length} $k$. The minimal length of all cycles in the graph is called its \textbf{girth}.

If a graph $L_n(S)$ has a 4--cycle (girth 4), then two points from $\mathcal P$ are connected with two distinct lines from $\mathcal L$. This forbids a geometric interpretation for such a graph as a configuration of straight lines and points in $\mathbb R^2$, for in Euclidean geometry two points may lie on at most one straight line. So Levi graphs $L_n(S)$ with \textbf{girth equal to four} can \textbf{not} form \textbf{a geometric configuration}.

\subsection{Isomorphism and Symmetry}

\begin{definition}[Bipartite Graph Isomorphism and Symmetry]
Two Levi graphs $L_n(S)$ and $L_n(S')$ are \textbf{isomorphic}, denoted $L_n(S) \cong L_n(S')$, if there exists a bijection $\phi: V(L_n(S)) \to V(L_n(S'))$ which preserves adjacency and the bipartite structure. Specifically, $\phi$ must satisfy:
\[
u \sim v \text{ in } L_n(S) \iff \phi(u) \sim \phi(v) \text{ in } L_n(S'),
\]
and one of the following two conditions holds:
\begin{enumerate}
    \item \textbf{Structure Preserving:} $\phi(\mathcal{P}) = \mathcal{P}'$ and $\phi(\mathcal{L}) = \mathcal{L}'$.
    \item \textbf{Structure Swapping:} $\phi(\mathcal{P}) = \mathcal{L}'$ and $\phi(\mathcal{L}) = \mathcal{P}'$.
\end{enumerate}
An isomorphism $\phi$ of a graph $G=L_n(S)$ to itself is called a \textbf{symmetry} (or automorphism). The set $\mathrm{Aut}(G)$ of all symmetries forms the \textbf{symmetry group} of $L_n(S)$.
\end{definition}

\subsection{Affine Isomorphisms and Canonical Offsets}

\subsubsection{The Action of the Affine Group}

We consider the group of affine transformations on $\mathbb{Z}_n$, denoted by $\mathrm{Aff}(\mathbb{Z}_n)$. An element $g \in \mathrm{Aff}(\mathbb{Z}_n)$ is defined by a pair $(\alpha, k)$ where $\alpha \in \mathbb{Z}_n^*$ and $k \in \mathbb{Z}_n$, acting on the ring elements as:
\[
g(x) = \alpha x + k \pmod n.
\]
This group acts naturally on the set of all offsets (subsets of size 3) by applying the transformation element-wise:
\[
g \cdot \{s_1, s_2, s_3\} = \{ \alpha s_1 + k, \alpha s_2 + k, \alpha s_3 + k \}.
\]

\begin{proposition}[Affine Isomorphism]
If two offsets $S$ and $S'$ belong to the same orbit under the action of $\mathrm{Aff}(\mathbb{Z}_n)$ (i.e., $S' = \alpha S + k$), then the graphs $L_n(S)$ and $L_n(S')$ are isomorphic. Such an isomorphism is called an \textbf{affine isomorphism}.
\end{proposition}

\begin{proof}
We construct the vertex mapping $\phi: V(L_n(S)) \to V(L_n(S'))$ as follows:
\begin{align*}
\phi(P_i) &= P'_{\alpha i}, \\
\phi(L_j) &= L'_{\alpha j + k}.
\end{align*}
Consider an edge $(P_i, L_{i+s})$ in $L_n(S)$ where $s \in S$.
The image of $P_i$ is $P'_{\alpha i}$.
The image of $L_{i+s}$ is $L'_{\alpha(i+s) + k} = L'_{\alpha i + (\alpha s + k)}$.
Let $s' = \alpha s + k$. By definition, $s' \in S'$.
Thus, the edge maps to $(P'_{\alpha i}, L'_{\alpha i + s'})$, which is a valid edge in $L_n(S')$. Since $\alpha$ is invertible, this mapping is a bijection and explicitly preserves the bipartition ($\mathcal{P} \to \mathcal{P}'$, $\mathcal{L} \to \mathcal{L}'$).
\end{proof}

\subsubsection{Algorithm: Orbit Enumeration}

To classify all Levi graphs up to affine isomorphism, we must partition the set of all normalized offsets into disjoint orbits. We employ a procedure that generates the full set of normalized members for each orbit and identifies the canonical representative:

\begin{enumerate}
    \item Let $\mathcal{N}$ be the set of all normalized offsets (offsets containing 0) in $\mathbb{Z}_n$, sorted in lexicographical order.
    \item Initialize a set of visited offsets $\mathcal{V} = \emptyset$ and a list of orbit classes $\mathcal{C}$.
    \item Iterate through each offset $S \in \mathcal{N}$:
    \begin{itemize}
        \item If $S \in \mathcal{V}$, skip it (it belongs to an orbit already generated).
        \item If $S \notin \mathcal{V}$, then $S$ is the \textbf{canonical representative} of a new orbit.
        \item \textbf{Generate the full affine orbit} $\mathcal{O}_S$ containing all normalized equivalents of $S$:
        \begin{enumerate}
             \item Initialize the orbit set $\mathcal{O}_S = \emptyset$.
             \item For every unit $\alpha \in \mathbb{Z}_n^*$:
             \begin{enumerate}
                  \item Compute the scaled set $S_\alpha = \{0, \alpha x, \alpha y\}$.
                  \item Generate the three normalized forms of $S_\alpha$ by shifting each of its elements to 0:
                  \begin{itemize}
                      \item The set $S_\alpha$ itself (already contains 0).
                      \item The set shifted by $-\alpha x$: $\{0, -\alpha x, \alpha y - \alpha x\}$.
                      \item The set shifted by $-\alpha y$: $\{0, -\alpha y, \alpha x - \alpha y\}$.
                  \end{itemize}
                  \item Sort the elements of each form to ensure standard representation (increasing order) and add them to $\mathcal{O}_S$.
             \end{enumerate}
        \end{enumerate}
        \item Record the class as $(S, \mathcal{O}_S)$ in $\mathcal{C}$ and mark all elements in $\mathcal{O}_S$ as visited (add to $\mathcal{V}$).
    \end{itemize}
\end{enumerate}
The final list $\mathcal{C}$ contains the disjoint affine equivalence classes, each fully listed and represented by its lexicographically smallest member.

\subsubsection{Non-Affine Isomorphisms}
It is important to note that affine equivalence does not capture all possible isomorphisms. There exist pairs of Levi graphs which are \textbf{isomorphic} but \textbf{affinely nonequivalent} (their canonical offsets distinct). Such "exceptional" isomorphisms correspond to combinatorial symmetries that do not arise from the cyclic structure of $\mathbb{Z}_n$. However, it is worthwhile to mention that for square-free orders $n$ of $\mathbb Z_n$, like $n=6$ and $n=12$, it is known \cite{muzychuk1995} that the affine classification coincides with the full isomorphism classification.

\subsection{Examples: cases of   $\mathbb{Z}_3$,  $\mathbb{Z}_4$,  $\mathbb{Z}_6$ and $\mathbb{Z}_{12}$}

\subsubsection{The case  $\mathbb{Z}_3$}
There is one normalized offset in this case, namely $S=\{0,1,2\}$, and thus there is only one class of affinely equivalent cyclic bipartite cubic Levi graphs in this case. Thus, here the only Levi graph is
$${\bf L_3(0,1,2)},$$
  with the following visualization:

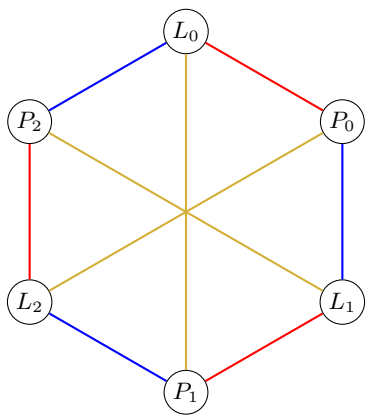
\begin{figure}[H]
    \centering
    \begin{tikzpicture}[scale=1.2]
        \definecolor{mygold}{RGB}{212, 175, 55}
        \tikzset{lnode/.style={circle, draw=black, fill=white, inner sep=1.5pt, font=\scriptsize}}
        \tikzset{pnode/.style={circle, draw=black, fill=white, inner sep=1.5pt, text=black, font=\scriptsize}}
        
        \def\n{3}
        \def\a{0} \def\b{1} \def\c{2}
        \def\R{2.0}
        
        \foreach \i in {0,...,2} {
            \pgfmathsetmacro{\angL}{90 - \i * (360/\n)}
            \pgfmathsetmacro{\angP}{\angL - (180/\n)}
            
            \node[lnode] (L\i) at (\angL:\R) {$L_{\i}$};
            \node[pnode] (P\i) at (\angP:\R) {$P_{\i}$};
        }
        
        \foreach \i in {0,...,2} {
            \pgfmathsetmacro{\idxA}{int(mod(\i+\a, \n))}
            \pgfmathsetmacro{\idxB}{int(mod(\i+\b, \n))}
            \pgfmathsetmacro{\idxC}{int(mod(\i+\c, \n))}
            
            \draw[red, thick] (P\i) -- (L\idxA);
            \draw[blue, thick] (P\i) -- (L\idxB);
            \draw[mygold, thick] (P\i) -- (L\idxC);
        }
    \end{tikzpicture}
    \caption{The unique Levi graph for $\mathbb{Z}_3$, $L_3(0,1,2)$. The nodes $L_i$ and $P_i$ alternate on the circle. Edges correspond to offset components: $0$ (Red), $1$ (Blue), and $2$ (Gold).}
    \label{fig:L4_graph}
\end{figure}
It is known that this graph has automorphism {\bf group order $72$}. 
  
\subsubsection{The case  $\mathbb{Z}_4$}
There is precisely one cyclic cubic bipartite Levi graph $L_4(a,b,c)$. As in this case there are only three normalized offsets $\{0,1,2\}$, $\{0,1,3\}$ and $\{0,2,3\}$ one sees that if ${\bf \{0,1,2\}}$ is shifted by (-1), then $\{0,1,2\}\sim\{-1,0,1\}=\{3,0,1\}$, which when sorted gives ${\bf \{0,1,3\}}$, and if ${\bf \{0,1,2\}}$ is shifted by (-2), then $\{0,1,2\}\sim\{-2,-1,0\}=\{2,3,0\}$, which when sorted gives ${\bf\{ 0,2,3\}}$. Thus all \textbf{3} normalized offsets belong to the same orbit.

\textbf{Result:} There is precisely \textbf{one} cyclic cubic bipartite Levi graphs for $n=4$: $$\mathbf{L_4(0,1,2)}.$$ 

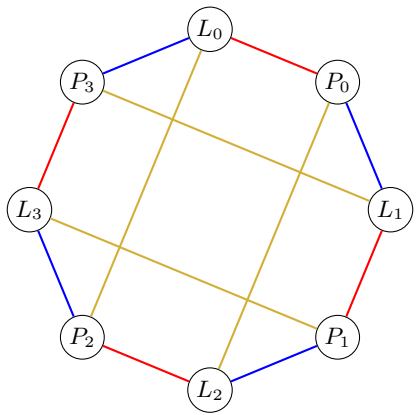
\begin{figure}[H]
    \centering
    \begin{tikzpicture}[scale=1.2]
        \definecolor{mygold}{RGB}{212, 175, 55}
        \tikzset{lnode/.style={circle, draw=black, fill=white, inner sep=1.5pt, font=\scriptsize}}
        \tikzset{pnode/.style={circle, draw=black, fill=white, inner sep=1.5pt, text=black, font=\scriptsize}}
        
        \def\n{4}
        \def\a{0} \def\b{1} \def\c{2}
        \def\R{2.0}
        
        \foreach \i in {0,...,3} {
            \pgfmathsetmacro{\angL}{90 - \i * (360/\n)}
            \pgfmathsetmacro{\angP}{\angL - (180/\n)}
            
            \node[lnode] (L\i) at (\angL:\R) {$L_{\i}$};
            \node[pnode] (P\i) at (\angP:\R) {$P_{\i}$};
        }
        
        \foreach \i in {0,...,3} {
            \pgfmathsetmacro{\idxA}{int(mod(\i+\a, \n))}
            \pgfmathsetmacro{\idxB}{int(mod(\i+\b, \n))}
            \pgfmathsetmacro{\idxC}{int(mod(\i+\c, \n))}
            
            \draw[red, thick] (P\i) -- (L\idxA);
            \draw[blue, thick] (P\i) -- (L\idxB);
            \draw[mygold, thick] (P\i) -- (L\idxC);
        }
    \end{tikzpicture}
    \caption{The unique Levi graph for $\mathbb{Z}_4$, $L_4(0,1,2)$. The nodes $L_i$ and $P_i$ alternate on the circle. Edges correspond to offset components: $0$ (Red), $1$ (Blue), and $2$ (Gold).}
\end{figure}
It is known that this graph has automorphism {\bf group order $48$}.
\subsubsection{The case  $\mathbb{Z}_6$}

\paragraph{Detailed Derivation of Affine Orbits}

Let $n=6$. The group of units is $\mathbb{Z}_6^* = \{1, 5\}$.
The total number of normalized offsets in $\mathbb{Z}_6$ is $\binom{5}{2} = 10$. The set of all normalized offsets is:
\[ \mathcal{N} = \big\{ \{0,1,2\}, \{0,1,3\}, \{0,1,4\}, \{0,1,5\}, \{0,2,3\}, \{0,2,4\}, \{0,2,5\}, \{0,3,4\}, \{0,3,5\}, \{0,4,5\} \big\}. \]

We apply our algorithm step-by-step.

\paragraph{Step 1:} We pick the first offset $S = \{0, 1, 3\}$. This is our first canonical offset.

\paragraph{Step 2:} We generate the orbit of $S=\{0,1,3\}$ by applying all $\alpha \in \mathbb{Z}_6^*$ and normalizing the results.

\begin{itemize}
    \item \textbf{Case $\alpha = 1$:} $S_1 = {\bf \{0, 1, 3\}}$.
    \begin{itemize}
        \item Shift by 0: $\{0, 1, 3\}$.
        \item Shift by $-1$: $\{ -1, 0, 2 \} \equiv \{5, 0, 2\}$. Sorted: \textbf{\{0, 2, 5\}}.
        \item Shift by $-3$: $\{ -3, -2, 0 \} \equiv \{3, 4, 0\}$. Sorted: \textbf{\{0, 3, 4\}}.
    \end{itemize}
    
    \item \textbf{Case $\alpha = 5$:} $S_2 = 5\times \{0, 1, 3\}$.
    \begin{itemize}
        \item Shift by 0: $\{0, 5, 15\}\sim \{0,5,3\}$. Sorted: \textbf{\{0, 3, 5\}}.
        \item Shift by $-3$: $\{ -3, 0, 2 \} \equiv \{3, 0, 2\}$. Sorted: \textbf{\{0, 2, 3\}}.
        \item Shift by $-5$: $\{ -5, -2, 0 \} \equiv \{1, 4, 0\}$. Sorted: \textbf{\{0, 1, 4\}}.
    \end{itemize}
\end{itemize}

\paragraph{Conclusion:} 
The orbit of $\{0,1,3\}$ is:
\[ \text{Orbit}_{{\bf \{0,1,3\}}} = \{ {\bf \{0,1,3\}}, \{0,1,4\}, \{0,2,3\}, \{0,2,5\}, \{0,3,4\}, \{0,3,5\} \}, \]
and it contains exactly \textbf{6} normalized offsets in $\mathbb{Z}_6$.

\paragraph{We iterate:}
We now take an offset ${\bf \{0,1,2\}}$ and applying our procedure we first obtain: $\{-1,0,1\}=\{5,0,1\}\sim {\bf\{0,1,5\}}$; then we obtain: $\{-2,-1,0\}=\{4,5,0\}\sim{\bf\{0,4,5\}}$. Furtheremore, since multiplication by 5 gives: $5\times\{0,1,2\}=\{0,5,4\}\sim\{0,4,5\}$, we see that there are no other offsets in this orbit. 
\paragraph{Conclusion:} 
The orbit of $\{0,1,2\}$ is:
\[ \text{Orbit}_{{\bf \{0,1,2\}}} = \{ {\bf \{0,1,2\}}, \{0,1,5\}, \{0,4,5\} \}.\]
It has \textbf{3} elements. Since 10-6-3=1, the orbit of the only left offset $S_3={\bf \{0,2,4\}}$ consists \textbf{1} element only, anf hence 
\[ \text{Orbit}_{{\bf \{0,2,4\}}} = \{{\bf  \{0,2,4\}} \}. \]

\textbf{Result:} There are precisely \textbf{three} cyclic cubic bipartite Levi graphs for $n=6$: $$\mathbf{L_6(0,1,2)}, \quad \mathbf{L_6(0,1,3)}\quad \&\quad  \mathbf{L_6(0,2,4)}.$$

\begin{figure}[H]
    \centering
    \begin{tikzpicture}[scale=1.2, transform shape]
        \definecolor{mygold}{RGB}{212, 175, 55}
        \tikzset{lnode/.style={circle, draw=black, fill=white, inner sep=1pt, font=\tiny}}
        \tikzset{pnode/.style={circle, draw=black, fill=white, inner sep=1pt, text=black, font=\tiny}}

        \newcommand{\drawLeviSix}[5]{
            \begin{scope}[xshift=#1cm]
                \def\n{6}
                \def\R{1.8}
                \foreach \i in {0,...,5} {
                    \pgfmathsetmacro{\angL}{90 - \i * (360/\n)}
                    \pgfmathsetmacro{\angP}{\angL - (180/\n)}
                    \node[lnode] (L\i) at (\angL:\R) {$L_{\i}$};
                    \node[pnode] (P\i) at (\angP:\R) {$P_{\i}$};
                }
                \foreach \i in {0,...,5} {
                    \pgfmathsetmacro{\idxA}{int(mod(\i+#3, \n))}
                    \pgfmathsetmacro{\idxB}{int(mod(\i+#4, \n))}
                    \pgfmathsetmacro{\idxC}{int(mod(\i+#5, \n))}
                    \draw[red] (P\i) -- (L\idxA);
                    \draw[blue] (P\i) -- (L\idxB);
                    \draw[mygold] (P\i) -- (L\idxC);
                }
                \node[yshift=-2.5cm, font=\bfseries] at (0,0) {#2};
            \end{scope}
        }

        \drawLeviSix{-5}{$L_6(0,1,2)$}{0}{1}{2}
        \drawLeviSix{0}{$L_6(0,1,3)$}{0}{1}{3}
        \drawLeviSix{5}{$L_6(0,2,4)$}{0}{2}{4}

    \end{tikzpicture}
    \caption{The three non-isomorphic Levi graphs for $\mathbb{Z}_6$. Left: $L_6(0,1,2)$, Middle: $L_6(0,1,3)$, Right: $L_6(0,2,4)$.}
    \label{fig:L6_graphs}
\end{figure}
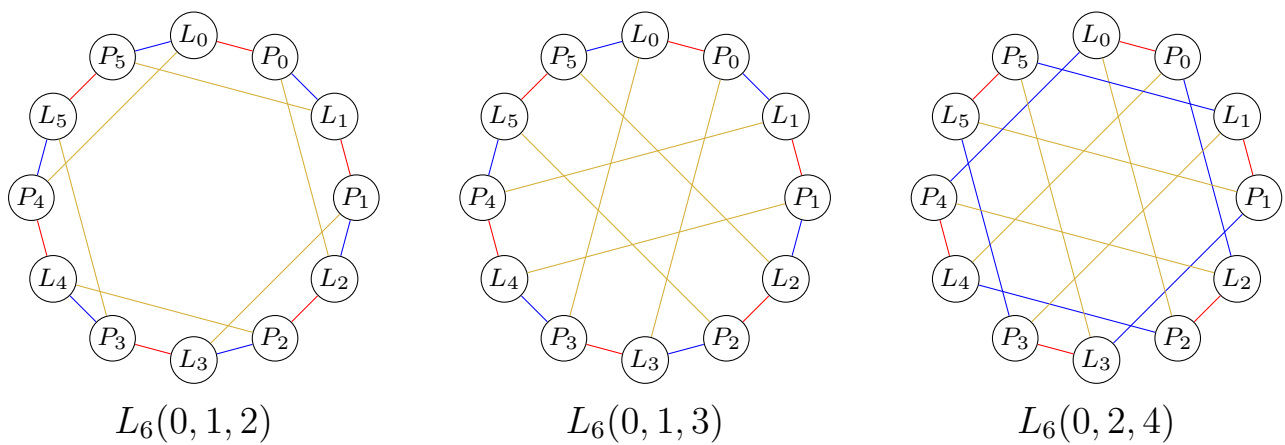

A noteworthy feature of the third graph, $L_6(0,2,4)$, distinguishes it from the previous two: it is \textbf{disconnected}. The vertex set partitions into two independent subsets that do not communicate: one containing all vertices with even indices ($L_{2k}, P_{2k}$) and another containing all vertices with odd indices ($L_{2k+1}, P_{2k+1}$). This occurs because the offset steps $\{0, 2, 4\}$ are all even numbers, making it impossible to traverse from an even node to an odd node. Geometrically, this implies the configuration space fractures into two separate, isomorphic sub-universes.

\begin{figure}[H]
    \centering
    \begin{tikzpicture}[scale=0.7, transform shape]
        \definecolor{mygold}{RGB}{212, 175, 55}
        
        \tikzset{lnode/.style={circle, draw=black, fill=white, inner sep=1.5pt, font=\scriptsize}}
        \tikzset{pnode/.style={circle, draw=black, fill=white, inner sep=1.5pt, text=black, font=\scriptsize}}

        \def\n{6}
        \def\R{2.0}
        
        \begin{scope}[xshift=-3.5cm]
            \node[font=\bfseries] at (0, -3.0) {Even Component};
            
            \foreach \i in {0, 2, 4} {
                \pgfmathsetmacro{\angL}{90 - \i * (360/\n)}
                \pgfmathsetmacro{\angP}{\angL - (180/\n)}
                
                \node[lnode] (L\i) at (\angL:\R) {$L_{\i}$};
                \node[pnode] (P\i) at (\angP:\R) {$P_{\i}$};
            }
            
            \foreach \i in {0, 2, 4} {
                \pgfmathsetmacro{\idxA}{int(mod(\i+0, \n))}
                \pgfmathsetmacro{\idxB}{int(mod(\i+2, \n))}
                \pgfmathsetmacro{\idxC}{int(mod(\i+4, \n))}
                
                \draw[red, thick] (P\i) -- (L\idxA);
                \draw[blue, thick] (P\i) -- (L\idxB);
                \draw[mygold, thick] (P\i) -- (L\idxC);
            }
        \end{scope}

        \begin{scope}[xshift=3.5cm]
            \node[font=\bfseries] at (0, -3.0) {Odd Component};
            
            \foreach \i in {1, 3, 5} {
                \pgfmathsetmacro{\angL}{90 - \i * (360/\n)}
                \pgfmathsetmacro{\angP}{\angL - (180/\n)}
                
                \node[lnode] (L\i) at (\angL:\R) {$L_{\i}$};
                \node[pnode] (P\i) at (\angP:\R) {$P_{\i}$};
            }
            
            \foreach \i in {1, 3, 5} {
                \pgfmathsetmacro{\idxA}{int(mod(\i+0, \n))}
                \pgfmathsetmacro{\idxB}{int(mod(\i+2, \n))}
                \pgfmathsetmacro{\idxC}{int(mod(\i+4, \n))}
                
                \draw[red, thick] (P\i) -- (L\idxA);
                \draw[blue, thick] (P\i) -- (L\idxB);
                \draw[mygold, thick] (P\i) -- (L\idxC);
            }
        \end{scope}

    \end{tikzpicture}
    \caption{The decomposition of $L_6(0,2,4)$ into two disjoint subgraphs. The left graph contains only even-indexed nodes, and the right graph contains only odd-indexed nodes. Note that the vertex positions preserve the orientation from the original circular layout.}
    \label{fig:L6_split}
\end{figure}
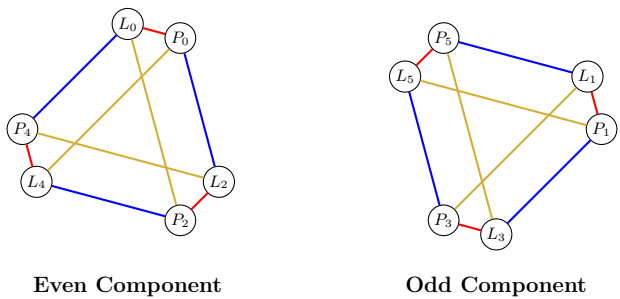

We also provide another—more symmetric and standard—visualization of these two graphs as Levi $L_3$ graphs. In this representation, the 6 nodes of each component ($3$ points, $3$ lines) are spaced equally on the perimeter. This reveals that each component forms a complete bipartite graph $L_3(0,1,2)$, where the connections corresponding to offsets $0$ and $2$ form the perimeter (hexagon), while the connections for offset $4$ form the long diagonals.

\begin{figure}[H]
    \centering
    \begin{tikzpicture}[scale=0.7, transform shape]
        \definecolor{mygold}{RGB}{212, 175, 55}
        \tikzset{lnode/.style={circle, draw=black, fill=white, inner sep=1.5pt, font=\scriptsize}}
        \tikzset{pnode/.style={circle, draw=black, fill=white, inner sep=1.5pt, text=black, font=\scriptsize}}
        \def\R{2.0}

        \begin{scope}[xshift=-3.0cm]
            \node[font=\bfseries] at (0, -2.8) {Even Component};
            
            \foreach \k in {0, 1, 2} {
                \pgfmathsetmacro{\idx}{int(2*\k)}
                
                \pgfmathsetmacro{\angL}{90 - \k * 120}
                \pgfmathsetmacro{\angP}{90 - \k * 120 - 60}
                
                \node[lnode] (L\idx) at (\angL:\R) {$L_{\idx}$};
                \node[pnode] (P\idx) at (\angP:\R) {$P_{\idx}$};
            }

            \foreach \k in {0, 1, 2} {
                \pgfmathsetmacro{\idx}{int(2*\k)}
                
                \pgfmathsetmacro{\nextIdx}{int(mod(\idx+2, 6))}
                \pgfmathsetmacro{\prevIdx}{int(mod(\idx+4, 6))} 
                
                \draw[red, thick] (P\idx) -- (L\idx);
                
                \draw[blue, thick] (P\idx) -- (L\nextIdx);
                
                \draw[mygold, thick] (P\idx) -- (L\prevIdx);
            }
        \end{scope}

        \begin{scope}[xshift=3.0cm]
            \node[font=\bfseries] at (0, -2.8) {Odd Component};
            
            \foreach \k in {0, 1, 2} {
                \pgfmathsetmacro{\idx}{int(2*\k + 1)}
                
                \pgfmathsetmacro{\angL}{90 - \k * 120}
                \pgfmathsetmacro{\angP}{90 - \k * 120 - 60}
                
                \node[lnode] (L\idx) at (\angL:\R) {$L_{\idx}$};
                \node[pnode] (P\idx) at (\angP:\R) {$P_{\idx}$};
            }

            \foreach \k in {0, 1, 2} {
                \pgfmathsetmacro{\idx}{int(2*\k + 1)}
                
                \pgfmathsetmacro{\nextIdx}{int(mod(\idx+2, 6))}
                \pgfmathsetmacro{\prevIdx}{int(mod(\idx+4, 6))}
                
                \draw[red, thick] (P\idx) -- (L\idx);
                
                \draw[blue, thick] (P\idx) -- (L\nextIdx);
                
                \draw[mygold, thick] (P\idx) -- (L\prevIdx);
            }
        \end{scope}

    \end{tikzpicture}
    \caption{Standard symmetric decomposition of $L_6(0,2,4)$. Each component is isomorphic to $L_3(0,1,2)$. The nodes are arranged equidistantly on the circle, revealing that the offsets $0$ (red) and $2$ (blue) form the hexagonal perimeter, while offset $4$ (gold) forms the diameters.}
    \label{fig:L6_split_symmetric}
\end{figure}
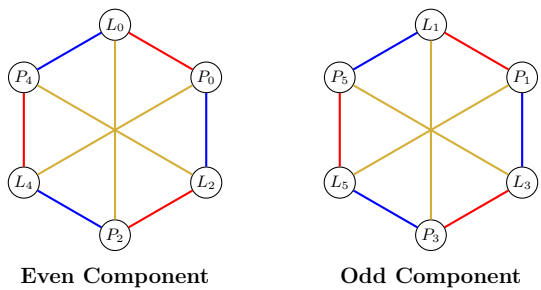
We conclude the $\mathbb{Z}_6$ section with a remark that:
\begin{enumerate}
\item Graph ${\bf L_6(0,1,2)}$ has symmetry group of {\bf order $24$}.
\item Graph ${\bf L_6(0,1,3)}$ has symmetry group of {\bf order $48$}.
 \item Graph ${\bf L_6(0,2,4)}$ has symmetry group of {\bf order $10368$}, (because, $10368=72^2\times 2$).
  \end{enumerate}

\subsubsection{The case $\mathbb{Z}_{12}$}

\paragraph{Affine equivalence of offsets}

For $n=12$, we have $\mathbb{Z}_{12}^* = \{1, 5, 7, 11\}$. Applying the same orbit enumeration algorithm, we find that the sets merge into exactly \textbf{nine} affine equivalence classes.
\begin{center}
\begin{tabular}{|c|m{10.2cm}|m{2.0cm}|c|}
\hline
Class & Offsets in the class (Normalized)&\centering Levi G& Order of $\mathrm{Aut}(G)$ \\ \hline
1 &${\bf \{0,1,2\}}, \{0,1,11\}, \{0,2,7\}, \{0,5,7\}, \{0,5,10\}, \{0,10,11\}$ &\centering ${\bf L_{12}(0,1,2)}$& {\bf 48}\\ \hline
2 &  ${\bf\{0,1,3\}}, \{0,1,10\}, \{0,2,3\}, \{0,2,5\}, \{0,2,9\}, \{0,2,11\},$ \newline $ \{0,3,5\}$, $\{0,3,10\}, \{0,7,9\}, \{0,7,10\}, \{0,9,10\}, \{0,9,11\}$ &\centering ${\bf L_{12}(0,1,3)}$ & {\bf 24}\\ \hline
3 &${\bf\{0,1,4\}}, \{0,1,9\}, \{0,3,4\}, \{0,3,7\}, \{0,3,8\}, \{0,3,11\}$ \newline $\{0,4,7\}, \{0,4,9\}, \{0,5,8\}, \{0,5,9\}, \{0,8,9\}, \{0,8,11\}$ &\centering ${\bf L_{12}(0,1,4)}$ & {\bf 24}\\ \hline
4 & ${\bf \{0,1,5\}}, \{0,1,8\}, \{0,4,5\}, \{0,4,11\}, \{0,7,8\}, \{0,7,11\}$ &\centering ${\bf L_{12}(0,1,5)}$ & {\bf 48}\\ \hline
5 & ${\bf \{0,1,6\}}, \{0,1,7\}, \{0,5,6\}, \{0,5,11\}, \{0,6,7\}, \{0,6,11\}$ & \centering ${\bf L_{12}(0,1,6)}$& {\bf 768}\\ \hline
6 & ${\bf \{0,2,4\}}, \{0,2,10\}, \{0,8,10\}$ & \centering ${\bf L_{12}(0,2,4)}$& {\bf 1152}\\ \hline
7 & ${\bf \{0,2,6\}}, \{0,2,8\}, \{0,4,6\}, \{0,4,10\}, \{0,6,8\}, \{0,6,10\}$ &\centering ${\bf L_{12}(0,2,6)}$ & {\bf 4608}\\ \hline
8 & ${\bf \{0,3,6\}}, \{0,3,9\}, \{0,6,9\}$ & \centering ${\bf L_{12}(0,3,6)}$&{\bf 663552} \\ \hline
9 & ${\bf \{0,4,8\}}$ &\centering ${\bf L_{12}(0,4,8)}$ & {\bf 644972544}\\ \hline
\end{tabular}
\end{center}

\begin{figure}[H]
    \centering
    \begin{tikzpicture}[scale=0.95, transform shape]
        \definecolor{mygold}{RGB}{212, 175, 55}
        \tikzset{lnode/.style={circle, draw=black, fill=white, inner sep=0.5pt, font=\fontsize{3}{4}\selectfont}}
        \tikzset{pnode/.style={circle, draw=black, fill=white, inner sep=0.5pt, text=black, font=\fontsize{3}{4}\selectfont}}

        \newcommand{\drawLeviTwelve}[6]{
            \begin{scope}[xshift=#1cm, yshift=#2cm]
                \def\n{12}
                \def\R{1.8}
                \foreach \i in {0,...,11} {
                    \pgfmathsetmacro{\angL}{90 - \i * (360/\n)}
                    \pgfmathsetmacro{\angP}{\angL - (180/\n)}
                    \node[lnode] (L\i) at (\angL:\R) {$L_{\i}$};
                    \node[pnode] (P\i) at (\angP:\R) {$P_{\i}$};
                }
                \foreach \i in {0,...,11} {
                    \pgfmathsetmacro{\idxA}{int(mod(\i+#4, \n))}
                    \pgfmathsetmacro{\idxB}{int(mod(\i+#5, \n))}
                    \pgfmathsetmacro{\idxC}{int(mod(\i+#6, \n))}
                    \draw[red, thin] (P\i) -- (L\idxA);
                    \draw[blue, thin] (P\i) -- (L\idxB);
                    \draw[mygold, thin] (P\i) -- (L\idxC);
                }
                \node[yshift=-2.2cm, font=\bfseries\scriptsize] at (0,0) {#3};
            \end{scope}
        }
        
        \def\dx{5.0}
        \def\dy{-5.0}

        \drawLeviTwelve{0}{0}{$L_{12}(0,1,2)$}{0}{1}{2}
        \drawLeviTwelve{\dx}{0}{$L_{12}(0,1,3)$}{0}{1}{3}
        \drawLeviTwelve{2*\dx}{0}{$L_{12}(0,1,4)$}{0}{1}{4}

        \drawLeviTwelve{0}{\dy}{$L_{12}(0,1,5)$}{0}{1}{5}
        \drawLeviTwelve{\dx}{\dy}{$L_{12}(0,1,6)$}{0}{1}{6}
        \drawLeviTwelve{2*\dx}{\dy}{$L_{12}(0,2,4)$}{0}{2}{4}

        \drawLeviTwelve{0}{2*\dy}{$L_{12}(0,2,6)$}{0}{2}{6}
        \drawLeviTwelve{\dx}{2*\dy}{$L_{12}(0,3,6)$}{0}{3}{6}
        \drawLeviTwelve{2*\dx}{2*\dy}{$L_{12}(0,4,8)$}{0}{4}{8}

    \end{tikzpicture}
    \caption{The nine canonical affine classes of cyclic cubic bipartite Levi graphs for $\mathbb{Z}_{12}$. Nodes $L_i$ and $P_i$ alternate on the circle. The three offsets correspond to Red (1st), Blue (2nd), and Gold (3rd) edges.}
    \label{fig:L12_classification}
\end{figure}
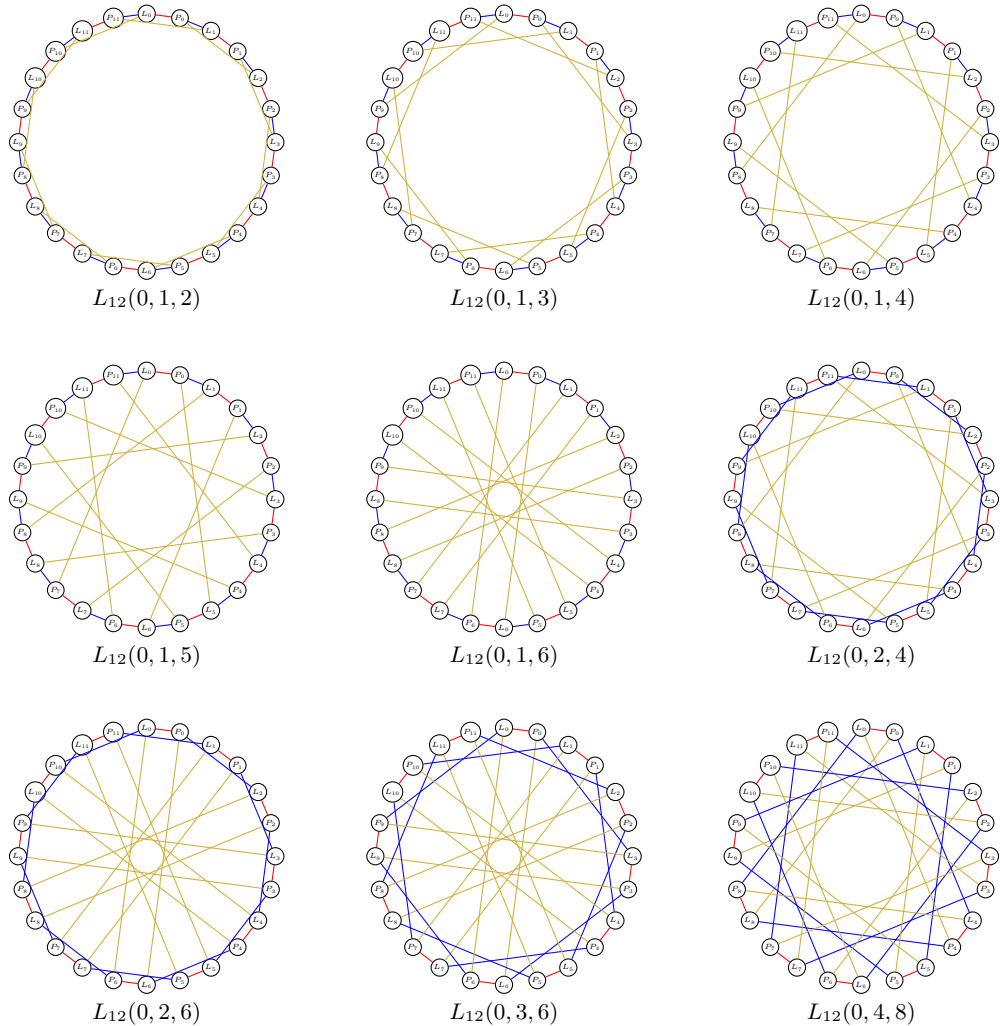

The phenomenon of graph disconnection observed in $L_6(0,2,4)$ reappears in the $\mathbb{Z}_{12}$ classification.
Among the nine classes of $L_{12}$, four exhibit this separation. We analyze these decompositions below.

\subsubsection*{Decomposition of $L_{12}(0,2,4)$}
The graph $L_{12}(0,2,4)$ decomposes into two identical components. 

\begin{itemize}
    \item **Even Component:** Contains vertices with indices $\{0, 2, 4, 6, 8, 10\}$.
    \item **Odd Component:** Contains vertices with indices $\{1, 3, 5, 7, 9, 11\}$.
\end{itemize}

By dividing the offsets by the GCD ($0/2, 2/2, 4/2$), we find that the internal structure of each component is described by the local offsets $\{0, 1, 2\}$ acting on 6 nodes. Thus, both components are isomorphic to $L_6(0,1,2)$.

\begin{figure}[H]
    \centering
    \begin{tikzpicture}[scale=0.9, transform shape]
        \definecolor{mygold}{RGB}{212, 175, 55}
        \tikzset{lnode/.style={circle, draw=black, fill=white, inner sep=1.5pt, font=\scriptsize}}
        \tikzset{pnode/.style={circle, draw=black, fill=white, inner sep=1.5pt, text=black, font=\scriptsize}}
        \def\R{2.2}

        \begin{scope}[xshift=-4cm]
            \node[font=\bfseries] at (0, -2.8) {Even Component};
            
            \foreach \j in {0,...,5} {
                \pgfmathsetmacro{\labelIdx}{int(2*\j)}
                
                \pgfmathsetmacro{\angL}{90 - (2*\j) * 30}
                \node[lnode] (L\j) at (\angL:\R) {$L_{\labelIdx}$};
                
                \pgfmathsetmacro{\angP}{90 - (2*\j + 1) * 30}
                \node[pnode] (P\j) at (\angP:\R) {$P_{\labelIdx}$};
            }

            \foreach \j in {0,...,5} {
                \pgfmathsetmacro{\targetA}{int(mod(\j+0, 6))}
                \pgfmathsetmacro{\targetB}{int(mod(\j+1, 6))}
                \pgfmathsetmacro{\targetC}{int(mod(\j+2, 6))}
                
                \draw[red, thick] (P\j) -- (L\targetA);    
                \draw[blue, thick] (P\j) -- (L\targetB);   
                \draw[mygold, thick] (P\j) -- (L\targetC); 
            }
        \end{scope}

        \begin{scope}[xshift=4cm]
            \node[font=\bfseries] at (0, -2.8) {Odd Component};
            
            \foreach \j in {0,...,5} {
                \pgfmathsetmacro{\labelIdx}{int(2*\j + 1)}
                
                \pgfmathsetmacro{\angL}{90 - (2*\j) * 30}
                \node[lnode] (L\j) at (\angL:\R) {$L_{\labelIdx}$};
                
                \pgfmathsetmacro{\angP}{90 - (2*\j + 1) * 30}
                \node[pnode] (P\j) at (\angP:\R) {$P_{\labelIdx}$};
            }

            \foreach \j in {0,...,5} {
                \pgfmathsetmacro{\targetA}{int(mod(\j+0, 6))}
                \pgfmathsetmacro{\targetB}{int(mod(\j+1, 6))}
                \pgfmathsetmacro{\targetC}{int(mod(\j+2, 6))}
                
                \draw[red, thick] (P\j) -- (L\targetA);
                \draw[blue, thick] (P\j) -- (L\targetB);
                \draw[mygold, thick] (P\j) -- (L\targetC);
            }
        \end{scope}

    \end{tikzpicture}
    \caption{Decomposition of $L_{12}(0,2,4)$. The graph separates into two disjoint components (Even and Odd indices), each forming a ring graph isomorphic to $L_6(0,1,2)$. All vertices lie on a single circle per component.}
    \label{fig:L12_024_split}
\end{figure}
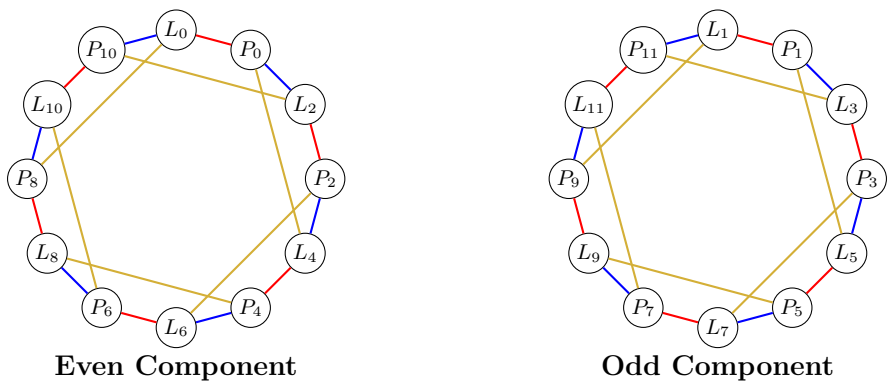

\subsubsection*{Decomposition of $L_{12}(0,2,6)$} Similar to the previous case, the graph splits into even and odd components. However, the reduced offsets are $0/2=0, 2/2=1, 6/2=3$. Thus, each component is isomorphic to $L_6(0,1,3)$, which has a different internal structure (specifically, it contains hexagram-like crossings).

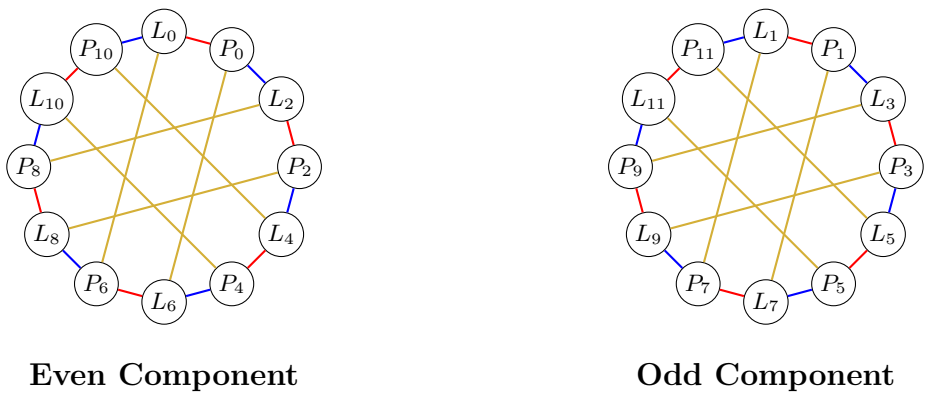
\begin{figure}[H]
    \centering
    \begin{tikzpicture}[scale=1.0, transform shape]
        \definecolor{mygold}{RGB}{212, 175, 55}
        \tikzset{lnode/.style={circle, draw=black, fill=white, inner sep=1.5pt, font=\scriptsize}}
        \tikzset{pnode/.style={circle, draw=black, fill=white, inner sep=1.5pt, text=black, font=\scriptsize}}
        \def\R{1.8}

        \begin{scope}[xshift=-4cm]
            \node[font=\bfseries] at (0, -2.8) {Even Component};
            \foreach \j in {0,...,5} {
                \pgfmathsetmacro{\origI}{int(2*\j)}
                \pgfmathsetmacro{\angL}{90 - \j * 60}
                \pgfmathsetmacro{\angP}{\angL - 30}
                \node[lnode] (L\j) at (\angL:\R) {$L_{\origI}$};
                \node[pnode] (P\j) at (\angP:\R) {$P_{\origI}$};
            }
            \foreach \j in {0,...,5} {
                \pgfmathsetmacro{\idxA}{int(mod(\j+0, 6))}
                \pgfmathsetmacro{\idxB}{int(mod(\j+1, 6))}
                \pgfmathsetmacro{\idxC}{int(mod(\j+3, 6))}
                \draw[red, thick] (P\j) -- (L\idxA);
                \draw[blue, thick] (P\j) -- (L\idxB);
                \draw[mygold, thick] (P\j) -- (L\idxC);
            }
        \end{scope}

        \begin{scope}[xshift=4cm]
            \node[font=\bfseries] at (0, -2.8) {Odd Component};
            \foreach \j in {0,...,5} {
                \pgfmathsetmacro{\origI}{int(2*\j + 1)}
                \pgfmathsetmacro{\angL}{90 - \j * 60}
                \pgfmathsetmacro{\angP}{\angL - 30}
                \node[lnode] (L\j) at (\angL:\R) {$L_{\origI}$};
                \node[pnode] (P\j) at (\angP:\R) {$P_{\origI}$};
            }
            \foreach \j in {0,...,5} {
                \pgfmathsetmacro{\idxA}{int(mod(\j+0, 6))}
                \pgfmathsetmacro{\idxB}{int(mod(\j+1, 6))}
                \pgfmathsetmacro{\idxC}{int(mod(\j+3, 6))}
                \draw[red, thick] (P\j) -- (L\idxA);
                \draw[blue, thick] (P\j) -- (L\idxB);
                \draw[mygold, thick] (P\j) -- (L\idxC);
            }
        \end{scope}
    \end{tikzpicture}
    \caption{Decomposition of $L_{12}(0,2,6)$ into two disjoint copies of $L_6(0,1,3)$.}
    \label{fig:L12_026_split}
\end{figure}

\subsubsection*{Decomposition of $L_{12}(0,3,6)$}
The graph decomposes into three disjoint components. Dividing the offsets by 3 gives $\{0, 1, 2\}$. Thus, each component is isomorphic to $L_4(0,1,2)$, which appears as a square (or cube-like) graph.

\begin{figure}[H]
    \centering
    \begin{tikzpicture}[scale=1.0, transform shape]
        \definecolor{mygold}{RGB}{212, 175, 55}
        \tikzset{lnode/.style={circle, draw=black, fill=white, inner sep=1.5pt, font=\scriptsize}}
        \tikzset{pnode/.style={circle, draw=black, fill=white, inner sep=1.5pt, text=black, font=\scriptsize}}
        \def\R{1.8}

        \newcommand{\drawCompThree}[2]{
            \begin{scope}[xshift=#1cm]
                \node[font=\bfseries] at (0, -2.5) {Residue #2};
                \foreach \j in {0,...,3} {
                    \pgfmathsetmacro{\origI}{int(3*\j + #2)}
                    \pgfmathsetmacro{\angL}{90 - \j * 90}
                    \pgfmathsetmacro{\angP}{\angL - 45}
                    \node[lnode] (L\j) at (\angL:\R) {$L_{\origI}$};
                    \node[pnode] (P\j) at (\angP:\R) {$P_{\origI}$};
                }
                \foreach \j in {0,...,3} {
                    \pgfmathsetmacro{\idxA}{int(mod(\j+0, 4))}
                    \pgfmathsetmacro{\idxB}{int(mod(\j+1, 4))}
                    \pgfmathsetmacro{\idxC}{int(mod(\j+2, 4))}
                    \draw[red, thick] (P\j) -- (L\idxA);
                    \draw[blue, thick] (P\j) -- (L\idxB);
                    \draw[mygold, thick] (P\j) -- (L\idxC);
                }
            \end{scope}
        }

        \drawCompThree{-5}{0}
        \drawCompThree{0}{1}
        \drawCompThree{5}{2}

    \end{tikzpicture}
    \caption{Decomposition of $L_{12}(0,3,6)$ into three disjoint copies of $L_4(0,1,2)$.}
    \label{fig:L12_036_split}
\end{figure}
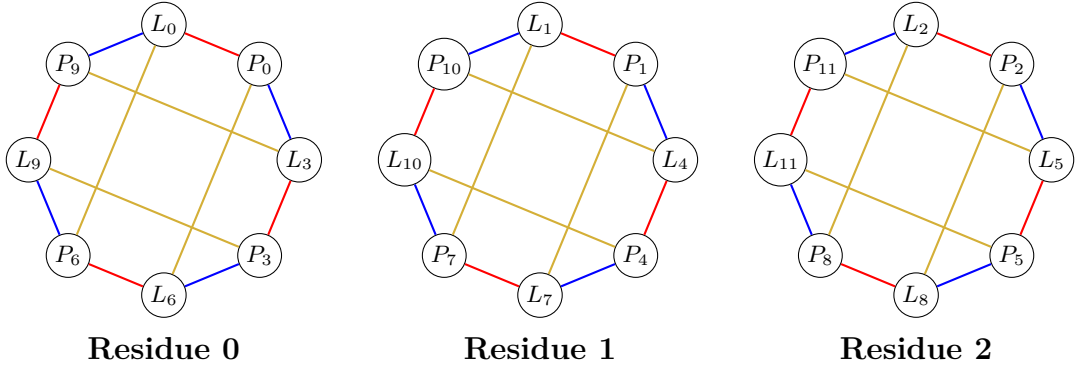

\subsubsection*{Decomposition of $L_{12}(0,4,8)$}
Finally, the $L_{12}(0,4,8)$ graph fractures into four independent components. Dividing offsets by 4 gives $\{0, 1, 2\}$ in $\mathbb{Z}_3$. Each component is an instance of $L_3(0,1,2)$. Since $L_3$ contains 3 points and 3 lines with every point connected to every line, each component forms a complete bipartite graph.

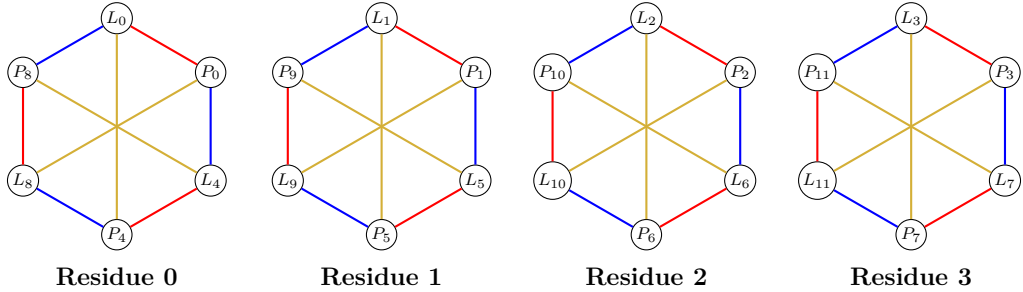
\begin{figure}[H]
    \centering
    \begin{tikzpicture}[scale=0.8, transform shape]
        \definecolor{mygold}{RGB}{212, 175, 55}
        \tikzset{lnode/.style={circle, draw=black, fill=white, inner sep=1.0pt, font=\tiny}}
        \tikzset{pnode/.style={circle, draw=black, fill=white, inner sep=1.0pt, text=black, font=\tiny}}
        \def\R{1.8}
        
        \newcommand{\drawCompFour}[2]{
            \begin{scope}[xshift=#1cm]
                \node[font=\bfseries] at (0, -2.5) {Residue #2};
                \foreach \j in {0,...,2} {
                    \pgfmathsetmacro{\origI}{int(4*\j + #2)}
                    \pgfmathsetmacro{\angL}{90 - \j * 120}
                    \pgfmathsetmacro{\angP}{\angL - 60}
                    \node[lnode] (L\j) at (\angL:\R) {$L_{\origI}$};
                    \node[pnode] (P\j) at (\angP:\R) {$P_{\origI}$};
                }
                \foreach \j in {0,...,2} {
                    \pgfmathsetmacro{\idxA}{int(mod(\j+0, 3))}
                    \pgfmathsetmacro{\idxB}{int(mod(\j+1, 3))}
                    \pgfmathsetmacro{\idxC}{int(mod(\j+2, 3))}
                    \draw[red, thick] (P\j) -- (L\idxA);
                    \draw[blue, thick] (P\j) -- (L\idxB);
                    \draw[mygold, thick] (P\j) -- (L\idxC);
                }
            \end{scope}
        }

        \drawCompFour{-6.6}{0}
        \drawCompFour{-2.2}{1}
        \drawCompFour{2.2}{2}
        \drawCompFour{6.6}{3}

    \end{tikzpicture}
    \caption{Decomposition of $L_{12}(0,4,8)$ into four disjoint copies of $K_{3,3}$ (isomorphic to $L_3(0,1,2)$).}
    \label{fig:L12_048_split}
\end{figure}

\subsection{Symmetry Groups of $\mathbb{Z}_{12}$ Levi Graphs}

The Levi graph of configuration $\mathbb{Z}_{12}$ with offsets $S$ is a bipartite graph with 24 vertices (12 Points $P_i$, 12 Lines $L_j$) and edges defined by $L_{i+s} \sim P_i$ for each $s \in S$. The structure of the automorphism group is determined by the divisibility and arithmetic properties of the offset set $S$.

For $n=12$, there are 9 distinct isomorphism classes. We categorize them below by their connectivity and symmetry properties.

\subsubsection{Connected Graphs}

This is about graphs from Classes 1 to 5 in our Table of the $\mathbb{Z}_{12}$ normalized offsets. The symmetry order depends on whether the offset structure allows for reflection or specific modular multipliers.

\paragraph{\bf Class 1: Offset $S = \{0, 1, 2\}$}

\textbf{Group Order:} $48$.

This graph corresponds to the edges of a hexagonal prism.
\begin{enumerate}
    \item \textbf{Rotational Symmetry ($\times 12$):} The cyclic group $\mathbb{Z}_{12}$ acts freely on the indices.
    \item \textbf{Reflectional Symmetry ($\times 2$):} The interval sequence $(1, 1)$ is a palindrome, allowing for mirror reflection ($x \to -x$).
    \item \textbf{Duality ($\times 2$):} The graph is self-dual ($P \leftrightarrow L$).
\end{enumerate}

\textbf{Total Order:} $12 \times 2 \times 2 = \mathbf{48}$.

\paragraph{\bf Classes 2 \& 3: Offsets $\{0, 1, 3\}$ and $\{0, 1, 4\}$}

\textbf{Group Orders:} $24$.

These graphs possess chirality. The interval sequences ($1,2$ and $1,3$ respectively) are not palindromes, eliminating reflectional symmetry.

\textbf{Total Order:} $12 \text{ (Rot)} \times 1 \text{ (No Refl)} \times 2 \text{ (Dual)} = \mathbf{24}$.

\paragraph{\bf Class 4: Offset  $S = \{0, 1, 5\}$}

\textbf{Group Order:} $48$.

Unlike the chiral classes, this graph regains high symmetry through number-theoretic properties specific to the multiplier $5$ in $\mathbb{Z}_{12}$.
\begin{enumerate}
    \item \textbf{Rotational Symmetry ($\times 12$):} Standard cyclic action.
    \item \textbf{Multiplier Symmetry ($\times 2$):} The map $\sigma(x) = 5x \pmod{12}$ preserves the set $S$. Note that $5 \times \{0, 1, 5\} = \{0, 5, 25\} \equiv \{0, 5, 1\}$, which is the original set $S$. This acts as a permutation of the connections.
    \item \textbf{Duality ($\times 2$):} The map $\tau(x) = 7x \pmod{12}$ swaps Points and Lines. Note that $7 \times S = \{0, 7, 35\} \equiv \{0, 7, 11\}$. This corresponds to the set $-S = \{0, -1, -5\}$ (modulo 12), which is the condition for self-duality.
\end{enumerate}

\textbf{Total Order:} $12 \times 2 \times 2 = \mathbf{48}$.

\paragraph{\bf Class 5: Offset $S = \{0, 1, 6\}$}

\textbf{Group Order:} 768.

The graph consists of 24 vertices. The specific connectivity $S=\{0,1,6\}$ partitions these vertices into 6 distinct rectangles (4-cycles) arranged in a hexagonal ring; see the central graph at Figure \ref{fig:L12_classification}. The automorphism group order arises from the independence of local operations within these rectangles relative to the global ring structure.

The total order is calculated as the product of the macroscopic ring symmetry and the microscopic local freedoms:
$$ |\mathrm{Aut}(G)| = 12 \times 64 = 768 $$

\begin{enumerate}
    \item \textbf{Macroscopic Symmetry (The Ring):} \\
    The 6 rectangles are arranged in a cycle $Q_0 - Q_1 - Q_2 - Q_3 - Q_4 - Q_5 - Q_0$. Treating each rectangle as a rigid unit, the structure forms a hexagon. The symmetry group of a hexagon is the Dihedral group $D_6$, which contributes:
    \begin{itemize}
        \item \textbf{6 Rotations:} Shifting the ring by $k$ steps.
        \item \textbf{6 Reflections:} Flipping the ring across its axes.
        \item \textbf{Total:} 12 symmetries.
    \end{itemize}

    \item \textbf{Microscopic Symmetry (The Independent Flips):} \\
    The connections between adjacent rectangles (offsets 1 and -1) link the "top" vertices to "top" vertices and "bottom" to "bottom" (parallel rails). This decoupling allows each rectangle $Q_i$ to undergo an internal vertical flip (swapping its top and bottom vertices) independently of its neighbors.
    
    Since there are 6 rectangles and each has 2 states (flipped or unflipped), the number of local operations is:
    $$ \underbrace{2}_{Q_0} \times \underbrace{2}_{Q_1} \times \underbrace{2}_{Q_2} \times \underbrace{2}_{Q_3} \times \underbrace{2}_{Q_4} \times \underbrace{2}_{Q_5} = 2^6 = 64 $$
\end{enumerate}

\vspace{1cm}

\subsection*{Visualization of the $12 \times 2^6$ Structure}
The diagram below illustrates the 6 rectangles forming the ring. The blue arrows indicate the 12 global symmetries ($D_6$), while the red arrows indicate the 6 independent local flips ($Z_2$).

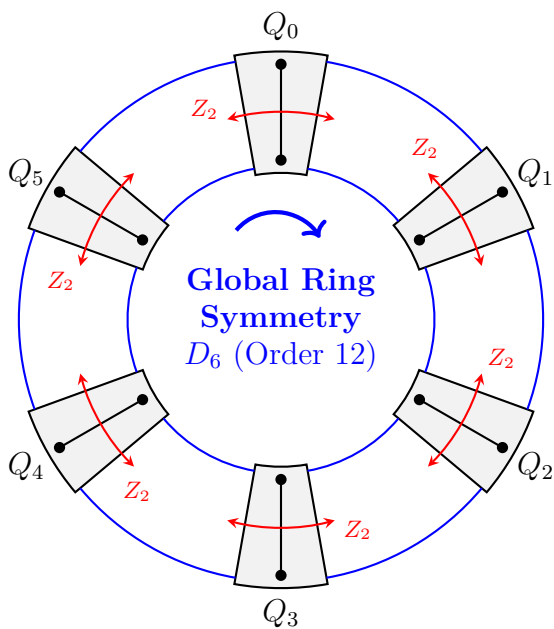
\begin{center}
\begin{tikzpicture}[scale=0.85]
    \foreach \i in {0,...,5} {
        \pgfmathsetmacro{\ang}{90 - \i*60}
        \pgfmathsetmacro{\nextang}{90 - mod(\i+1,6)*60}
        
        \coordinate (T\i) at (\ang:4cm);
        \coordinate (B\i) at (\ang:2.5cm);
        
        \draw[thick, fill=gray!10] (\ang-10:4.2cm) arc (\ang-10:\ang+10:4.2cm) -- 
                                  (\ang+10:2.3cm) arc (\ang+10:\ang-10:2.3cm) -- cycle;
        
        \node[circle, fill=black, inner sep=1.5pt] at (T\i) {};
        \node[circle, fill=black, inner sep=1.5pt] at (B\i) {};
        
        \draw[thick] (T\i) -- (B\i);
        
        \node at (\ang:4.6cm) {\textbf{$Q_\i$}};
        
        \draw[<->, red, thick, >=stealth] (\ang+15:3.25cm) arc (\ang+15:\ang-15:3.25cm);
        \node[red, font=\scriptsize] at (\ang+20:3.5cm) {$Z_2$};
    }
    
    \foreach \i in {0,...,5} {
        \pgfmathsetmacro{\ang}{90 - \i*60}
        \pgfmathtruncatemacro{\nxt}{mod(\i+1,6)}
        
        \draw[thick, blue] (\ang-10:4.1cm) arc (\ang-10:\ang-50:4.1cm);
        \draw[thick, blue] (\ang-10:2.4cm) arc (\ang-10:\ang-50:2.4cm);
    }

    \draw[->, blue, ultra thick] (-0.7,1.4) arc (140:30:0.8cm);
    \node[blue, align=center] at (0,0) {\textbf{Global Ring}\\ \textbf{Symmetry}\\ $D_6$ (Order 12)};

\end{tikzpicture}
\end{center}

\subsubsection{Disconnected Graphs (Parallel Universes)}

When $d = \gcd(S \cup \{12\}) > 1$, the graph splits into $d$ disjoint components. The total group order is the wreath product of the component symmetry group with the permutation group $S_d$.

\paragraph{\bf Class 6: Offset $S = \{0, 2, 4\}$}

\textbf{Group Order:} $1152$.

\begin{enumerate}
    \item \textbf{Decomposition:} The graph splits into $d=2$ components, each isomorphic to $L_6(0, 1, 2)$.
    \item \textbf{Component Symmetry ($\times 24$):} Each component is a connected, palindromic Levi graph of order 6. Its symmetry is $6 \times 2 \times 2 = 24$.
    \item \textbf{Permutation ($\times 2$):} The two components are identical and can be swapped.
\end{enumerate}

\textbf{Total Order:} $(24)^2 \times 2! = 576 \times 2 = \mathbf{1,152}$.

\paragraph{\bf Class 7: Offset $S = \{0, 2, 6\}$}

\textbf{Group Order:} $4608 = 2 \times 48^2$.

The graph is disconnected. Since the step size for indices is always even ($0, 2, 6$), the set of vertices partitions perfectly into two disjoint sets: those with even indices and those with odd indices.

\begin{enumerate}
    \item \textbf{Decomposition:}
    The graph decomposes into two isomorphic components:
    $$ L_{12}(0,2,6) \cong L_{even} \cup L_{odd} $$
    Each component is isomorphic to the Franklin Graph ($L_6(0,1,3)$).

    \item \textbf{Symmetry Calculation:}
    The total group order is the product of the symmetries of the two components and the swap operation between them:
    $$ |\mathrm{Aut}(G)| = |Aut(L_{even})| \times |Aut(L_{odd})| \times 2 = 48 \times 48 \times 2 = 4608 $$
\end{enumerate}

\paragraph{\bf Class 8: Offset $S = \{0, 3, 6\}$}

\textbf{Group Order:} $663552$.

\begin{enumerate}
    \item \textbf{Decomposition:} Splits into 3 components, each isomorphic to $L_4(0, 1, 2)$.
    \item \textbf{Component Identity:} $L_4(0, 1, 2)$ is the 3-Cube graph ($Q_3$).
    \item \textbf{Component Symmetry ($\times 48$):} The hypercubic group on 8 vertices has order 48.
\end{enumerate}

\textbf{Total Order:} $(48)^3 \times 3! = 110,592 \times 6 = \mathbf{663,552}$.

\paragraph{\bf Class 9: Offset $S = \{0, 4, 8\}$}

\textbf{Group Order:} $644972544$.
\begin{enumerate}
    \item \textbf{Decomposition:} Splits into 4 components, each isomorphic to $L_3(0, 1, 2)$.
    \item \textbf{Component Identity:} $L_3(0, 1, 2)$ is the complete bipartite graph $K_{3,3}$ (the Utility Graph).
    \item \textbf{Component Symmetry ($\times 72$):} The automorphism group of $K_{3,3}$ has order 72.
\end{enumerate}

\textbf{Total Order:} $(72)^4 \times 4! = \mathbf{644,972,544}$.

\section{Characterizations of the 12-TET Musical Systems}\label{sec9}
Table \ref{tab:12tet_systems_classification} summarizes the eight systems, their defining intervals, and their corresponding offset sets $S$. The entries in the last column of this table can be understood by examining the Levi graphs of the \textbf{eight} musical systems discussed in Sections 4--9. Alternatively, they are explained in terms of the explicit isomorphisms between the relevant graphs in the characterization theorems below the table.

\begin{table}[H]
\centering
\begin{tabular}{|l|c|c|c|l|}
\hline
\textbf{System Name} & \textbf{Thirds} $(t, s)$ & \textbf{Fifth} $q$ & \textbf{Offset Set} $S=\{0, t, n-s\}$ & \textbf{Canonical Class} \\
\hline
Wide Thirds & $(5, 2)$ & $7$ & $\{0, 5, 10\}$ & $L_{12}(0,1,2)$ \\
Far Narrow Thirds & $(10, 9)$ & $7$ & $\{0, 10, 3\}$ & $L_{12}(0,1,3)$ \\
Standard Thirds & $(4, 3)$ & $7$ & $\{0, 4, 9\}$ & $L_{12}(0,1,4)$ \\
Far Wide Thirds & $(11, 8)$ & $7$ & $\{0, 11, 4\}$ & $L_{12}(0,1,5)$ \\
Tritone System & $(6, 1)$ & $7$ & $\{0, 6, 11\}$ & $L_{12}(0,1,6)$ \\
Tritone-Wide System & $(8, 2)$ & $10$ & $\{0, 8, 10\}$ & $L_{12}(0,2,4)$ \\
Tritonian Fifth System & $(4, 2)$ & $6$ & $\{0, 4, 10\}$ & $L_{12}(0,2,6)$ \\
Equally Spaced System & $(6, 3)$ & $9$ & $\{0, 6, 9\}$ & $L_{12}(0,3,6)$ \\
\hline
\end{tabular}
\caption{Classification of 12-TET Harmonic Systems by Levi Graph Offsets}
\label{tab:12tet_systems_classification}
\end{table}

\subsection{Characterization Theorems}

\subsubsection{Class 1: System $q=7$, $(t,s)=(5,2)$}

\begin{theorem}
The "Wide Thirds" system $(5,2)$ forms a unique equivalence class in 12-TET. It is isomorphic to the Levi graph belonging to the equivalence class $L_{12}(0,1,2)$.
\end{theorem}

\begin{proof}
The offset set is $S_{(5,2)} = \{0, 5, 10\}$. The map $\phi(x) = 5x \pmod{12}$ transforms this set into the canonical offset:
\[ \{0, 5, 10\} \xrightarrow{\times 5} \{0, 25, 50\} \equiv \{0, 1, 2\}. \]
\end{proof}

\subsubsection{The Connected Cubic Classes 2, 3, and 4: Systems (10,9), (4,3), and (11,8)}

These three systems represent the standard connected cubic graphs of girth 6, but they belong to distinct topological classes.

\begin{theorem}
The "Far Narrow Thirds" $(10,9)$, "Standard Thirds" $(4,3)$, and "Far Wide Thirds" $(11,8)$ systems, all with $q=7$, are topologically distinct. They correspond to the Levi graph classes $L_{12}(0,1,3)$, $L_{12}(0,1,4)$, and $L_{12}(0,1,5)$ respectively.
\end{theorem}

\begin{proof}
We construct the affine maps to their canonical forms:
\begin{enumerate}
    \item \textbf{(10,9) to Class 3:} The offset set is $S_{(10,9)} = \{0, 10, 3\}$. The map $\phi(x) = 11x+1$ transforms:
    \[ \{0, 10, 3\} \xrightarrow{\times 11} \{0, 110, 33\} \equiv \{0, 2, 11\}\xrightarrow{+ 1} \{1,3,0\}= \{0,1,3\} . \]
    \item \textbf{(4,3) to Class 4:} The offset set is $S_{(4,3)} = \{0, 4, 9\}$. The map $\phi(x) = 5x+4$ transforms:
    \[ \{0, 4, 9\} \xrightarrow{\times 5} \{0, 20, 45\} \equiv \{0, 8, 9\} \xrightarrow{+4} \{4, 0, 1\} = \{0, 1, 4\}. \]
    \item \textbf{(11,8) to Class 5:} The offset set is $S_{(11,8)} = \{0, 11, 4\}$. The map $\phi(x) = 5x+5$ transforms:
    \[ \{0, 11, 4\} \xrightarrow{\times 5} \{0, 55, 20\} \equiv \{0, 7, 8\}\xrightarrow{+ 5}\{5,0,1\}=\{0,1,5\}. \]
\end{enumerate}
\end{proof}

\subsubsection{Class 5: System $q=7$, $(t,s)=(6,1)$}

\begin{theorem}
The "Tritone" system $(6,1)$ forms a unique equivalence class with symmetry group order 768. It is isomorphic to the Levi graph $L_{12}(0,1,6)$.
\end{theorem}

\begin{proof}
The offset set is $S_{(6,1)} = \{0, 6, 11\}$. The map $\phi(x) = 11x$ transforms:
\[ \{0, 6, 11\} \xrightarrow{\times 11} \{0, 66, 121\} \equiv \{0, 6, 1\} = \{0, 1, 6\}. \]
\end{proof}

\subsubsection{The Disconnected Classes 6, 7, and 8: Systems (8,2), (4,2), and (6,3)}

These three "exotic" systems are characterized by their disconnected topology, splitting into disjoint isomorphic components.

\begin{theorem}
The ``Tritone Wide'' System $q=10$, $(t,s)=(8,2)$, the ``Tritonian Fifth'' System $q=6$, $(t,s)=(4,2)$, and the ``Equally Spaced'' System $q=9$, $(t,s)=(6,3)$ correspond to the three topologically distinct classes of cyclic bipartite cubic Levi graphs in $\mathbb{Z}_{12}$. Their topology is the same as the topology of the respective Levi graph classes $L_{12}(0,2,4)$, $L_{12}(0,2,6)$, and $L_{12}(0,3,6)$.
\end{theorem}

\begin{proof}\phantom{}\\
\begin{enumerate}
    \item \textbf{(8,2) to Class 6:} The offset set is $S_{(8,2)} = \{0, 8, 10\}$. The map $\phi(x) = 5x$ transforms:
    \[ \{0, 8, 10\} \xrightarrow{\times 5} \{0, 40, 50\} \equiv \{0, 4, 2\} = \{0, 2, 4\}. \]
    \item \textbf{(4,2) to Class 7:} The offset set is $S_{(4,2)} = \{0, 4, 10\}$. The shift map $\phi(x) = x+2$ transforms:
    \[ \{0, 4, 10\} \xrightarrow{+2} \{2, 6, 0\} = \{0, 2, 6\}. \]
    \item \textbf{(6,3) to Class 8:} The offset set is $S_{(6,3)} = \{0, 6, 9\}$. The shift map $\phi(x) = x+6$ gives the required isomorphism:
    \[ \{0, 6, 9\} \xrightarrow{+6} \{6, 0, 3\} = \{0, 3, 6\}. \]
\end{enumerate}
\end{proof}

\subsection{Completeness of the Musical Classification}

The eight systems analyzed in Sections 4 through 9 were not chosen arbitrarily. They were selected to provide a specific "musical" representative for each admissible topological class of cyclic bipartite cubic Levi graphs in $\mathbb{Z}_{12}$.

As shown in the classification (Table \ref{tab:12tet_systems_classification}), our musical systems cover Classes 1 through 8. However, the mathematical classification of $L_{12}(S)$ graphs identifies a ninth class, represented by the canonical offset $\{0, 4, 8\}$. We now demonstrate why this class cannot support a harmonic system defined by our axioms, thereby confirming that our analysis of musical geometries in 12-TET is exhaustive.

\subsubsection{The Holographic Theorem for Musical Systems}

A fundamental property of our harmonic definition is that every system comes in conjugate pairs. If a system is defined by the partition $t+s=q$, the physical set of intervals $\{t, s\}$ is inherent to the system regardless of which interval is assigned to the "Major" or "Minor" role.

\begin{theorem}[The Holographic Principle]
Let a harmonic system in $n$-TET be defined by the generator $q$ and the partition $t+s=q$. The four offset sets associated with this system:
\[ S_1 = \{0, t, n-s\}, \quad S_2 = \{0, s, n-t\}, \]
\[ S_3 = \{0, t, q\}, \quad S_4 = \{0, s, q\} \]
are all affinely equivalent. Consequently, they all generate isomorphic Levi graphs.
\end{theorem}

\begin{proof}
Let $S_1 = \{0, t, n-s\}$.
\begin{enumerate}
    \item \textbf{Equivalence of $S_1$ and $S_2$:} The map $\phi(x) = -x$ (reflection) transforms $S_1$ to $S_2$:
      \[S_1= \{0, t, -s\} \xrightarrow{\times( -1)} \{0, -t, s\} \equiv \{0, s,n-t\} = S_2. \]
      Thus $S_1$ and $S_2$ are affinely equivalent.
    \item \textbf{Equivalence of $S_3$ and $S_2$, and $S_4$ and $S_1$:} The shift $\phi(x)=x-t$ affinely transforms $S_3$ to $S_2$:
\[S_3= \{0, t, q\}= \{0, t, t+s\} \xrightarrow{ -t} \{-t,0, s\} \equiv \{0, s, n-t\} = S_2. \]
Likewise, the shift $\phi(x)=x-s$ affinely transforms $S_4$ to $S_1$:
\[S_4= \{0, s, q\}= \{0, s, t+s\} \xrightarrow{ -s} \{-s,0, t\} \equiv \{0, t, n-s\} = S_1. \]
\end{enumerate}
Thus we have established affine equivalences $S_4\equiv S_1\equiv S_2\equiv S_3$, which completes the proof.
\end{proof}

\subsubsection{Exclusion of the Non-Musical Class 9}

A necessary condition for a Levi graph $L_{12}(S)$ to represent a valid musical system is that its equivalence class must contain at least two distinct normalized offset sets corresponding to the conjugate pair $\{0, t, n-s\}$ and $\{0, s, n-t\}$.

\begin{theorem}
The equivalence class represented by the canonical offset $S_{can} = \{0, 4, 8\}$ contains only a single normalized element. Therefore, it cannot represent a musical system generated by a partition $t+s=q$ with $t \neq s$.
\end{theorem}

\begin{proof}
The set $S = \{0, 4, 8\}$ in $\mathbb{Z}_{12}$ consists of the elements of the subgroup generated by 4.
The normalized elements of an equivalence class are found by applying all affine transformations $\phi(x) = ax+b$ where $a \in \mathbb{Z}_{12}^*$ and $b \in \mathbb{Z}_{12}$, then sorting and shifting to start at 0.
\begin{itemize}
    \item Any translation $x+b$ preserves the relative differences $\{4, 4, 4\}$.
    \item Any multiplication by a unit $a \in \{1, 5, 7, 11\}$ preserves the set $\{0, 4, 8\}$ because:
    \begin{itemize}
        \item $1 \times \{0,4,8\} = \{0,4,8\}$
        \item $5 \times \{0,4,8\} = \{0, 20, 40\} \equiv \{0, 8, 4\}$
        \item $7 \times \{0,4,8\} = \{0, 28, 56\} \equiv \{0, 4, 8\}$
        \item $11 \times \{0,4,8\} = \{0, 44, 88\} \equiv \{0, 8, 4\}$
    \end{itemize}
\end{itemize}
Since the set is invariant (up to permutation) under the entire affine group, it is the unique normalized representative of its class. A musical system requires distinct $t$ and $s$ (or at least the structural capacity for $t$ and $s$ to define a Fifth $q$). The set $\{0, 4, 8\}$ implies $t=4$ and $n-s=8 \implies s=4=t$, which contradicts $t\neq s$ for a valid musical system. Thus, Class 9 is mathematically possible but musically empty.
\end{proof}

This confirms that our eight studied systems exhaust all possible harmonic geometries in 12-TET.

\subsection{Identification with Daublebsky von Sterneck Configurations}

Among the class of cyclic cubic bipartite graphs---specifically the Levi graphs $L_{12}(S)$---there exists a distinguished subset corresponding to geometric configurations. We recall that a symmetric configuration $12_3$ is an incidence structure consisting of 12 points and 12 lines, where every point lies on exactly 3 lines and every line contains exactly 3 points. A necessary requirement for such a structure is that its associated Levi graph must have a girth $g \ge 6$. This condition ensures that the incidence structure is triangle-free in the graph-theoretic sense (no 4-cycles), meaning that any two distinct points share at most one line.

Our three connected musical systems---Standard Thirds $(4,3)$, Far Narrow Thirds $(10,9)$, and Far Wide Thirds $(11,8)$---satisfy this requirement and thus constitute valid $12_3$ configurations. To characterize them, we refer to the historical and modern census of these structures.

The classification of all connected symmetric configurations $12_3$ by Daublebsky von Sterneck \cite{DvS1,DvS2}, redone and corrected recently by Al-Azemi and Betten \cite{AlAzemi2014}, lists 229 distinct isomorphism classes, denoted $D_1 \dots D_{229}$. According to Adam's Conjecture (verified for $n \le 22$) \cite{adam1967research}, cyclic configurations are isomorphic if and only if their generating sets are affinely equivalent. Consequently, our three systems correspond to three distinct classes of cyclic $12_3$ configurations. 

The identification is performed using a two-step algorithmic sieve: first, filtering by the Automorphism Group Order (Ago), which differs from the order of the authomorphism group $\mathrm{Aut}(G)$ of its Levi graph by a factor of two: $2\mathrm{Ago}=|\mathrm{Aut}(G)|$; and second, filtering for the \textit{cyclic} property, as our musical systems are generated by the rotation of an interval set modulo 12.

\begin{enumerate}
    \item \textbf{System (11,8) $\cong$ Configuration $D_{226}$} \\
    \textbf{Generator:} $\{0, 1, 5\}$ (Offsets). \\
    \textbf{Symmetry:} Ago 24 (Levi Graph Order 48). \\
    \textbf{Identification:} This corresponds to configuration $D_{226}$. In the Al-Azemi census, $D_{226}$ is the unique cyclic configuration with Ago 24. It represents the maximal symmetry among our systems, possessing the additional multiplicative automorphism $x \mapsto 5x$.

    \item \textbf{System (4,3) $\cong$ Configuration $D_{222}$} \\
    \textbf{Generator:} $\{0, 1, 4\}$ (Offsets). \\
    \textbf{Symmetry:} Ago 12 (Levi Graph Order 24). \\
    \textbf{Identification:} The standard diatonic system corresponds to $D_{222}$. Among the cyclic configurations with Ago 12, $D_{222}$ is the class associated with the starter set $\{0, 1, 4\}$. This confirms the identification previously noted by Boland and Hughston \cite{lane}.

    \item \textbf{System (10,9) $\cong$ Configuration $D_{228}$} \\
    \textbf{Generator:} $\{0, 1, 3\}$ (Offsets). \\
    \textbf{Symmetry:} Ago 12 (Levi Graph Order 24). \\
    \textbf{Identification:} This system corresponds to configuration $D_{228}$. Although $D_{202}$ also shares the same group order (Ago 12), $D_{202}$ is not cyclic. The only other cyclic configuration with Ago 12 in the classification table in \cite{AlAzemi2014} is $D_{228}$, which is generated by the starter $\{0, 1, 3\}$. Thus, the Far Narrow system is identified as $D_{228}$.
\end{enumerate}

\begin{theorem}[Isomorphism to Daublebsky von Sterneck Configurations]
The connected musical systems, ``Standard Thirds'' $(4,3)$, ``Far Wide Thirds'' $(11,8)$ and ``Far Narrow Thirds'' $(10,9)$,  of 12-TET are isomorphic to the distinct cyclic $12_3$ configurations defined in the Al-Azemi \& Betten census as follows:

\begin{center}
\begin{tabular}{|l|c|c|c|c|}
\hline
\emph{\textbf{Musical System}} & \emph{\textbf{Generating Set}} & \emph{\textbf{Symmetry (Levi)}} & \emph{\textbf{Ago}} & $\mathbf{D_i}$ \emph{\textbf{Class}} \\
\hline
\emph{Standard (4,3)} & $\{0, 1, 4\}$ & \emph{24} & \emph{12} & $\mathbf{D_{222}}$ \\
\emph{Far Wide (11,8)} & $\{0, 1, 5\}$ & \emph{48} & \emph{24} & $\mathbf{D_{226}}$ \\
\emph{Far Narrow (10,9)} & $\{0, 1, 3\}$ & \emph{24} & \emph{12} & $\mathbf{D_{228}}$ \\
\hline
\end{tabular}
\end{center}
\vspace{0.2cm}
\end{theorem}



\section{General Classification of 12-TET Pythagorean Systems}
\label{sec:classification_theorem}

We are now in a position to generalize our findings into a complete classification of all possible harmonic systems within the 12-tone equal temperament.

Our analysis has shown that every musical system is uniquely determined by the choice of the generator interval (the "Fifth") $q$ and the partitioning of this interval into a Major Third $t$ and a Minor Third $s$. This partitioning satisfies the condition $t+s \equiv q \pmod{12}$. Each such partition defines a specific harmonic flavor characterized by the difference $\Delta = t-s$.

From a topological perspective, each system $(t,s,q)$ generates a Tonnetz defined by the offset set $\{0, t, 12-s\}$. As established in Section 10, there are exactly eight non-isomorphic affine classes of cyclic $12_3$ configurations. We can therefore systematically scan all possible generators $q \in \{1, \dots, 11\}$ and all possible integer partitions $t+s \equiv q$ ($b \neq c$) to assign every conceivable 12-TET triad to its fundamental geometric class.

\begin{theorem}[The Classification of 12-TET Harmonies]
Every tuning system in 12-TET defined by a generator $q$ and thirds $(t,s)$ with $t \neq s$ is isomorphic to one of the eight canonical Levi graph structures:
\begin{enumerate}
    \item \textbf{Wide Thirds} $L_{12}(0,1,2)$ (Signature $\{1,1,10\}$)
    \item \textbf{Far Narrow} $L_{12}(0,1,3)$ (Signature $\{1,2,9\} \cong \{2,3,7\}$)
    \item \textbf{Standard} $L_{12}(0,1,4)$ (Signature $\{1,3,8\} \cong \{3,4,5\}$)
    \item \textbf{Far Wide} $L_{12}(0,1,5)$ (Signature $\{1,4,7\}$)
    \item \textbf{Tritone} $L_{12}(0,1,6)$ (Signature $\{1,5,6\}$)
    \item \textbf{Tritone-Wide} $L_{12}(0,2,4)$ (Signature $\{2,2,8\}$)
    \item \textbf{Tritonian Fifth} $L_{12}(0,2,6)$ (Signature $\{2,4,6\}$)
    \item \textbf{Equally Spaced} $L_{12}(0,3,6)$ (Signature $\{3,3,6\}$)
\end{enumerate}
\end{theorem}

Table \ref{tab:full_classification_12} presents the complete atlas of these 54 systems. Each cell contains the tuple $(t, s; \Delta)$ representing the musical system's thirds and their difference. Note that while most systems satisfy $t > s$, four specific symmetric geometries (appearing in rows $q=5, 9, 10, 11$) require $t < s$ to be generated; these are included to complete the set of 54 valid offset configurations.

\begin{table}[h!]
\centering
\scriptsize
\renewcommand{\arraystretch}{1.5}
\setlength{\tabcolsep}{2pt}
\begin{tabular}{|c||c|c|c|c|c|c|c|c|}
\hline
\textbf{Fifth} & \textbf{Wide} & \textbf{Far Narrow} & \textbf{Standard} & \textbf{Far Wide} & \textbf{Tritone} & \textbf{Tr.-Wide} & \textbf{Tr.-Fifth} & \textbf{Equal} \\
$q$ & $L_{12}\{0,1,2\}$ & $L_{12}\{0,1,3\}$ & $L_{12}\{0,1,4\}$ & $L_{12}\{0,1,5\}$ & $L_{12}\{0,1,6\}$ & $L_{12}\{0,2,4\}$ & $L_{12}\{0,2,6\}$ & $L_{12}\{0,3,6\}$ \\ \hline\hline
1 & (11,2; 9) & (10,3; 7) & (9,4; 5) & (8,5; 3) & (7,6; 1) & -- & -- & -- \\ \hline
2 & -- & (11,3; 8) & -- & -- & -- & (10,4; 6) & (8,6; 2) & -- \\
  &    & (9,5; 4)  &    &    &    &            &           &    \\ \hline
3 & -- & (10,5; 5) & (11,4; 7) & -- & -- & -- & -- & (9,6; 3) \\
  &    & (2,1; 1)   & (8,7; 1)  &    &    &    &    &          \\ \hline
4 & -- & -- & (9,7; 2) & (11,5; 6) & -- & -- & (10,6; 4) & -- \\
  &    &    & (3,1; 2) &           &    &    &           &    \\ \hline
5 & (10,7; 3) & (3,2; 1) & (9,8; 1) & (4,1; 3) & (11,6; 5) & -- & -- & -- \\ \hline
6 & -- & -- & -- & -- & (5,1; 4)  & --  & {\bf (4,2; 2)} & -- \\
  &    &    &    &    & (11,7; 4) &           & (10,8; 2) &    \\ \hline
7 & {\bf (5,2; 3)} & {\bf (10,9; 1)} & {\bf (4,3; 1)} & {\bf (11,8; 3)} & {\bf (6,1; 5)} & -- & -- & -- \\ \hline
8 & -- & -- & (11,9; 2) & (7,1; 6) & -- & -- & (6,2; 4) & -- \\
  &    &           & (5,3; 2)  &           &    &    &          &    \\ \hline
9 & -- & (7,2; 5) & (8,1; 7) & -- & -- & -- & -- & {\bf (6,3; 3)} \\
  &    & (11,10; 1)  & (5,4; 1) &    &    &    &    &          \\ \hline
10 & -- & (9,1; 8)  & -- & -- & -- & {\bf (8,2; 6)} & (6,4; 2) & -- \\
   &    & (7,3; 4) &          &    &    &          &          &    \\ \hline
11 & (10,1; 9) & (9,2; 7) & (8,3; 5) & (7,4; 3) & (6,5; 1) & -- & -- & -- \\ \hline
\end{tabular}
\caption{Classification of 12-TET Systems based on offset pair analysis. Values are $(t, s; \Delta)$ where $t+s=q$ and $t > s$. Bold entries indicate the specific systems analyzed in Sections 4--9.}\label{tab:full_classification_12}
\end{table}

This classification reveals the scarcity of "structurally unique" harmonic worlds. Despite the combinatorial possibilities of choosing thirds and fifths in 12-TET, any composer is ultimately navigating one of these eight topological spaces.

\subsection{Mathematics versus Music}

We conclude this study by drawing a necessary line between the objective reality of the graph and the subjective reality of the ear. From a strictly mathematical point of view, our classification theorem implies a collapse of complexity: there are not dozens of musical systems in 12-TET, but only the eight topological classes analyzed in detail in Sections 4--9.

A closer inspection of Table \ref{tab:full_classification_12} reveals a striking topological dichotomy. The distribution of musical systems among these eight classes is far from uniform. Two specific topologies—the "Standard" class $L_{12}\{0,1,4\}$ and the "Far Narrow" class $L_{12}\{0,1,3\}$—are overwhelmingly dominant. Each of these two classes contains exactly \textbf{twelve} distinct musical representatives. Together, they account for nearly half (24 out of 50) of all possible harmonic systems in 12-TET. This suggests that the 12-tone universe is naturally polarized between two fundamental geometries: one favoring tertian structures (represented by Western harmony) and another favoring chromatic or cluster-like structures (represented by the Far Narrow systems).

Let us focus specifically on the topological class $L_{12}\{0,1,4\}$. As noted, this single geometric structure hosts 12 distinct musical systems distributed across various generators $q$. Only one of them—the partition $(t,s)=(4,3)$ of the classical fifth $q=7$—corresponds to the familiar harmony of the Western world. The other 11 systems are its mathematical clones.

To a mathematician relying solely on the topological tools used so far, these 12 systems are indistinguishable. Whether one analyzes the "Standard" system $(4,3)$ or the "Alien" system $(8,1)$ (appearing in the $q=9$ row), they are the same object. They share the exact same set of invariants, the same graph spectrum, and the same Levi graph structure. Meaning, there exists an affine isomorphism that maps the chords and Tonnetz of one perfectly onto the other. The structural "skeleton"—the cycles, the connectivity, the "wormholes"—is identical.

However, isomorphism is not identity in the realm of perception. While the topologist sees equivalence, the musician hears a chasm of difference. The standard system $(4,3)$ approximates the acoustic consonances derived from the harmonic series. In contrast, its isomorphic twin $(8,1)$ utilizes a Major third of 8 steps (a minor sixth) and a Minor third of 1 step (a semitone). Consequently, these "shadow" systems represent parallel harmonic universes where the grammar is identical to that of Bach or Mozart, but the vocabulary is completely alien. A chord progression that sounds "natural" and smooth in the standard system would sound jarring in the $(8,1)$ system, despite preserving the exact same voice-leading logic.

This creates a new frontier for composition. The explicit isomorphisms $\phi(x) = ax+b$ derived in Section 10 are not merely abstract proofs; they are translation tools. It is now possible to compose a piece in the "standard" system $(4,3)$, and then strictly transform it via isomorphism into a mathematically equivalent but sonically distinct system like $(9,8)$ or $(8,3)$. The harmony would remain syntactically correct—every resolution and common-tone connection would be preserved—but the semantic "color" would shift radically. The graph provides the map; it is the choice of the musician to decide which of these landscapes is worth inhabiting.

\subsection{Future Outlook: The Search for Canonical Harmony}

Is the story truly over? Does mathematics have no preference among these twelve topological twins? 

If we restrict ourselves to graph theory, the answer is yes. But if we invoke number theory, a hidden hierarchy emerges. Among the 12 representatives of the standard class, only four possess a generator $q$ that satisfies the condition $\gcd(q, 12)=1$. These are:
\begin{itemize}
    \item The \textbf{Standard} system $(4,3; 1)$ with $q=7$;
    \item The system $(9,4; 5)$ with $q=1$ (the semitone generator);
    \item The system $(9,8; 1)$ with $q=5$ (the fourth generator);
    \item The system $(8,3; 5)$ with $q=11$ (the major seventh generator).
\end{itemize}
These four systems form a "privileged quartet." Unlike the other eight shadows—which utilize generators like $q=2$ or $q=4$ that fracture the chromatic scale into disjoint loops—these four systems preserve the global connectivity of the "Circle of Fifths." 

We can go even deeper. The number 12 possesses a composite structure that allows for a decomposition via the Chinese Remainder Theorem. Since $12 = 3 \times 4$ and the factors 3 and 4 are coprime, the cyclic group $\mathbb{Z}_{12}$ is isomorphic to the direct product $\mathbb{Z}_3 \times \mathbb{Z}_4$. 

This isomorphism is not merely an abstract equivalence; it provides a specific recipe for mapping every musical note $x \in \{0, \dots, 11\}$ to a unique coordinate pair $(u, v)$ on a grid. The rule is simple: the first coordinate is the remainder of $x$ modulo 3, and the second is the remainder of $x$ modulo 4:
$$
x \longmapsto (x \bmod 3, \;\; x \bmod 4).
$$
Let us calculate the position of a few key intervals to see how this geometry constructs itself.
\begin{itemize}
    \item The root note 0 (C) maps to $(0 \bmod 3, 0 \bmod 4) = (0,0)$. This is the origin.
    \item The semitone {\color{mygold}1} ({\color{mygold}C$\sharp$}) maps to ${\color{mygold}(1,1)}$, forming the diagonal generator of the grid.
    \item The Major Third {\color{blue}4} ({\color{blue}E}) maps to $(4 \bmod 3, 4 \bmod 4) = {\color{blue}(1, 0)}$. This vector moves purely along the horizontal $\mathbb{Z}_3$ axis.
    \item The Major Sixth {\color{red}9} ({\color{red}A}) maps to $(9 \bmod 3, 9 \bmod 4) = {\color{red}(0, 1)}$. This vector moves purely along the vertical $\mathbb{Z}_4$ axis.
\end{itemize}
Mapping the entire chromatic scale onto this product space reveals the hidden geometric grid visualized in Figure \ref{fig:crt_grid}.

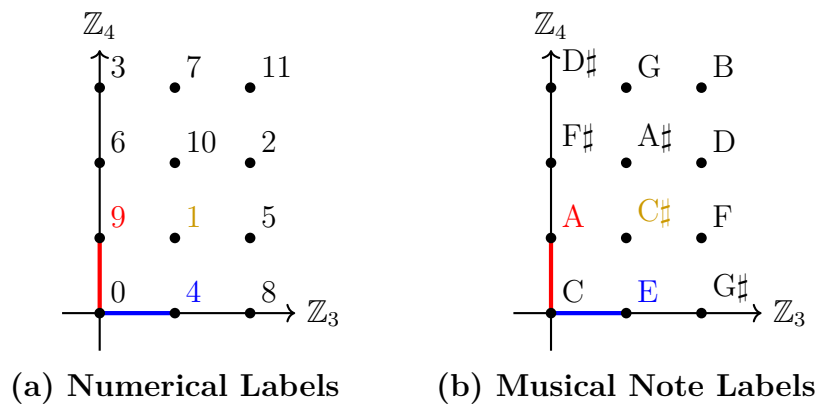
\begin{figure}[H]
    \centering
    \definecolor{mygold}{rgb}{0.8, 0.6, 0.0}
    \begin{tikzpicture}[scale=1.0]
        \begin{scope}
            \draw[->, thick] (-0.5,0) -- (2.6,0) node[right] {$\mathbb{Z}_3$};
            \draw[->, thick] (0,-0.5) -- (0,3.5) node[above] {$\mathbb{Z}_4$};
            
            \draw[-, ultra thick, blue] (0,0) -- (1,0);
            \draw[-, ultra thick, red] (0,0) -- (0,1);

            \fill (0,0) circle (2pt) node[above right] {0};
            \fill (1,0) circle (2pt) node[above right] {{\color{blue}4}};
            \fill (2,0) circle (2pt) node[above right] {8};
            
            \fill (0,1) circle (2pt) node[above right] {{\color{red}9}};
            \fill (1,1) circle (2pt) node[above right] {{\color{mygold}1}};
            \fill (2,1) circle (2pt) node[above right] {5};
            
            \fill (0,2) circle (2pt) node[above right] {6};
            \fill (1,2) circle (2pt) node[above right] {10};
            \fill (2,2) circle (2pt) node[above right] {2};
            
            \fill (0,3) circle (2pt) node[above right] {3};
            \fill (1,3) circle (2pt) node[above right] {7};
            \fill (2,3) circle (2pt) node[above right] {11};
            
            \node at (1, -1) {\textbf{(a) Numerical Labels}};
        \end{scope}

        \begin{scope}[xshift=6cm]
            \draw[->, thick] (-0.5,0) -- (2.8,0) node[right] {$\mathbb{Z}_3$};
            \draw[->, thick] (0,-0.5) -- (0,3.5) node[above] {$\mathbb{Z}_4$};
            
            \draw[-, ultra thick, blue] (0,0) -- (1,0);
            \draw[-, ultra thick, red] (0,0) -- (0,1);
            
            \fill (0,0) circle (2pt) node[above right] {C};
            \fill (1,0) circle (2pt) node[above right] {{\color{blue}E}};
            \fill (2,0) circle (2pt) node[above right] {G$\sharp$};
            
            \fill (0,1) circle (2pt) node[above right] {{\color{red}A}};
            \fill (1,1) circle (2pt) node[above right] {{\color{mygold}C$\sharp$}};
            \fill (2,1) circle (2pt) node[above right] {F};
            
            \fill (0,2) circle (2pt) node[above right] {F$\sharp$};
            \fill (1,2) circle (2pt) node[above right] {A$\sharp$};
            \fill (2,2) circle (2pt) node[above right] {D};
            
            \fill (0,3) circle (2pt) node[above right] {D$\sharp$};
            \fill (1,3) circle (2pt) node[above right] {G};
            \fill (2,3) circle (2pt) node[above right] {B};
            
            \node at (1, -1) {\textbf{(b) Musical Note Labels}};
        \end{scope}
    \end{tikzpicture}
    \caption{The decomposition of the 12-tone universe onto the $\mathbb{Z}_3 \times \mathbb{Z}_4$ grid via the Chinese Remainder Theorem. The segments corresponding to the basis vectors are highlighted: the blue segment represents the interval {\color{blue}4} (Major Third), and the red segment represents the interval {\color{red}9} (Major Sixth). The ``Unison'' is at {\color{mygold}(1,1)}.}
    \label{fig:crt_grid}
\end{figure}

The points corresponding to the integers $4$ and $9$ play a distinguished role in this geometry. In the product coordinates, $4$ corresponds to the basis vector $(1,0)$ and $9$ corresponds to the basis vector $(0,1)$.

These two intervals are the \textbf{minimal generators} of the grid along the $\mathbb{Z}_3$ and $\mathbb{Z}_4$ axes. This observation forces us to revisit our "privileged quartet" with fresh eyes. We recall that one of its members is the system defined by the partition $(t,s) = (9,4)$ with the generator $q=1$. 

In the light of the Chinese Remainder Theorem, this specific system appears to be the most mathematically natural of all. Its "Major Third" is the basis vector $(0,1)$ (the note A relative to C), its "Minor Third" is the basis vector $(1,0)$ (the note E), and its "Fifth" (the generator $q=1$) is the diagonal unit $(1,1)$ (the note C$\sharp$).

Compare this to the standard Western harmony $(4,3)$. While $(4,3)$ is acoustically privileged by the physics of vibrating strings, structurally it is more complex: the interval $4$ is the basis vector $(1,0)$, but the interval $3$ is $(0,3)$—a vector three times longer than the unit step in the $\mathbb{Z}_4$ direction.

This suggests that the "alien" system $(9,4)$—where the "Major Third" is a major sixth and the "Minor Third" is a major third—is not merely a distortion of our familiar harmony. It is, perhaps, the canonical form of harmony in the algebraic universe of $\mathbb{Z}_{12}$, from which our own Western tradition is but a specific, acoustically biased derivation. We leave the exploration of this "Canonical CRT Harmony" for future research.

\subsection{Examples of the $(4,3)$ to $(9,4)$ Transform}

Having established the topological classification of the harmonic universes in 12-TET, it is natural to ask how these non-standard mathematical structures manifest acoustically. To bridge the gap between algebraic theory and musical experience, we explored the auditory geometry of the ``alien'' $(9,4)$ system by translating existing classical compositions from the standard Western $(4,3)$ system.

A remarkable property of the passage from the standard $(4,3)$ system to the $(9,4)$ system is that the entire translation can be governed by a single affine transformation on the pitch classes $x \in \mathbb{Z}_{12}$:
\begin{equation}
    x \mapsto 5x + 1 \pmod{12}.
\end{equation}
Because $\gcd(5,12)=1$, this transformation is a bijection, ensuring that no melodic or harmonic information is destroyed (no two distinct pitch classes are mapped to the same value). Most strikingly, this map acts universally across the harmonic hierarchy. In the transition from $(4,3)$ to $(9,4)$, the offsets of the individual notes, the roots and internal intervals of the minor and major chords, and even the coordinates of the underlying Tonnetz are all consistently transformed by this exact same mapping $x \mapsto 5x + 1 \pmod{12}$. 

To demonstrate the musical viability of this mathematical symmetry, we developed a computational script to apply this affine transformation to the MIDI data of classical repertoire. The algorithm processes each note event, applying the $(5x+1)\bmod 12$ map to its pitch class while preserving its original octave register and temporal rhythm. The transformed MIDI sequences were subsequently rendered into high-fidelity audio using the physical modeling synthesizer Pianoteq 9 PRO, utilizing its advanced acoustic piano models to provide a neutral, deeply resonant acoustic space for the new harmony.

The results of this transformation produce a profoundly different harmonic landscape that nonetheless retains the rigid logical voice-leading of the original compositions. We invite the reader to listen to the following rendered examples of this $(4,3) \to (9,4)$ transformation:

\begin{itemize}
    \item \textbf{George Frideric Händel -- Passacaglia in G minor (HWV 432):}\\
    \href{https://www.fuw.edu.pl/~nurowski/passacaglia_transformed.wav}{Listen to the $(9,4)$ transformed version (.wav)}
    
    \item \textbf{Franz Schubert -- Piano Trio No. 2 in E-flat major, Op. 100 (Parts 1 \& 2):}\\
    \href{https://www.fuw.edu.pl/~nurowski/schubert_trio_transformed.wav}{Listen to the $(9,4)$ transformed version (.wav)}
    
    \item \textbf{Franz Schubert -- Piano Trio No. 2 in E-flat major, Op. 100 (Parts 3 \& 4):}\\
    \href{https://www.fuw.edu.pl/~nurowski/schubert_trio_transformed_parts3_4.wav}{Listen to the $(9,4)$ transformed version (.wav)}
\end{itemize}

These audio examples provide a direct sensory verification of our algebraic classification. They demonstrate that the mathematical elegance of the $(9,4)$ system---predicted theoretically via the Chinese Remainder Theorem and Levi graphs---translates into a coherent, albeit completely alien, auditory universe.


\begin{thebibliography}{99}
\bibitem{adam1967research}
A.~\'{A}d\'{a}m, Research Problem 2-10, \textit{J. Combin. Theory} \textbf{2}, 393, 1967; \href{https://www.sciencedirect.com/science/article/pii/S0021980067800371/pdfft?md5=38955351dbc7b9b22b94b2a6120ef4af&pid=1-s2.0-S0021980067800371-main.pdf}{online}

  
  \bibitem{AlAzemi2014}
A.~AlAzemi and D.~Betten,
``The configurations $12_3$ revisited,''
\textit{Journal of Geometry}, vol.~105, no.~2, pp.~391--417, 2014, \href{https://link.springer.com/article/10.1007/s00022-014-0228-0}{online}

  
    \bibitem{lane}
      J. R. Boland, L. P. Hughston (2025) ``Configurations, Tessellations and Tone Networks''  \href{https://arxiv.org/abs/2505.08752}{arXiv:2505.08752}

\bibitem{DvS1} Daublebsky von Sterneck, R. ``Die Configurationen $12_3$'' . Monatshefte Math. Phys. 5, 223–255 (1895)
\bibitem{DvS2} Daublebsky von Sterneck, R.: ``\"Uber die zu den Configurationen $12_3$ zugeh\"origen Gruppen von Substitutionen''. Monatshefte Math. Phys. 14, 254–260 (1903)
      
      
\bibitem{Euler1739}
L. Euler,
\textit{Tentamen novae theoriae musicae},
Saint Petersburg Academy, 1739; \href{https://scholarlycommons.pacific.edu/cgi/viewcontent.cgi?article=1032&context=euler-works}{online}

``On the Möbius ladders,'' 


\bibitem{muzychuk1995} Muzychuk, M. E. (1995). Ádám's conjecture is true in the square-free case. \textit{Journal of Combinatorial Theory, Series A}, 72(1), 118-134.


\bibitem{Naumann1858}
C.~E.~Naumann,
\emph{\"Uber die verschiedenen Bestimmungen der Tonverh\"altnisse
und die Bedeutung des pythagor\"aischen oder reinen Quintensystems
f\"ur unsere heutige Musik},
Breitkopf und H\"artel, Leipzig, 1858; \href{https://books.google.pl/books?id=Pu-wwE1Ug-cC&printsec=frontcover&hl=pl&source=gbs_ge_summary_r&cad=0#v=onepage&q&f=false}{online}

\bibitem{Nurowski10TET}
P. Nurowski,
`` Holographic Harmony: A Topological Atlas of the Decaphonic Musical Universe''
\href{https://arxiv.org/pdf/2601.02271}{arXiv: 2601.02271} , 2026.

\bibitem{Oettingen1866}
A. von Oettingen,
\textit{Harmoniesystem in dualer Entwickelung},
Dorpat and Leipzig: W. Gl\"aser, 1866; \href{https://archive.org/details/harmoniesystemin00oett/page/n11/mode/2up}{online}

\bibitem{Rameau1722}
J.-P. Rameau,
\textit{Trait{\'e} de l'harmonie r{\'e}duite {\`a} ses principes naturels},
Jean-Baptiste-Christophe Ballard, Paris, 1722, \href{https://archive.org/details/traitdelharmon00rame/page/n5/mode/2up}{online}

\bibitem{Riemann1893}
H.~Riemann,
\emph{Vereinfachte Harmonielehre oder die Lehre von den tonalen Funktionen der Akkorde},
Breitkopf \& H\"artel, Leipzig, 1893; \href{https://archive.org/details/McGillLibrary-002565430-19384/page/n1/mode/2up}{online}


\bibitem{Riemann1914}
H. Riemann,
``Ideen zu einer `Lehre von den Tonvorstellungen',''
\textit{Jahrbuch der Musikbibliothek Peters}, vol. 21/22, pp. 1--26, 1914/15; \href{https://www.scribd.com/document/637371695/H-Riemann-Ideen-zu-einer-Lehre-von-den-Tonvorstellungen}{online}

\end{thebibliography}
\end{document}